\documentclass[11pt]{article}

\usepackage[T1]{fontenc}
\usepackage[utf8]{inputenc}
\usepackage{lmodern}
\usepackage{microtype}
\usepackage{amsmath,amssymb,amsthm,mathtools,mathrsfs}
\usepackage{aliascnt}
\usepackage{tikz-cd}
\usetikzlibrary{arrows.meta}
\usepackage{graphicx}
\usepackage{float}
\usepackage{enumitem}
\usepackage{booktabs,array,longtable}
\usepackage[nottoc,notlot,notlof]{tocbibind}
\usepackage{sectsty}
\usepackage[a4paper,total={6in,8in}]{geometry}
\usepackage{hyperref}

\sectionfont{\centering}
\hypersetup{
  colorlinks=true,
  linkcolor=blue!45!black,
  citecolor=blue!45!black,
  urlcolor=blue!45!black,
  pdftitle={Negative contacts in genus one: A comparison of punctured and root-stack Gromov-Witten theories},
  pdfauthor={Yu Wang}
}

\numberwithin{equation}{section}
\newtheorem{theorem}{Theorem}[section]
\newaliascnt{proposition}{theorem}
\newtheorem{proposition}[proposition]{Proposition}
\aliascntresetthe{proposition}
\newaliascnt{lemma}{theorem}
\newtheorem{lemma}[lemma]{Lemma}
\aliascntresetthe{lemma}
\newaliascnt{corollary}{theorem}
\newtheorem{corollary}[corollary]{Corollary}
\aliascntresetthe{corollary}
\theoremstyle{definition}
\newaliascnt{definition}{theorem}
\newtheorem{definition}[definition]{Definition}
\aliascntresetthe{definition}
\theoremstyle{definition}
\newtheorem{example}[theorem]{Example}
\theoremstyle{remark}
\newtheorem{remark}[theorem]{Remark}

\usepackage[nameinlink,noabbrev]{cleveref}

\newcommand{\Acal}{\mathcal A}
\newcommand{\Dcal}{\mathcal D}
\newcommand{\Mbar}{\overline{\mathcal M}}
\newcommand{\Mfrak}{\mathfrak M}

\newcommand{\Punct}{\operatorname{Punct}}
\newcommand{\PunctOrb}{\operatorname{PunctOrb}}
\newcommand{\DF}{\operatorname{DF}}
\newcommand{\Spec}{\operatorname{Spec}}
\newcommand{\CT}{\operatorname{CT}}
\newcommand{\vir}{\mathrm{vir}}
\newcommand{\refc}{\mathrm{ref}}
\newcommand{\main}{\mathrm{main}}
\newcommand{\red}{\mathrm{red}}

\newcommand{\bbA}{\mathbf A}
\newcommand{\bbC}{\mathbf C}
\newcommand{\bbG}{\mathbf G}
\newcommand{\bbQ}{\mathbf Q}
\newcommand{\bbZ}{\mathbf Z}
\newcommand{\cO}{\mathcal O}

\newcommand{\cC}{\mathcal C}
\newcommand{\cL}{\mathcal L}

\newcommand{\proofstep}[1]{\par\smallskip\noindent\emph{#1.}\enspace}

\title{\textbf{NEGATIVE CONTACTS IN GENUS ONE: A COMPARISON OF PUNCTURED AND ROOT-STACK GROMOV--WITTEN THEORIES}}
\author{YU WANG}
\date{}

\begin{document}

\maketitle

\begin{abstract}
Let $D$ be a smooth divisor in a smooth projective complex variety $X$.
For connected curves of arithmetic genus one with prescribed signed
contact orders, we prove that the refined punctured cycle of
Battistella--Nabijou--Ranganathan and the negative-contact cycle of
Fan--Wu--You agree after pushforward to the common moduli space of
stable maps with divisor evaluations.  Thus the pushed-forward refined
punctured cycle is the constant coefficient of the pushed-forward
root-stack virtual class, normalized by one power of the root order
for each negative contact.
The key step is a comparison for the universal target. 
After restricting to finite-type open substacks determined by the fixed pair $(X,D)$ and numerical data $\Gamma$, we prove that the positive BNR space maps
finitely and with generic degree one onto Crumplin's main component.  Using
Crumplin's genus-one component description and degree formulas, we
identify this component's fundamental cycle with the constant
coefficient of the universal orbifold virtual class under comparison
of root orders.  Refined zero-section pullback recovers the negative
contacts, and compatible virtual pullbacks and root-forgetting
pushforwards transfer the resulting identity to $(X,D)$.
\end{abstract}

\tableofcontents

\section{Introduction}
\label{sec:introduction}

\subsection{Negative contacts and the comparison problem}

Punctured Gromov--Witten invariants provide the structure constants for
theta-function multiplication in the intrinsic mirror ring of a log
Calabi--Yau pair \cite{GS18,GS19}, and are related to scattering diagrams
through the canonical wall structure \cite{GS22}.  The output marking in
these counts may have negative contact order.  Negative contacts also
enter the proper Landau--Ginzburg potential and, under suitable positivity
assumptions, determine the inverse relative mirror map \cite{You24}.

Relative Gromov--Witten theory for smooth pairs was constructed using
expanded targets by J.~Li \cite{Li01,Li02}; logarithmic compactifications
were developed by Gross--Siebert, Chen, and Abramovich--Chen
\cite{GrossSiebertLGW,Chen,AC14}.  Their comparison and behavior under
modifications and degeneration are established in
\cite{AMW,AW18,ACGS20,Ran22}.  These theories prescribe nonnegative
contacts at markings, whereas cutting a logarithmic map at a node produces
opposite contact vectors on its branches.  Abramovich--Chen--Gross--Siebert
(ACGS) accommodate negative contacts through punctured logarithmic maps,
with moduli, obstruction, and gluing theories \cite[Sections~2--5]{ACGS}.
The behavior of invariants associated with tropical types under
logarithmic modifications is studied in \cite{JohnstonBir}.

Root stacks give an alternative approach \cite{Cad07,CC08}, within twisted
stable-map and orbifold Gromov--Witten theory \cite{AV02,AGV08}.  For a
smooth pair and nonnegative contacts, Abramovich--Cadman--Wise prove the
genus-zero relative/root-stack comparison for sufficiently large and
divisible root order \cite[Theorem~1.1]{ACW}.  Tseng--You prove higher-genus
polynomiality with relative theory as constant coefficient, and remove
the divisibility requirement in genus zero
\cite[Theorem~1.5, Remark~1.6, and Section~4]{TY20}.  Fan--Wu--You encode
a contact $-d<0$ at root order $N>d$ by age $1-d/N$.  After multiplication
by one power of $N$ per negative marking, their pushed-forward orbifold
cycle is independent of large $N$ in genus zero
\cite[Theorems~3.2 and~6.1]{FWY20}, and eventually polynomial in higher
genus.  Its constant coefficient is their negative-contact cycle,
equivalently given by a relative/rubber graph sum
\cite[Theorems~3.1 and~3.13]{FWY}.  Extensions include modified loop
identities \cite{You21} and an all-genus multiroot theory for simple normal
crossings pairs \cite{TY23}.

Battistella--Nabijou--Ranganathan (BNR) define an all-genus
\emph{refined punctured class}: a refined intersection on the universal
Artin fan supplies a pure-dimensional base cycle for the ACGS relative
obstruction theory \cite[Section~1.4, Definition~1.13]{BNR}.
We use this class throughout.  Herr--Holmes--Spelier's pierced theory
provides splitting and loop formulae in logarithmic Chow theory and
expresses BNR's class in that framework
\cite[Section~1.4 and Proposition~5.1.4]{HHS}.
Building on their positive-contact comparison after suitable blowups
\cite{BNR24}, Battistella--Nabijou--Ranganathan prove, for sufficiently
large root order, the genus-zero negative-contact comparison through a
\emph{chimera} space
retaining both logarithmic and root data: forgetting the logarithmic
enhancement is an isomorphism onto the orbifold space, while forgetting
the roots contributes one inverse power of $N$ per puncture to the refined
pushforward \cite[Theorem~A and Proposition~3.13]{BNR}.
This connects the tropical information of punctured theory to orbifold
recursion and virtual localization \cite{AGV08,GP99}, and has applications
to the intrinsic mirror ring \cite{Johnston}.

\subsection{The genus-one problem}

In genus one, a circuit in the dual graph allows an integral circulation
in the balancing equations.  Alternatively, an elliptic component on
which the section vanishes can carry an independent degree-zero Jacobian
factor.  The orbifold space can consequently have components not seen by
the logarithmic construction.

We first study the universal pair
\[
 \Acal=[\bbA^1/\bbG_m],\qquad \Dcal=B\bbG_m\subset\Acal,
 \qquad \Acal_N=\sqrt[N]{(\Acal,\Dcal)}.
\]
A map to $\Acal$ is a line bundle with a section; the classifying map
$(X,D)\to(\Acal,\Dcal)$ is defined by $(\cO_X(D),s_D)$.
Thus the universal problem retains the boundary line--section data.
Crumplin's \emph{main component} is the closure of the smooth-source,
nonzero-section locus in his positive universal orbifold space
\cite[Section~2.1.3]{Crumplin}.  His decisive genus-one result expresses the virtual class of the
universal orbifold space as the sum of the fundamental cycles of all
irreducible components \cite[Corollaries~4.3 and~4.5]{Crumplin}.
For the bounded genus-one problem considered below, comparison of root
orders gives degree one on the main component and positive powers of the
multiplier on the others.  Our
first theorem identifies this constant-coefficient cycle using BNR's
positive logarithmic construction; the second transfers the identity to
the negative-contact cycles of $(X,D)$.  The component decomposition is
specific to genus one and does not assert unobstructedness of maps to $X$.

\subsection{Main results}

\paragraph{Numerical data.}
Let $D$ be a smooth effective Cartier divisor in a smooth projective
complex variety $X$.  Fix an effective class $\beta\in H_2(X,\bbZ)$,
$n$ interior markings $y_1,\ldots,y_n$, and $\rho$ relative markings
$x_1,\ldots,x_\rho$ with nonzero integral contacts $\mu_1,\ldots,\mu_\rho$.
Write
\[
 \Gamma=(1,n,\beta,\rho,\boldsymbol\mu),\qquad
 \sum_i\mu_i=D\cdot\beta,\qquad
 \mathcal I=\{y_1,\ldots,y_n\}\sqcup\{x_1,\ldots,x_\rho\}.
\]
The ordered label set has size $\ell=n+\rho$; put $\mu_{y_j}=0$ and
$\mu_{x_i}=\mu_i$.  A label also denotes its marking section and,
when untwisted, its Cartier divisor.  Set
\[
 P=\{x_i:\mu_i<0\},\quad m=|P|,\quad
 d_p=-\mu_p>0\ (p\in P),\quad
 \rho_+=|\{i:\mu_i>0\}|.
\]
All root orders satisfy
$N>M_{\boldsymbol\mu}:=\max(\{0\}\cup\{|\mu_i|:1\leq i\leq\rho\})$.
The root stack is $X_{D,N}=\sqrt[N]{(X,D)}$, with coarse projection $p_N$
and reduced root divisor $D_N$, so $p_N^*D=ND_N$.

BNR's \emph{positivised numerical data} $\Lambda_\Gamma^+$ retain every
label, replace each negative contact by $0$, and add $d_p$ to the coarse
degree for each $p\in P$.  Thus
\[
 c_i=\max(\mu_i,0)\ (i\in\mathcal I),\qquad
 d^+=D\cdot\beta+\sum_{p\in P}d_p=\sum_{\mu_i>0}\mu_i,\qquad
 Z=\{y_1,\ldots,y_n\}\sqcup P.
\]
The sets $P$ and $Z$ inherit their orders from $\mathcal I$.  In this
auxiliary problem, the labels in $Z$ have contact zero and trivial source
stabilizer.  Ages and source indices for both data are specified in
\Cref{sec:setup}.

\paragraph{The universal comparison.}
Let $\overline O_N^+$ be the ambient universal orbifold stack with data
$\Lambda_\Gamma^+$, all labels retained, and root line--section pair
$(L_N^+,\sigma_N^+)$.  Its open substack $O_N^+$ imposes Crumplin's
condition $|\deg(L_N^+|_E)|<\tfrac12$ for every proper subcurve $E$.
The bar denotes the ambient stack, not a closure.  The positive BNR space
$K_N^+$ retains a compatible basic logarithmic and tropical enhancement;
forgetting it gives $\omega_N^+:K_N^+\to\overline O_N^+$.  The target is
ambient because the entire source need not satisfy the degree condition.
Let $Z^\circ_{\main,N}\subset O_N^+$ be the distinguished one-vertex trivial-type
locus with all labels retained, and $Z_{\main,N}$ its closure.

Section~\ref{sec:setup} constructs a bounded finite-type open
$W_{r,Z}\subset O_r^+$ containing the entire image relevant to the fixed
pair and datum $\Gamma$, for sufficiently large and divisible $r$.
For $R=\lambda r$, the ambient comparison $\overline\pi_{R,r}^+$ takes
the $\lambda$th tensor power and coarsens the kernel of the resulting
stabilizer character.  Define
\begin{equation}
 W_{R,Z}:=\overline O_R^+
 \mathop{\times}_{\overline O_r^+}W_{r,Z},
 \qquad \pi_{R,r}^+:W_{R,Z}\longrightarrow W_{r,Z}.
\label{eq:intro-root-base-change}
\end{equation}
All fibre products of stacks are $2$-fibre products.
Lemma~\ref{lem:proper-root-comparison} proves $W_{R,Z}\subset O_R^+$.
At the fixed order $r$, set
\begin{equation}
 K_{W_r}^+:=K_r^+\mathop{\times}_{\overline O_r^+}W_{r,Z},
 \qquad \omega_{W_r}^+:K_{W_r}^+\longrightarrow W_{r,Z}.
\label{eq:intro-BNR-base-change}
\end{equation}
Thus $\pi$ changes root order and $\omega$ forgets the logarithmic
enhancement at fixed order.  Finally, put
\[
 U_{r,Z}:=Z^\circ_{\main,r}\times_{O_r^+}W_{r,Z},\qquad
 Z_{W_r}^{\main}:=(Z_{\main,r}\times_{O_r^+}W_{r,Z})_{\red}.
\]
We use rational Chow groups $A_k(-)_\bbQ$, graded by dimension, and write
$[Y]$ for the fundamental cycle of a pure-dimensional stack.
Superscripts $\vir$ and $\refc$ denote the orbifold virtual class and
BNR's refined class, respectively; $\CT_t$ extracts the constant
coefficient of a Chow-valued eventual polynomial in $t$.

\begin{theorem}[Comparison with Crumplin's main component]
\label{thm:main-component}
Assume $n+\rho>0$.  For every sufficiently large and divisible base root order $r$, the
finite-type substacks defined in \Cref{sec:setup} have the following
properties.  For every integer $\lambda\geq1$, with $R=\lambda r$,
\[
 \pi_{R,r}^+:W_{R,Z}\longrightarrow W_{r,Z}
\]
is proper, quasi-finite, and of Deligne--Mumford type.  The morphism
\[
 \omega_{W_r}^+:K_{W_r}^+\longrightarrow W_{r,Z}
\]
is finite, its source is reduced and irreducible, and it is an isomorphism
over $U_{r,Z}$.

Moreover,
\[
 Q_Z(\lambda):=(\pi_{\lambda r,r}^+)_*
 [W_{\lambda r,Z}]^{\vir}\in A_*(W_{r,Z})_\bbQ
\]
is polynomial in $\lambda$ of degree at most one, and
\begin{equation}
 (\omega_{W_r}^+)_*[K_{W_r}^+]
   =[Z_{W_r}^{\main}]
   =\CT_\lambda Q_Z(\lambda)
\label{eq:intro-main-component}
\end{equation}
in $A_*(W_{r,Z})_\bbQ$.
\end{theorem}

\begin{remark}[Relation with the genus-zero comparison of BNR]
For sufficiently large root order, BNR's genus-zero chimera isomorphism
accounts for the entire orbifold space \cite[Theorem~A; Theorem~4.2]{BNR}.  Theorem~\ref{thm:main-component}
replaces it by a finite comparison with the main component and extraction
of that component as the constant coefficient.  The latter uses
Crumplin's genus-one decomposition of the virtual class; this argument
does not extend directly to arbitrary genus.
\end{remark}

The marked-source hypothesis is automatic if $m>0$.  The case $m=0$ of
the next theorem follows independently from the logarithmic/relative
comparison and requires no additional marking.

\paragraph{The geometric cycle comparison.}
The common target is
\[
 B_\Gamma=\overline{\mathcal M}_{1,n+\rho}(X,\beta)
 \mathop{\times}_{X^\rho}D^\rho,
\]
where $\overline{\mathcal M}_{1,n+\rho}(X,\beta)$ is the Kontsevich
stack of connected stable maps and the fibre product uses evaluation at
the relative markings and $D^\rho\hookrightarrow X^\rho$.
Let $\Punct_\Gamma(X\mid D)$ be the ACGS stack of basic stable punctured
maps, and $\Mbar_\Gamma(X_{D,N})$ the twisted stable-map stack with the
prescribed ages.  Forgetting logarithmic or root data, with coarsening
and stabilization as needed, gives proper morphisms
\[
 \varrho:\Punct_\Gamma(X\mid D)\longrightarrow B_\Gamma,
 \qquad \tau_N:\Mbar_\Gamma(X_{D,N})\longrightarrow B_\Gamma.
\]
Fan--Wu--You prove eventual polynomiality of
$N^m\tau_{N*}[\Mbar_\Gamma(X_{D,N})]^{\vir}$ and identify its constant
coefficient with their bipartite graph-sum cycle $\mathfrak c_\Gamma(X/D)$
\cite[Theorems~3.1 and~3.13]{FWY}.  Here $T_X(-\log D)$ is the locally
free kernel of $T_X\to\cO_D(D)$, consisting of vector fields tangent to $D$.

\begin{theorem}[Genus-one negative-contact comparison]
\label{thm:main}
With the notation defined in \Cref{sec:setup},
\begin{equation}
 \varrho_*[\Punct_\Gamma(X\mid D)]^{\refc}
 =\mathfrak c_\Gamma(X/D)
 =\CT_N\!\left(
 N^m\tau_{N*}[\Mbar_\Gamma(X_{D,N})]^{\vir}
 \right)
\label{eq:main-comparison}
\end{equation}
in
\begin{equation}
 A_{d_\Gamma}(B_\Gamma)_\bbQ,
 \qquad
 d_\Gamma=
 \int_\beta c_1\!\left(T_X(-\log D)\right)+n+\rho_+.
\label{eq:vdim}
\end{equation}
Here $\CT_N$ denotes the constant coefficient of the eventual polynomial in
$N$.
\end{theorem}

\begin{remark}[Meaning of the comparison]
The equality is between cycles pushed forward to $B_\Gamma$; it makes no
identification of the two geometric moduli spaces.  The proof uses
Theorem~\ref{thm:main-component} to extend the genus-zero BNR comparison,
with the existing BNR and Fan--Wu--You normalizations.
\end{remark}

\subsection{Proof strategy}

We first prove \Cref{thm:main-component}, then deduce \Cref{thm:main}.
Fix $r$ and vary $R=\lambda r$: root-order comparison puts the orbifold
cycles in one Chow group, while the BNR comparison at $r$ identifies
their constant coefficient with a logarithmic cycle.

\begin{enumerate}[label=\textbf{Step \arabic*.},leftmargin=*]
\item \textbf{Express the negative orbifold space as an evaluation zero locus.}
Section~\ref{sec:negative-in-positive} uses twisting and coarsening at the
punctures to identify negative root data with positive data whose
evaluations vanish at $P$, scheme-theoretically and with compatible
obstruction theories.  Under $R=\lambda r$, the equations become
$\lambda$th powers; the corresponding refined pullback contributes
$\lambda^{-m}$.

\item \textbf{Bound the universal geometry and construct proper root-order
comparisons.}
The open $W_0$ defined in Section~\ref{sec:setup} contains the entire
geometric image.  Uniform bounds on
its tropical slopes, including special-fibre types, permit one root
order $r$ to work throughout.  Section~\ref{sec:bounded-root-geometry}
proves properness of the faithful-root problem and hence of the
comparisons $\pi_{\lambda r,r}^+$ on the full base-change opens.  Their
pushforwards allow constant-coefficient extraction in $A_*(W_{r,Z})_\bbQ$.

\item \textbf{Identify the main-component cycle as the constant coefficient.}
The genus-one essential types indexing components are the trivial type
and elliptic-centred stars.  Crumplin's virtual class is their
fundamental-cycle sum.  Verifying the hypotheses of his degree theorem
in Proposition~\ref{prop:compatible-support} gives degree one on the main
component and positive powers of $\lambda$ on all other components.

\item \textbf{Show that the positive BNR cycle pushes forward to the
main-component cycle.}
Section~\ref{sec:positive-space} proves that $K_r^+$ is reduced and
irreducible.  A finite collection of cones closed under faces contains
all relevant types, including specializations, and supports the
properness argument for logarithmic enhancements.  The rooted comparison
is finite and is generically an isomorphism onto the main component;
restoring the zero-contact labels preserves this result and proves
\Cref{thm:main-component}.

\item \textbf{Derive the universal negative-contact identity by refined
pullback.}
BNR's negative chimera is the classical zero-section locus in $K_r^+$.
Its puncturing-offset line--section pairs agree with the evaluation
pairs in Step~1.  Refined-Gysin base change, the powered-section formula,
and the main-component identity give the universal negative comparison.

\item \textbf{Pass to $(X,D)$ and obtain the Fan--Wu--You comparison.}
Section~\ref{sec:geometric} establishes the geometric fibre products and
compatible virtual pullbacks, checking the hypotheses for commuting them
with proper pushforward.  BNR's all-genus root-forgetting formula supplies
$r^{-m}$, which cancels the remaining factor $r^m$.  The substitution
$N=\lambda r$ preserves the constant coefficient, identified by
Fan--Wu--You with $\mathfrak c_\Gamma(X/D)$.
\end{enumerate}

\paragraph{Acknowledgments.}
I am sincerely grateful to Professor Mark Gross, my Ph.D. advisor at the
University of Cambridge, for initiating this project and for his guidance.
I also thank Fenglong You for his support throughout the entire project and for many helpful
conversations about root-stack virtual cycles with negative contact orders. I am also grateful to H\"ulya Arg\"uz and Pierrick Bousseau for many useful
discussions and for their support during my time at the University of Georgia.
This work was partially supported by the European Research Council
through ERC grant MSAG.

\section{Moduli spaces and conventions}\label{sec:setup}

\subsection{Ground field and numerical data}

We work over $\bbC$ with rational Chow groups.  Thus all root stacks and
twisted curves below are tame, and invariant pushforward along their finite
stabilizers is exact.  Curves are connected, and genus means arithmetic
genus; rational circuits are allowed.  Write $\Mfrak_{g,k}$ for the stack
of prestable curves of genus $g$ with $k$ ordered markings, and use the
superscript $\mathrm{tw}$ for twisted curves.  Stability is imposed in
every stable-map moduli problem.  Cartesian squares and fibre products of
stacks are $2$-Cartesian; an unlabelled arrow is the structural projection
or stated inclusion.  A commutative diagram means $2$-commutative and need
not be Cartesian.

For a logarithmic stack $Y$, write $M_Y$ for its logarithmic structure and
$\overline M_Y=M_Y/\mathcal O_Y^\times$ for its characteristic sheaf;
$(-)^{\mathrm{gp}}$ denotes groupification.  The Deligne--Faltings functor
$\DF_{M_Y}=\DF_Y$ takes characteristic sections to line bundles with
sections and addition to tensor product.  Its target $\operatorname{Div}_Y$
is the symmetric monoidal stack of line--section pairs, with isomorphisms
preserving both the line and section; $\operatorname{Div}(Y)$ is its
category of global objects.  The notation $\Gamma(Y,\mathcal F)=H^0(Y,\mathcal F)$
is distinct from the numerical datum $\Gamma$.  We use $\mathbb L_{F/G}$
for the relative cotangent complex, $\mathbf R$ for derived pushforward,
and $(-)^\vee$ for derived dual; pullbacks of complexes are derived even
when $\mathbf L$ is suppressed.  For a vector bundle $E$ on $Y$, put
$\mathbb V(E)=\underline{\operatorname{Spec}}_Y\operatorname{Sym}(E^\vee)$.

Retain the pair $(X,D)$, effective class $\beta$, ordered labels
$\mathcal I$, and signed contacts from the Introduction.  In particular,
\begin{equation}
 \Gamma=(1,n,\beta,\rho,\boldsymbol\mu),\qquad
 \sum_{i=1}^{\rho}\mu_i=D\cdot\beta,\qquad
 P=\{x_i:\mu_i<0\},\quad m=|P|,\quad \ell=n+\rho.
\label{eq:discrete-data}
\end{equation}
Here $P$ consists of labels, and $\mu_p=-d_p<0$ for $p\in P$.
In BNR's terminology, all nonnegative-contact markings, including positive
relative markings, are ordinary; negative-contact markings are punctures.
Let $\Lambda_\Gamma$ be BNR's datum with genus $1$, class $\beta$, and
ordered contacts consisting of $n$ zeros followed by
$\mu_1,\ldots,\mu_\rho$.  For the universal target,
the same symbol retains the genus and contacts and replaces $\beta$ by the
coarse degree $D\cdot\beta$.  Write
$\Punct_\Gamma(X\mid D)=\Punct_{\Lambda_\Gamma}(X\mid D)$ for the basic
stable punctured-map stack.  Superscript $\refc$ denotes BNR's refined
cycle \cite[Definition~1.13]{BNR}; superscript $\vir$ denotes the
orbifold virtual class from Behrend--Fantechi theory
\cite{BF97}\cite[Section~4.5]{AGV08}.

The common target and punctured forgetful map are
\begin{equation}
 B_\Gamma=\Mbar_{1,n+\rho}(X,\beta)\times_{X^\rho}D^\rho,
 \qquad
 \varrho:\Punct_\Gamma(X\mid D)\longrightarrow B_\Gamma,
\label{eq:common-target}
\end{equation}
using evaluation at the last $\rho$ markings; if $\rho=0$, then
$B_\Gamma=\Mbar_{1,n}(X,\beta)$.  ACGS's finite-type and properness
results apply because the global boundary class of a smooth Cartier
divisor generates the rationalized characteristic group
\cite[Theorems~3.12 and~3.18, Corollary~3.19]{ACGS}.  Thus
$\Punct_\Gamma(X\mid D)$ is proper over $\bbC$.  Since $B_\Gamma$ is
separated, the graph of $\varrho$ is proper as a base change of its
diagonal; composing with the proper projection to $B_\Gamma$ proves that
$\varrho$ is proper.  The graph need not be a closed immersion.

\subsection{Root-stack spaces and the Fan--Wu--You cycle}

For $N>M_{\boldsymbol\mu}$, put $X_{D,N}=\sqrt[N]{(X,D)}$.
The stack $\Mbar_\Gamma(X_{D,N})$ parametrises connected representable
twisted stable maps of genus one and class $\beta$, with untwisted interior
markings and the following ages of the normal root line and source
stabilizer orders at $x_i$:
\begin{equation}
 a_i(N)=
 \begin{cases}
  \mu_i/N,&\mu_i>0,\\
  (N+\mu_i)/N,&\mu_i<0,
 \end{cases}
 \qquad
 s_i(N)=\frac{N}{\gcd(N,|\mu_i|)}.
\label{eq:ages}
\end{equation}
Coarse stabilisation and evaluation give the proper map
\begin{equation}
 \tau_N:\Mbar_\Gamma(X_{D,N})\longrightarrow B_\Gamma.
\label{eq:orbifold-push}
\end{equation}
Properness follows from the twisted stable-map theory for the proper tame
Deligne--Mumford stack $X_{D,N}$.  Fan--Wu--You prove that
\begin{equation}
 F_\Gamma(N):=N^m\tau_{N*}[\Mbar_\Gamma(X_{D,N})]^{\vir}
\label{eq:FWY-polynomial}
\end{equation}
is eventually polynomial, with constant coefficient their bipartite
graph-sum cycle $\mathfrak c_\Gamma(X/D)$
\cite[Theorems~3.1 and~3.13, Definition~3.2]{FWY}.
As in the Introduction, $\CT_N$ extracts a polynomial coefficient, not an
analytic limit.

\begin{remark}[Expected dimension]
\label{subsec:expected-dimension}
For the coarse map $p_N:X_{D,N}\to X$,
\begin{equation}
 \begin{aligned}
 c_1(T_{X_{D,N}})
 &=p_N^*\!\left(c_1(T_X)-\left(1-\frac1N\right)[D]\right),\\
 \int_\beta c_1(T_{X_{D,N}})
 &=\int_\beta c_1(T_X)-\left(1-\frac1N\right)D\cdot\beta,
 \qquad
 \sum_i a_i(N)=m+\frac{D\cdot\beta}{N}.
 \end{aligned}
\label{eq:root-tangent-c1}
\end{equation}
Orbifold Riemann--Roch \cite[Theorem~7.2.1]{AGV08} gives
\begin{equation}
 \operatorname{vdim}\Mbar_\Gamma(X_{D,N})
 =\int_\beta c_1(T_X(-\log D))+n+\rho-m=d_\Gamma.
\label{eq:orbifold-vdim}
\end{equation}
For a smooth divisor the puncturing rank is $m$, so BNR's refined-class
dimension formula gives the same dimension
\cite[Definition~1.9 and the formula following Definition~1.13]{BNR}.
Consequently every coefficient of $F_\Gamma(N)$ and every term of
\eqref{eq:main-comparison} lies in $A_{d_\Gamma}(B_\Gamma)_\bbQ$.
\end{remark}

\subsection{Universal targets and the positivised numerical data}
\label{subsec:universal-targets}

The universal pair and its roots are
\begin{equation}
 \Acal=[\bbA^1/\bbG_m],\qquad \Dcal=B\bbG_m,\qquad
 \Acal_N=\sqrt[N]{(\Acal,\Dcal)},\qquad \Dcal_N\subset\Acal_N,
\label{eq:universal-pair}
\end{equation}
where $\Dcal_N$ is the reduced root divisor.  For
$\epsilon\in\{-,+\}$, the ambient stack $\overline O_N^\epsilon$
parametrises balanced tame twisted prestable curves $q:\cC\to C$ of genus
one with the ordered markings $\mathcal I$, together with a representable
map to $\Acal_N$.  The latter means an actual line--section pair
$(L_N^\epsilon,\sigma_N^\epsilon)$ and a chosen isomorphism
\[
 ((L_N^\epsilon)^{\otimes N},(\sigma_N^\epsilon)^{\otimes N})
 \simeq q^*(A^\epsilon,a^\epsilon),
\]
where $(A^\epsilon,a^\epsilon)$ is the descended pair on $C$.
The $N$th power has trivial stabilizer characters.  Node indices and
faithful characters may vary, subject to balancedness and representability.
Marking indices, characters, and total coarse degrees are fixed as follows.

For $\epsilon=-$, $\deg A^-=D\cdot\beta$, and the interior markings are
untwisted.  At $x_i$, put $g_i(N)=\gcd(N,|\mu_i|)$ and
$s_i(N)=N/g_i(N)$.  Relative to the standard marking generator, the
faithful reduced character is $\mu_i/g_i(N)$ if $\mu_i>0$ and
$s_i(N)-d_{x_i}/g_i(N)$ otherwise.  These give \eqref{eq:ages}.
No Crumplin degree inequality is imposed on $\overline O_N^-$.

\begin{definition}[Positivised numerical data]\label{def:positivised-data}
BNR's positivised datum $\Lambda_\Gamma^+$ retains all labels and positive
contacts, replaces the contact $-d_p$ at $p\in P$ by $0$ and its source
index by $1$, and has coarse degree
\begin{equation}
 d^+=D\cdot\beta+\sum_{p\in P}d_p=\sum_{\mu_i>0}\mu_i.
\label{eq:positive-degree}
\end{equation}
Thus its contacts are $c_i=\max(\mu_i,0)$ and its zero-contact labels are
$Z$, as in the Introduction \cite[Section~4.3.1, Definition~4.4]{BNR}.
The retained point $p$ is a \emph{former puncture}.  Contact order $0$ is
an assigned integer condition; it asserts neither vanishing nor
nonvanishing of the section at that point.
\end{definition}

For $\epsilon=+$, $\deg A^+=d^+$.  Positive markings keep the indices
and characters above, while every marking in $Z$ is untwisted, of age $0$.
Write $K_N^-:=\PunctOrb_{\Lambda_\Gamma}(\Acal_N\mid\Dcal_N)$ for BNR's
chimera and $K_N^+$ for its positive Cartesian construction, defined in
\eqref{eq:positive-space-fibre-product} below.

The Crumplin open $O_N^+\subset\overline O_N^+$ is defined by
\begin{equation}
 -\frac12<\deg(L_N^+|_E)<\frac12
 \qquad\text{for every proper subcurve }E.
\label{eq:crumplin-degree-requirement}
\end{equation}
The bar denotes the ambient stack, not a closure.  Twisting the negative
root line at $p\in P$ by $k_p=d_p/\gcd(N,d_p)$ times its tautological
marking divisor, multiplying its section by the corresponding tautological
section, and coarsening that marking gives
$\overline\Theta_N:\overline O_N^-\to\overline O_N^+$.
The construction and inverse are proved in
\Cref{lem:marking-picard-equivalence,prop:negative-zero-locus}.  Set
\begin{equation}
 O_N^-:=\overline O_N^-\times_{\overline O_N^+}O_N^+,
 \qquad \Theta_N:O_N^-\longrightarrow O_N^+.
\label{eq:negative-open-base-change}
\end{equation}

For tropical types, reserve $G$ for the dual graph, with vertices $V(G)$,
bounded edges $E(G)$ and labelled legs.  Write $\cC_v$ for the
normalization component indexed by $v$, $g_v$ for its genus, and $b_1(G)$
for the first Betti number of the graph.
For an oriented edge, write $\ell_e$ for its length and $m_e$ for its
coarse slope; reversal negates $m_e$, and incidence sums count both
half-edges of a loop.  The integer vertex degree
$d_v=N\deg(L_N^+|_{\cC_v})$ is the degree of the descended power on the
coarse component, not its contact order or target position.  The external vertices
$V_0$ support sections that are not identically zero, and the internal
vertices $V_+$ support identically zero sections; for a compatible tropical map they lie at the origin and in the
interior of the target ray, respectively.

\paragraph{Essential types.}
A mod-$N$ type $[\tau]$ is \emph{essential} if its graph is bipartite for
$V(G)=V_+\sqcup V_0$ and every internal vertex has positive genus.  It is
\emph{inducible} if its distinguished stratum $Z^\circ_{[\tau]}$ is
nonempty; write $Z_{[\tau]}$ for its reduced closure.  Crumplin identifies
the irreducible components with these closures for inducible essential
types \cite[Definitions~2.21 and~2.23, Theorem~2.24]{Crumplin}. 

The next examples illustrate the two possibilities in
$1=b_1(G)+\sum_v g_v$.  Black graphs are source graphs; the blue ray is
$\Sigma(\Acal,\Dcal)=\mathbb R_{\geq0}$, and blue decorations record
degrees, slopes, and contacts.

\begin{example}[A circuit type and its circulation parameter]
\label{ex:circuit-tropical-type}
For signed contacts $(2,-2)$, the positive contacts are $(2,0)$ and
$d^+=2$.  In \Cref{fig:circuit-tropical-type}, the two rational vertices
have degrees $2,0$.  Orient the upper edge left to right and the lower
edge right to left.  Their slope residues satisfy
\begin{equation}
 \overline m_1-\overline m_2\equiv2\pmod N,
 \qquad (\overline m_1,\overline m_2)=(a,a-2),\quad
 a\in\mathbb Z/N\mathbb Z.
\label{eq:example-circuit-balancing}
\end{equation}
Thus the weightings form a torsor under $\mathbb Z/N\mathbb Z$; choosing
one identifies the circulation parameter.  An integral lift has
$(m_1,m_2)=(k,k-2)$ with $k\in\mathbb Z$ reducing to $a$ modulo $N$.
Piecewise-linear continuity further requires
\[
 m_1\ell_1+m_2\ell_2=0.
\]
The edge lengths and continuity equation therefore constrain the
logarithmic lifts of an orbifold circulation residue.

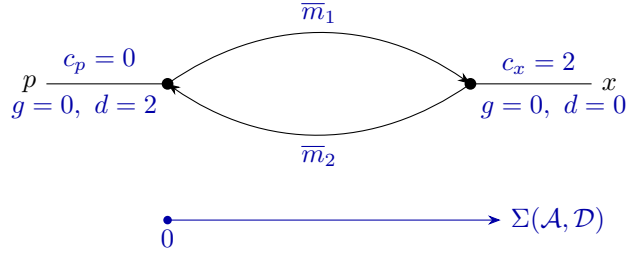
\begin{figure}[H]
 \centering
 \begin{tikzpicture}[>=Stealth, every node/.style={font=\small}]
  \coordinate (vL) at (0,0);
  \coordinate (vR) at (4,0);
  \draw[->] (vL) to[bend left=35]
    node[above,text=blue!65!black] {$\overline m_1$} (vR);
  \draw[->] (vR) to[bend left=35]
    node[below,text=blue!65!black] {$\overline m_2$} (vL);
  \draw (vL) -- (-1.6,0) node[left] {$p$};
  \draw (vR) -- (5.6,0) node[right] {$x$};
  \node[text=blue!65!black,above] at (-.9,0) {$c_p=0$};
  \node[text=blue!65!black,above] at (4.9,0) {$c_x=2$};
  \fill (vL) circle (2.2pt); \fill (vR) circle (2.2pt);
  \node[below left,text=blue!65!black] at (vL) {$g=0,\ d=2$};
  \node[below right,text=blue!65!black] at (vR) {$g=0,\ d=0$};
  \draw[->,blue!65!black] (0,-1.8) -- (4.4,-1.8)
    node[right] {$\Sigma(\Acal,\Dcal)$};
  \fill[blue!65!black] (0,-1.8) circle (1.6pt);
  \node[below,text=blue!65!black] at (0,-1.8) {$0$};
 \end{tikzpicture}
 \caption{A genus-one circuit type. All vertex genera vanish and
 $b_1(G)=1$.  The leg $p$ is the former puncture of contact order $0$
 in the positivised data.  Edge arrows specify the orientations used in
 \eqref{eq:example-circuit-balancing}; residue data alone do not specify
 a realizable piecewise-linear map to the ray.}
 \label{fig:circuit-tropical-type}
\end{figure}
\end{example}

\begin{example}[An elliptic-vertex type]
\label{ex:elliptic-vertex-type}
For signed contacts $(1,1,-2)$ at $x_1,x_2,p$ (with $p=x_3$), the
positive contacts are $(1,1,0)$, $Z=\{p\}$, and $d^+=2$.
In the tropical map $h:G\to\mathbb R_{\geq0}$ of
\Cref{fig:elliptic-vertex-type}, rational vertices $u_1,u_2$ map to
$0$ and the genus-one vertex $v_+$ maps to $a>0$.  Each edge $e_i$ is
oriented from $u_i$ to $v_+$, with $m_{e_i}=1$ and $\ell_{e_i}=a$, so
$h(v_+)-h(u_i)=m_{e_i}\ell_{e_i}$.  Outgoing-slope balancing gives
\[
 d_{u_i}=m_{e_i}+c_{x_i}=2,\qquad
 d_{v_+}=-m_{e_1}-m_{e_2}+c_p=-2,\qquad
 \sum_v d_v=2=d^+.
\]
These are degree decorations, distinct from contacts and heights
\cite[Definitions~2.2 and~2.4, Construction~2.9]{Crumplin}.
The graph is a tree, so balancing uniquely determines its slopes
\cite[Lemma~2.8]{Crumplin}.  Here $V_0=\{u_1,u_2\}$ and
$V_+=\{v_+\}$; the bipartite graph with its positive-genus internal
vertex is a nontrivial essential type.  Although $c_p=0$, the section
vanishes at $p$ because its entire supporting component is internal.
An independent factor $q_{v_+}^*J$ with
$J\in\operatorname{Pic}^0(C_{v_+})$ may occur on that orbifold component,
where $q_{v_+}:\cC_{v_+}\to C_{v_+}$ is its coarse projection
\cite[Proposition~2.13(2)]{Crumplin}.  Thus eliminating graph circulation
does not eliminate the elliptic line-bundle freedom.

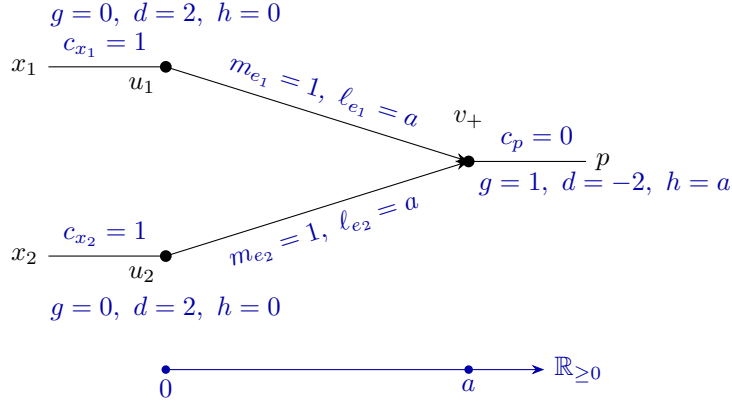
\begin{figure}[H]
 \centering
 \begin{tikzpicture}[>=Stealth, every node/.style={font=\small}]
  \coordinate (u1) at (0,1.25);
  \coordinate (u2) at (0,-1.25);
  \coordinate (vp) at (4,0);
  \draw[->] (u1) -- node[above,sloped,text=blue!65!black]
    {$m_{e_1}=1,\ \ell_{e_1}=a$} (vp);
  \draw[->] (u2) -- node[below,sloped,text=blue!65!black]
    {$m_{e_2}=1,\ \ell_{e_2}=a$} (vp);
  \draw (u1) -- (-1.55,1.25) node[left] {$x_1$};
  \draw (u2) -- (-1.55,-1.25) node[left] {$x_2$};
  \draw (vp) -- (5.55,0) node[right] {$p$};
  \node[above,text=blue!65!black] at (-.8,1.25) {$c_{x_1}=1$};
  \node[above,text=blue!65!black] at (-.8,-1.25) {$c_{x_2}=1$};
  \node[above,text=blue!65!black] at (4.9,0) {$c_p=0$};
  \fill (u1) circle (2.2pt); \fill (u2) circle (2.2pt);
  \fill (vp) circle (2.2pt);
  \node[below left] at (u1) {$u_1$};
  \node[below left] at (u2) {$u_2$};
  \node[above] at (4,.3) {$v_+$};
  \node[above,text=blue!65!black] at (0,1.65) {$g=0,\ d=2,\ h=0$};
  \node[below,text=blue!65!black] at (0,-1.65) {$g=0,\ d=2,\ h=0$};
  \node[below right,text=blue!65!black] at (vp) {$g=1,\ d=-2,\ h=a$};
  \draw[->,blue!65!black] (0,-2.75) -- (5,-2.75)
    node[right] {$\mathbb R_{\geq0}$};
  \fill[blue!65!black] (0,-2.75) circle (1.6pt);
  \fill[blue!65!black] (4,-2.75) circle (1.6pt);
  \node[below,text=blue!65!black] at (0,-2.75) {$0$};
  \node[below,text=blue!65!black] at (4,-2.75) {$a$};
 \end{tikzpicture}
 \caption{A nontrivial elliptic-vertex type.  The rational vertices
 $u_1,u_2$ map to $0$, while the genus-one internal vertex $v_+$ maps to
 $a>0$. Blue labels at vertices, edges, and legs
 record degrees, slopes, and contact orders.  At a vertex $v$, the labels
 $g,d,h$ abbreviate $g_v,d_v,h(v)$, respectively.}
 \label{fig:elliptic-vertex-type}
\end{figure}
\end{example}

\subsection{The BNR space, the finite-type substacks, and the comparison maps}
\label{subsec:theorem-spaces}

For the universal Chow calculations in Sections~3--5, assume
$n+\rho>0$.  This is automatic when $m>0$; Section~6 treats $m=0$
independently.  Forgetting $Z$ below is used to describe components; no
proper Chow pushforward on a completely unmarked genus-one auxiliary
stack is required.

\paragraph{The BNR space for the positivised data.}
For a cone stack $T$, let $\mathcal A(T)$ be the Artin fan obtained by
gluing its Artin cones \cite[Theorem~3]{CCUW20}; thus
$\Acal=\mathcal A(\mathbb R_{\geq0})$.
Let $\mathsf M^{\mathrm{trop}}$ and
$\mathsf M^{\mathrm{trop,tw}}$ be the cone stacks of ordinary and twisted
tropical prestable curves, with genus and labels matching the accompanying
curve stack.  BNR's rooted cone stack $T_N^+$ for $\Lambda_\Gamma^+$
records twisted dual graphs, genera, degrees, marking and node indices,
vertex positions, slopes, continuity equations, image faces, and graph
automorphisms.  Write $T^-$ and $T_N^-$ for the unrooted and rooted
cone stacks for $\Lambda_\Gamma$, retaining all labels.
BNR's construction is
\begin{equation}
 K_N^+:=\Mfrak^{\mathrm{tw}}_{1,n+\rho}
 \mathop{\times}_{\mathcal A(\mathsf M^{\mathrm{trop,tw}})}
 \mathcal A(T_N^+).
\label{eq:positive-space-fibre-product}
\end{equation}
Here the twisted-curve stack has the prescribed marking indices; its map
to the Artin fan comes from the nodal boundary log structure, and the
other map forgets the tropical map but retains its twisted source
\cite[Construction~3.15 and Lemma~3.18]{BNR}.  The rooted line--section
pair defines
\begin{equation}
 \omega_N^+:K_N^+\longrightarrow\overline O_N^+.
\label{eq:positive-forgetful-map}
\end{equation}
It forgets the compatible basic logarithmic and tropical enhancement.
The analogous negative map is $\omega_N^-:K_N^-\to\overline O_N^-$.
The positive target is ambient because the full cone stack need not
satisfy \eqref{eq:crumplin-degree-requirement}.  Section~5 proves the
modular interpretation and the required properties of these maps.

\paragraph{Forgetting the markings indexed by $Z$.}
Subscript $Z$ means that all labels are retained:
$\overline O_{N,Z}^+=\overline O_N^+$, $O_{N,Z}^+=O_N^+$, and
$K_{N,Z}^+=K_N^+$.  Subscript $\varnothing$ means that exactly $Z$
has been forgotten, \emph{without stabilising the prestable curve or
changing its root pair}.  This defines ambient and degree-restricted
orbifold stacks and the corresponding BNR construction, with maps
\begin{equation}
 F_N:O_{N,Z}^+\longrightarrow O_{N,\varnothing}^+,
 \qquad
 F_N^K:K_{N,Z}^+\longrightarrow K_{N,\varnothing}^+,
 \qquad
 \omega_{N,\varnothing}^+:K_{N,\varnothing}^+
 \longrightarrow\overline O_{N,\varnothing}^+.
\label{eq:zero-order-marking-forgetful}
\end{equation}
We also write $F_N$ for the ambient map.  These maps commute with
$\omega_N^+$; their Cartesian compatibility on the bounded substacks is
proved in \Cref{cor:zero-order-marking-base-change}.
Under Crumplin's size, divisibility, and coprimality assumptions
\cite[Section~2.1.1]{Crumplin}, including $N>2d^+$, the stack
$O_{N,\varnothing}^+$ is his stack for the remaining positive contacts.
Let $Z^\circ_{\main,N,\varnothing}$ be its distinguished trivial-type
locus and $Z_{\main,N,\varnothing}$ its closure, called the main
component \cite[Section~2.1.3]{Crumplin}.  The marked versions are
\begin{equation}
 Z^\circ_{\main,N}:=
 O_{N,Z}^+\times_{O_{N,\varnothing}^+}Z^\circ_{\main,N,\varnothing},
 \qquad
 Z_{\main,N}:=\overline{Z^\circ_{\main,N}}^{\,O_{N,Z}^+}.
\label{eq:main-locus-marking-base-change}
\end{equation}

Let $\mathfrak U_{\varnothing}^+$ parametrise connected genus-one
prestable curves with the remaining positive-contact markings and an
actual line--section pair $(M,t)$ of degree $d^+$, with no root or
logarithmic enhancement.  The map
$O_{N,\varnothing}^+\to\mathfrak U_{\varnothing}^+$ retains the coarse
curve and descended $N$th power pair.

\paragraph{The finite-type substack determined by the geometric problem.}
On the universal stable curve $\pi_\Gamma:C_\Gamma\to B_\Gamma$ with
map $f_\Gamma:C_\Gamma\to X$, define
\begin{equation}
 M_\Gamma^+:=f_\Gamma^*\cO_X(D)\otimes
 \cO_{C_\Gamma}\!\left(\sum_{p\in P}d_pp\right),
 \qquad
 t_\Gamma^+:=f_\Gamma^*s_D\otimes\prod_{p\in P}\sigma_p^{d_p},
\label{eq:canonical-positive-pair}
\end{equation}
where $s_D$ and $\sigma_p$ are the canonical sections of $\cO_X(D)$
and $\cO_{C_\Gamma}(p)$.  Forgetting $Z$ without stabilisation gives
\begin{equation}
 g_\Gamma^+:B_\Gamma\longrightarrow\mathfrak U_{\varnothing}^+.
\label{eq:canonical-unrooted-map}
\end{equation}
For a nodal curve $C$ with graph $G_C$ and $A\subseteq V(G_C)$,
let $C_A$ be the corresponding union of components.  Write $G_x$ and
$M_{\Gamma,x}^+$ for the graph and line over a geometric point
$x\in|B_\Gamma|$.  Set
\begin{equation}
 B_0:=\max\left\{
 1,d^+,M_{\boldsymbol\mu},
 \max_{x\in|B_\Gamma|}|E(G_x)|,
 \max_{\substack{x\in|B_\Gamma|\\A\subseteq V(G_x)}}
 \left|\deg(M_{\Gamma,x}^+|_{C_A})\right|
 \right\},
\label{eq:B0-definition}
\end{equation}
with the last two maxima zero if $B_\Gamma$ is empty.  Let
$W_0\subset\mathfrak U_{\varnothing}^+$ be the full substack satisfying,
on every geometric fibre,
\begin{equation}
 |E(G_C)|\leq B_0,\qquad
 |\deg(M|_{C_A})|\leq B_0
 \quad\text{for every }A\subseteq V(G_C).
\label{eq:W0-definition}
\end{equation}
Section~4 proves that $B_0$ is finite and $W_0$ is a finite-type open
substack containing $g_\Gamma^+(B_\Gamma)$ and the smooth one-vertex
locus with nonzero section.

Choose a base root order satisfying
\begin{equation}
 r>\max\{4B_0,2d^+,M_{\boldsymbol\mu}\},\qquad
 \operatorname{lcm}(1,\ldots,4B_0)\mid r,\qquad
 |\mu_i|\mid r\quad(1\leq i\leq\rho).
\label{eq:large-r}
\end{equation}
For $N=\lambda r$, $\lambda\in\bbZ_{\geq1}$, define
\begin{equation}
 \begin{aligned}
 W_{N,\varnothing}
 &:=O_{N,\varnothing}^+\times_{\mathfrak U_{\varnothing}^+}W_0,\\
 W_{N,Z}
 &:=O_{N,Z}^+\times_{O_{N,\varnothing}^+}W_{N,\varnothing}.
 \end{aligned}
\label{eq:bounded-rooted-opens}
\end{equation}
These are full base changes; no irreducible component is deleted.
Twisting and coarsening the geometric root pair at $P$ gives a map
$\Mbar_\Gamma(X_{D,N})\to W_{N,Z}$: its descended pair is
\eqref{eq:canonical-positive-pair}, whose image lies in $W_0$, and
its subcurve root degrees are bounded by $B_0/N<1/2$.

Let $U_r=Z^\circ_{\main,r,\varnothing}\subset W_{r,\varnothing}$.
This inclusion holds because its curve has one vertex, no nodes, and
degree $d^+\leq B_0$.  Set
\begin{equation}
 \begin{aligned}
 U_{r,Z}&:=O_{r,Z}^+\times_{O_{r,\varnothing}^+}U_r,\\
 K_{W_{r,\varnothing}}^+
 &:=K_{r,\varnothing}^+\times_{\overline O_{r,\varnothing}^+}
 W_{r,\varnothing}.
 \end{aligned}
\label{eq:reintroduced-distinguished-locus}
\end{equation}

\paragraph{Comparison between root orders.}
For $R=\lambda r$, taking the $\lambda$th tensor power of the root pair
and then the AOV relative coarse space for the kernel of its stabilizer
characters defines
\begin{equation}
 \overline\pi_{R,r}^\pm:\overline O_R^\pm\longrightarrow
 \overline O_r^\pm.
\label{eq:root-comparison}
\end{equation}
The analogous map after forgetting $Z$ is
$\overline\pi_{R,r}:\overline O_{R,\varnothing}^+\to
\overline O_{r,\varnothing}^+$.  Tensor power and coarsening retain the
coarse curve and descended pair.  Moreover, over $W_0$ the root degree
on every subcurve has absolute value at most $B_0/N<1/2$ at both orders.
Consequently the bounded comparisons are the full base changes
\begin{equation}
 \begin{aligned}
 W_{R,\varnothing}
 &\simeq\overline O_{R,\varnothing}^+
 \times_{\overline O_{r,\varnothing}^+}W_{r,\varnothing},
 &\pi_{R,r}&:W_{R,\varnothing}\to W_{r,\varnothing},\\
 W_{R,Z}
 &\simeq\overline O_R^+\times_{\overline O_r^+}W_{r,Z},
 &\pi_{R,r}^+&:W_{R,Z}\to W_{r,Z}.
 \end{aligned}
\label{eq:full-high-open}
\end{equation}
Section~4 proves finite type and that $\pi_{R,r}^+$ is proper,
quasi-finite, and of Deligne--Mumford type.  This does not assert that
the whole of $O_R^+$ maps into $O_r^+$.

Finally, the bounded BNR space and main component are
\begin{equation}
 K_{W_r}^+:=K_r^+\times_{\overline O_r^+}W_{r,Z},\qquad
 \omega_{W_r}^+:K_{W_r}^+\longrightarrow W_{r,Z},
\label{eq:positive-space-open}
\end{equation}
\begin{equation}
 Z_{W_r}^{\main}
 :=\bigl(Z_{\main,r}\times_{O_r^+}W_{r,Z}\bigr)_{\red}.
\label{eq:restored-main-component}
\end{equation}
The polynomial in \Cref{thm:main-component} is
\begin{equation}
 Q_Z(\lambda):=(\pi_{\lambda r,r}^+)_*
 [W_{\lambda r,Z}]^{\vir}\in A_*(W_{r,Z})_\bbQ.
\label{eq:QZ}
\end{equation}
Sections~4 and~5 establish the properness and polynomiality required here.
\section{Negative inside positive}\label{sec:negative-in-positive}

We compare the original and positivised orbifold spaces by twisting at the
negative markings, following BNR \cite[Section~4.3.1, Definition~4.4]{BNR}.
The resulting zero loci need not have their expected ordinary dimension;
refined Gysin pullback gives the scaled identity
\eqref{eq:scaled-root-change} used in \Cref{prop:universal-identity}.
The arguments of this section apply in every genus.

\subsection{Twisting at the negative markings}

\begin{lemma}[Twisted-Picard equivalence at the markings indexed by $P$]
\label{lem:marking-picard-equivalence}
Fix integers $s_p>1$ and $0<k_p<s_p$ with
$\gcd(k_p,s_p)=1$.  For a base scheme $S$, let $\mathfrak P^+(S)$ classify
a family of tame twisted prestable curves $\pi:C\to S$ with pairwise
disjoint sections $(p)_{p\in P}$ in the nonstacky smooth locus and a line
bundle $M$ on $C$.  An object of $\mathfrak P^-(S)$ consists of
such a family $C\to S$, the root stack
\[
 q:\cC=\sqrt[s_p]{(C,p)}_{p\in P}\longrightarrow C,
\]
and a line bundle $L$ whose character at the $p$th marking gerbe is
$-k_p\pmod{s_p}$.  Here $q$ is the partial coarsening at the markings indexed by $P$,
and $\widetilde p\subset\cC$ is the tautological marking gerbe above
$p$, so $q^*p=s_p\widetilde p$ as Cartier divisors.
Fix the same labelled curve and the same node stabilizer orders and
characters.  The line degrees and the chosen target-root power isomorphisms on
the two sides are required to be related by the displayed twist and by
\eqref{eq:picard-N-power}; they are not asserted to be numerically equal.  Then
\begin{equation}
 (C,M)\longmapsto
 \left(\cC,q^*M\otimes\cO_\cC
       \left(-\sum_{p\in P}k_p\widetilde p\right)\right)
\label{eq:picard-forward}
\end{equation}
is an equivalence of stacks.  Its inverse twists in the opposite direction
and descends along $q$.  It commutes with arbitrary base change, is an
equivalence on Isom sheaves, and introduces no additional degree factor from the marking gerbes.
Suppose in addition that an integer $N$ is divisible by every $s_p$, and put
$\delta_p=N/s_p$.  If $A$ is a line bundle on $C$, then this equivalence
is compatible with $N$th powers in the precise sense that
\begin{equation}
 L^{\otimes N}\simeq q^*A
 \quad\Longleftrightarrow\quad
 M^{\otimes N}\simeq
 A\Bigl(\sum_{p\in P}k_p\delta_p p\Bigr).
\label{eq:picard-N-power}
\end{equation}
For each $p$, let $T_p=\cO_\cC(\widetilde p)$ with tautological section
$\vartheta_p$, and let $x_p$ be the canonical section of $\cO_C(p)$.  Thus
\begin{equation}
 T_p^{\otimes s_p}\simeq q^*\cO_C(p),\qquad
 \vartheta_p^{s_p}=q^*x_p.
\label{eq:tautological-marking-root}
\end{equation}
Multiplication by these tautological sections identifies sections of $L$
with sections of $M$ that vanish scheme-theoretically at every marking
indexed by $P$.  Writing $s_-$ for a section of $L$ and $s_+$ for
the corresponding section of $M$, the identification is
\begin{equation}
 H^0(\cC,L)\xrightarrow{\sim}
 \ker\!\left(H^0(C,M)\longrightarrow
              \bigoplus_{p\in P}H^0(S,p^*M)\right),
 \qquad
 q^*s_+=s_-\prod_{p\in P}\vartheta_p^{k_p}.
\label{eq:section-descent-with-vanishing}
\end{equation}
This is an isomorphism of the corresponding functors of sections and
commutes with arbitrary base change.  If $a$ is a section of $A$ and, under the left-hand
isomorphism in \eqref{eq:picard-N-power},
$s_-^{\otimes N}=q^*a$, then, under the right-hand isomorphism,
\begin{equation}
 s_+^{\otimes N}=a\prod_{p\in P}x_p^{k_p\delta_p}.
\label{eq:section-N-power}
\end{equation}
Conversely, a section $s_+$ satisfying the vanishing condition and this
$N$th-power equation determines a unique section $s_-$ mapping to it.  The equivalences retain
the chosen $N$th-power isomorphisms as part of the data, on both objects and
arrows.
\end{lemma}

\begin{proof}
The tautological line of the canonical marking root stack has character
one, so \eqref{eq:picard-forward} has character $-k_p$.  Conversely, twisting
by $\sum k_p\widetilde p$ kills the relative stabilizer characters, and
tame descent recovers $M$.  Exact invariant pushforward and pullback are
inverse on these bundles and their arrows, compatibly with arbitrary base
change.

To check the stack degree explicitly, fix $p\in P$ and abbreviate $s=s_p$, $k=k_p$, $T=T_p$, and
$\vartheta=\vartheta_p$.
Over a base ring $B$, use the local chart
$[\Spec(B[x,z]/(z^s-x))/\mu_s]\to\Spec B[x]$, $x=z^s$, and write
$L=q^*M\otimes T^{-k}$ with $T=\cO(\widetilde p)$.  A lift of the identity
on the pointed curve $C$ together with a line isomorphism is a pair
$(a,b)\in\mu_s\times\bbG_m$.  A $2$-isomorphism
$\eta\in\mu_s$ identifies $(a,b)$ with
$(\eta a,b\eta^k)$.  Therefore
\[
 (\mu_s\times\bbG_m)/\mu_s\simeq\bbG_m,
 \qquad (a,b)\longmapsto ba^{-k},
\]
which is the Isom sheaf of $M$.  The marking gerbe is labelled but not neutralised, and no root of a base
line bundle is chosen.  Thus there is no residual $B\mu_s$ factor.

For $\delta=N/s$, the tautological identity
$(T,\vartheta)^{\otimes s}=q^*(\cO_C(p),x_p)$ gives
\[
 \bigl(q^*M\otimes T^{-k}\bigr)^{\otimes N}
 =q^*\bigl(M^{\otimes N}(-k\delta p)\bigr).
\]
Full faithfulness of $q^*$, since $q_*\cO_\cC=\cO_C$, proves
\eqref{eq:picard-N-power} with its chosen isomorphisms.  Here
$\deg\cO_\cC(k\widetilde p)=k/s$; the twists multiply over the disjoint markings.

On the same chart, trivialize $M$.  The $(-k)$-weight summand
of $B[x,z]/(z^s-x)$ is the free $B[x]$-module generated by $z^{s-k}$.
Consequently every section of $L=q^*M\otimes T^{-k}$ is uniquely
\[
 s_-=z^{s-k}u(x),\qquad u(x)\in B[x].
\]
Multiplication by $\vartheta^k=z^k$ identifies this module with the Cartier
ideal $(x)$ of the marking.  Linear reductivity of $\mu_s$ makes the weight
decomposition valid over arbitrary, possibly nonreduced bases.  Combining
the disjoint markings proves \eqref{eq:section-descent-with-vanishing}.

Finally, if $s_-^{\otimes N}=q^*a$, then
\[
 q^*(s_+^{\otimes N})
 =s_-^{\otimes N}\vartheta^{kN}
 =q^*(a x^{k\delta}),
\]
which proves \eqref{eq:section-N-power}; the converse follows by reversing
the calculation, since $x$ is a nonzerodivisor.  All identifications retain
the chosen power isomorphisms and commute with arrows and base change.
\end{proof}

\begin{remark}[Relation with the positivised data]
The twist absorbs the stacky negative marking and records the change in
its $N$th power on the coarse line.  It contributes no degree factor; the
later normalization factors come from zero-section pullback and
root-forgetting.
\end{remark}

\subsection{The negative space as a zero locus in the positive space}

Let $\pi_N^\pm:\cC_N^\pm\to\overline O_N^\pm$ denote the
universal twisted curves.  Write
$(M_N,\sigma_N):=(L_N^+,\sigma_N^+)$ and
$(L_N,\sigma_N^-):=(L_N^-,\sigma_N^-)$ for their universal root
line--section pairs.  These notations also denote their restrictions to $O_N^\pm$ and the
specified base changes.  For each $p\in P$ set
\begin{equation}
 \begin{aligned}
 E_{p,N}&=p^*M_N, & e_{p,N}&=p^*\sigma_N,\\
 E_N&=\bigoplus_{p\in P}E_{p,N}, & e_N&=(e_{p,N})_{p\in P}.
 \end{aligned}
\label{eq:evaluation-bundle}
\end{equation}

\begin{proposition}[The negative space as a zero locus]
\label{prop:negative-zero-locus}
Assume $N>M_{\boldsymbol\mu}$ and $d_p\mid N$ for every $p\in P$.
There is a canonical morphism
$\overline\Theta_N:\overline O_N^-\to\overline O_N^+$ identifying
$\overline O_N^-$ with the scheme-theoretic zero locus $Z(e_N)$.
Let $\Theta_N:O_N^-\to O_N^+$ be its restriction.

Let $\mathfrak P_N^\pm$ be the line-level stacks obtained by forgetting both
the root section and its descended $N$th-power section, while retaining the
marked twisted curve, the root line, the descended coarse line, their total
degrees, all remaining node data, and the chosen $N$th-power isomorphism of
line bundles.
Lemma~\ref{lem:marking-picard-equivalence} canonically identifies them with a
common fixed-character sector $\mathfrak P_N^{\mathrm{sec}}$ of the universal twisted Picard
stack.  This sector is smooth over the stack of twisted curves and hence is a
valid smooth base for the standard obstruction theories of sections.  Relative to
$\mathfrak P_N^{\mathrm{sec}}$, the standard obstruction theories of sections fit into a
compatible distinguished triangle
\begin{equation}
 \mathbf L\Theta_N^*\mathbb E_N^+\longrightarrow
 \mathbb E_N^-\longrightarrow
 \mathbf L\Theta_N^*(E_N^\vee)[1]\longrightarrow
 (\mathbf L\Theta_N^*\mathbb E_N^+)[1],
\label{eq:pot-triangle}
\end{equation}
where
\[
 \mathbb E_N^+=\bigl(\mathbf R(\pi_N^+)_*M_N\bigr)^\vee,
 \qquad
 \mathbb E_N^-=\bigl(\mathbf R(\pi_N^-)_*L_N\bigr)^\vee.
\]
Define $\Theta_N^!$ to be the refined pullback obtained by base change from
the regular zero section of the vector bundle $E_N\to O_N^+$.
For every base change $T\to O_N^+$ considered here, it induces an operator
\[
 \Theta_N^!:A_k(T)_\bbQ\longrightarrow
 A_{k-m}(T\mathop{\times}_{O_N^+}O_N^-)_\bbQ.
\]
We use the same symbol for these operators.  For every finite-type open
$U_N^+\subset O_N^+$, set
\begin{equation}
 U_N^-:=U_N^+\mathop{\times}_{O_N^+}O_N^-.
\label{eq:negative-finite-open-base-change}
\end{equation}
Then
\begin{equation}
 [U_N^-]^{\vir}=\Theta_N^![U_N^+]^{\vir}
 \quad\text{in }A_*(U_N^-)_\bbQ.
\label{eq:positive-negative-vfc}
\end{equation}
No additional stabilizer or root-order factor occurs.
\end{proposition}

\begin{proof}
If $P=\varnothing$, all assertions are tautological.  Suppose otherwise.
At a puncture $p$ with contact $-d_p$, put $s=N/d_p$.  The negative root
line has character $s-1$.  The case $k_p=1$ of
Lemma~\ref{lem:marking-picard-equivalence}, applied at the disjoint
punctures, gives the partial marking coarsening
\[
 q_N:\cC_N^-\longrightarrow
 \cC_{N,-}^+:=\cC_N^+\mathop{\times}_{\overline O_N^+}
 \overline O_N^-.
\]
It removes only the stabilizers at the markings indexed by $P$.
On $\cC_{N,-}^+$, write $M_N$ also for the pullback of the positive
universal line.  With this convention, the marking twist is
\begin{equation}
 L_N\Bigl(\sum_{p\in P}\widetilde p\Bigr)\simeq q_N^*M_N.
\label{eq:twist-positive-line}
\end{equation}
The functorial section identification
\eqref{eq:section-descent-with-vanishing} identifies negative pairs with
positive pairs whose evaluations at all former punctures vanish.
Equation~\eqref{eq:section-N-power}, with $N/s=d_p$, preserves the chosen
target-root power equations.  Thus $\overline O_N^-=Z(e_N)$
scheme-theoretically, over arbitrary bases and on arrows in every fixed
discrete sector.

For the obstruction-theory calculation, restrict these universal curves
to $O_N^-$ and let $i_{p,N}:\widetilde p\hookrightarrow\cC_N^-$ be
the marking immersion.  On the negative universal curve there is an exact
sequence
\begin{equation}
 0\longrightarrow L_N\longrightarrow q_N^*M_N\longrightarrow
 \bigoplus_{p\in P}(i_{p,N})_*i_{p,N}^*q_N^*M_N\longrightarrow0.
\label{eq:marking-exact-sequence}
\end{equation}
The quotient has trivial inertia at each marking.  Tame pushforward along
$q_N$ is exact; after pushforward along the universal curve, its derived
pushforward is $\Theta_N^*E_N$ concentrated in degree zero.  Derived
pushforward of the exact sequence therefore gives
\begin{equation}
 \mathbf R(\pi_N^-)_*L_N \longrightarrow \mathbf L\Theta_N^*\mathbf R(\pi_N^+)_*M_N \longrightarrow \mathbf L\Theta_N^*E_N \longrightarrow \bigl(\mathbf R(\pi_N^-)_*L_N\bigr)[1].
\label{eq:pushforward-marking-triangle}
\end{equation}
Dualizing and rotating gives \eqref{eq:pot-triangle}.  The arrows to the
relative cotangent complexes are the standard deformation-theoretic arrows
for sections and for the regular zero section.  Exactness of
\eqref{eq:marking-exact-sequence} makes the resulting diagram of triangles
commute.  Manolache's composition theorem
\cite[Theorem~4.8]{Manolache} then gives
\eqref{eq:positive-negative-vfc}.  Finally,
Lemma~\ref{lem:marking-picard-equivalence} is an equivalence of stacks, so no
extra degree factor is introduced.
\end{proof}

\begin{remark}[Negative contact as a zero-section condition]
Recovering the negative theory imposes evaluation vanishing by an ordinary
zero-section Gysin pullback, as in BNR's classical puncturing construction.
No separate derived enhancement of the chimera is needed.
\end{remark}

\subsection{The defining equations at different root orders}

\begin{lemma}[Comparison of the evaluation equations]
\label{lem:root-change-evaluation}
Let $r$ satisfy \eqref{eq:large-r}, let $\lambda\geq1$ be an integer, and
put $R=\lambda r$.  On the ambient orbifold stacks there are canonical
isomorphisms of lines with sections
\begin{equation}
 (\overline\pi_{R,r}^+)^*(E_{p,r},e_{p,r})\simeq
 (E_{p,R}^{\otimes\lambda},e_{p,R}^{\lambda})
 \qquad(p\in P).
\label{eq:power-section}
\end{equation}
Moreover the square
\[
\begin{tikzcd}
 \overline O_R^-\ar[r,"\overline\Theta_R"]
   \ar[d,"\overline\pi_{R,r}^-"']&
 \overline O_R^+\ar[d,"\overline\pi_{R,r}^+"]\\
 \overline O_r^-\ar[r,"\overline\Theta_r"']&\overline O_r^+
\end{tikzcd}
\]
is $2$-commutative.  Both assertions restrict to every specified
finite-type base change.
\end{lemma}

\begin{proof}
The target morphism $\Acal_R\to\Acal_r$ sends the positive root pair to
$(M_R^{\otimes\lambda},\sigma_R^\lambda)$ before source coarsening.  At
$p\in P$, the positive contact is zero and the source index is one at both
orders, so coarsening is the identity near $p$.  Restriction to $p$ proves
\eqref{eq:power-section}.

For compatibility with the negative datum, work over $\overline O_R^-$, pulling back the order-$r$ negative
universal curve and its line bundles along $\overline\pi_{R,r}^-$.
For $p\in P$ and $N=r,R$, write $s_{p,N}:=N/d_p$.  Divisibility gives
$s_{p,R}=R/d_p=\lambda s_{p,r}$.  Let $\widetilde p_N$ be the
marking gerbe at order $N$ and put $T_{p,N}=\cO(\widetilde p_N)$
on the negative order-$N$ curve.  The canonical
partial coarsening $\varpi:\cC_R^-\to\cC_r^-$ satisfies
\[
 \varpi^*T_{p,r}\simeq T_{p,R}^{\otimes\lambda};
\]
locally this is $w=z^\lambda$ on the marking-root charts.  The line-bundle
modification defining $\Lambda_\Gamma^+$ gives
$L_N\simeq q_N^*M_N\otimes T_{p,N}^{-1}$ at this marking, whence
\[
 \varpi^*L_r
 \simeq q_R^*M_R^{\otimes\lambda}\otimes T_{p,R}^{-\lambda}
 \simeq L_R^{\otimes\lambda}.
\]
The tautological sections satisfy the same identities at each disjoint
former puncture.  Elsewhere positivization changes no curve or line data,
including at nodes and other markings.  These canonical identifications
are functorial on arrows, proving $2$-commutativity.
\end{proof}

\begin{lemma}[Refined-Gysin formula for powered sections]
\label{lem:power-gysin}
Let $r,\lambda,R$ be as in \Cref{lem:root-change-evaluation}, and let
$U_r^+\subset O_r^+$ be a finite-type open.  Define $U_R^+$ and,
for $N=r,R$, $U_N^-$ by the $2$-Cartesian squares
\begin{equation}
\begin{tikzcd}
 U_R^+\ar[r]\ar[d,"f"']
   & \overline O_R^+\ar[d,"\overline\pi_{R,r}^+"]\\
 U_r^+\ar[r,hook]&\overline O_r^+,
\end{tikzcd}
\qquad
\begin{tikzcd}
 U_N^-\ar[r]\ar[d]
   & O_N^-\ar[d,"\Theta_N"]\\
 U_N^+\ar[r,hook]&O_N^+.
\end{tikzcd}
\label{eq:power-gysin-base-change-squares}
\end{equation}
By \Cref{lem:root-change-evaluation}, the ambient negative comparison
restricts to $\pi_{R,r}^-:U_R^-\to U_r^-$.
Assume that $f:U_R^+\to U_r^+$ is proper.  Then, for every
$\xi\in A_*(U_R^+)_\bbQ$,
\begin{equation}
 (\pi_{R,r}^-)_*\Theta_R^!(\xi)
 =\lambda^{-m}\Theta_r^!f_*(\xi).
\label{eq:power-gysin-formula}
\end{equation}
In particular,
\begin{equation}
 R^m(\pi_{R,r}^-)_*[U_R^-]^{\vir}
 =r^m\Theta_r^!f_*[U_R^+]^{\vir}
\label{eq:scaled-root-change}
\end{equation}
in $A_*(U_r^-)_\bbQ$.  No additional gcd factor or stabilizer-degree factor occurs.
\end{lemma}

\begin{proof}
Set $Y_r=U_r^+$, $Y_R=U_R^+$, $V_r=U_r^-=Z(e_r)$, and
$V_R=U_R^-=Z(e_R)$.  Form the Cartesian square
\begin{equation}
\begin{tikzcd}
 \widetilde V:=Y_R\times_{Y_r}V_r\ar[r,"j"]\ar[d,"\widetilde f"'] &
 Y_R\ar[d,"f"]\\
 V_r\ar[r,"\Theta_r"'] & Y_r.
\end{tikzcd}
\label{eq:genuine-power-cartesian-square}
\end{equation}
By Lemma~\ref{lem:root-change-evaluation},
$\widetilde V=Z((e_{p,R}^{\lambda})_{p\in P})$.  In local
trivializations of the evaluation lines, the inclusion of ideals
$(e_{p,R}^{\lambda})_p\subset(e_{p,R})_p$ defines a closed nilpotent
immersion
\[
 h:V_R=Z(e_R)\longrightarrow
 \widetilde V=Z((e_{p,R}^{\lambda})_p),
\]
and $\widetilde f\circ h=\pi_{R,r}^-$.

For a line bundle $L$ with section $s$ on a finite-type stack $T$,
let $h_s:Z(s)\hookrightarrow Z(s^\lambda)$ be the natural closed
immersion.  Write $c_1(L,s)$ for the localized first Chern class,
which sends $A_k(T)_\bbQ$ to $A_{k-1}(Z(s))_\bbQ$
\cite[Sections~2.3 and~2.5]{Fulton98}.
For $\alpha\in A_k(T)_\bbQ$, these classes satisfy
\begin{equation}
 c_1(L^{\otimes\lambda},s^\lambda)(\alpha)
 =\lambda(h_s)_*c_1(L,s)(\alpha),
\label{eq:localized-c1}
\end{equation}
in $A_{k-1}(Z(s^\lambda))_\bbQ$, including on components where $s$
vanishes identically.  This follows from
the divisor identity on the nonzero components and from
$c_1(L^{\otimes\lambda})=\lambda c_1(L)$ on the zero components.
Write
\[
 j^!:A_k(Y_R)_\bbQ\longrightarrow A_{k-m}(\widetilde V)_\bbQ
\]
for the refined operator obtained by pulling back the regular zero section
of $f^*E_r$ along $f^*e_r$.  This notation does not require the
immersion $j$ itself to be regular
\cite[Section~6.2]{Fulton98}.  Applying
\eqref{eq:localized-c1} successively to the $m$ summands gives
\begin{equation}
 j^!(\xi)=\lambda^m h_*\Theta_R^!(\xi).
\label{eq:power-localized-top-chern}
\end{equation}
Proper pushforward commutes with the refined pullback by the regular zero
section in the Cartesian square, so
\[
 \Theta_r^!f_*(\xi)=\widetilde f_*j^!(\xi)
 =\lambda^m(\pi_{R,r}^-)_*\Theta_R^!(\xi).
\]
This proves \eqref{eq:power-gysin-formula}.  Taking
$\xi=[U_R^+]^{\vir}$ and using Proposition
\ref{prop:negative-zero-locus} proves
\eqref{eq:scaled-root-change}.
\end{proof}

\begin{remark}[Origin of the root-order normalization]
Each former puncture contributes $\lambda^{-1}$ because its evaluation
equation pulls back to a $\lambda$th power.  The resulting factor
$\lambda^{-m}$ in \eqref{eq:power-gysin-formula} is a Gysin multiplicity,
distinct from the BNR root-forgetting factor $r^{-m}$.
\end{remark}
\section{Finite-type opens and root-order comparison in genus one}
\label{sec:bounded-root-geometry}

We prove boundedness and properness for the full base-change opens of
\Cref{subsec:theorem-spaces}, then identify the main-component cycle as the
constant coefficient under comparison of root orders.

\subsection{Faithful roots and bounded tropical slopes}

Recall from \Cref{subsec:theorem-spaces} that $Z$ is the ordered set of
markings assigned contact order $0$ in $\Lambda_\Gamma^+$,
that
\[
 F_N:O_{N,Z}^+\longrightarrow O_{N,\varnothing}^+
\]
forgets these markings without stabilising the prestable curve, and that the
line--section pair itself is not changed.  In particular, the twist by the
former punctures in \eqref{eq:canonical-positive-pair} remains part of the
line bundle after the markings have been forgotten.  We now prove the
finite-type and slope assertions used in the definitions of
$W_0$, $W_{r,\varnothing}$, and $W_{r,Z}$.

\begin{lemma}[Finiteness of roots of a section]
\label{lem:finite-section-roots}
Let $S$ be a locally Noetherian algebraic stack, let
$p:\cC\to S$ be a proper flat finitely presented family of geometrically
reduced nodal curves, possibly tame twisted curves, and let $\cL$ be a line
bundle on $\cC$.  Fix an integer $\lambda\geq1$ and a section
$a\in H^0(\cC,\cL^{\otimes\lambda})$.  For an $S$-scheme $T$, subscripts $T$ denote pullback along $T\to S$.  The functor on $S$-schemes
\[
 \operatorname{Root}_{\lambda}(a)(T)
 =\{b\in H^0(\cC_T,\cL_T):b^{\otimes\lambda}=a_T\}
\]
is represented by a morphism finite and representable over $S$; after any
base change $T\to S$ from a scheme, it is represented by a finite
$T$-scheme.  Its formation commutes with arbitrary base change on $S$.
\end{lemma}

\begin{proof}
The assertions are fppf local on $S$, so we may take $S$ affine
Noetherian.  Cohomology and base change for proper tame nodal curves
give, locally on $S$, a two-term complex of finite locally free modules
$E^0\to E^1$ in degrees $0,1$ representing $\mathbf Rp_*\cL$,
compatibly with arbitrary base change.  With
$\mathbf V(E)=\operatorname{Spec}_S\operatorname{Sym}(E^\vee)$,
the section functor of $\cL$ is therefore the affine scheme
\[
 \ker\bigl(\mathbf V(E^0)\longrightarrow\mathbf V(E^1)\bigr).
\]
The same construction applies to $\cL^{\otimes\lambda}$.
Taking the $\lambda$th tensor power is a morphism between these
section schemes, and $\operatorname{Root}_\lambda(a)$ is its fibre
over $a$.  It is thus affine and of finite presentation, and its
formation commutes with arbitrary base change.

It remains to prove properness.  Let $A$ be a discrete valuation ring,
with fraction field $K$ and uniformizer $\varpi$, and pull the family
back along a morphism $\operatorname{Spec}A\to S$.  Suppose a root
$b_K\in H^0(\cC_K,\cL_K)$ is given.  Properness of the curve and
flat base change give
\[
 H:=H^0(\cC_A,\cL_A),\qquad
 H\otimes_A K\simeq H^0(\cC_K,\cL_K),
\]
where $H$ is a finite torsion-free $A$-module.  Choose the least
$m\geq0$ such that $b:=\varpi^m b_K$ belongs to $H$.
Flatness of $\cC_A/A$ implies that multiplication by $\varpi$ is
injective on $\cL_A^{\otimes\lambda}$, so the generic-fibre identity
extends to
\[
 b^{\otimes\lambda}=\varpi^{m\lambda}a
 \quad\text{on }\cC_A.
\]
If $m>0$, restriction to the reduced special fibre gives
$(b|_{\cC_s})^{\otimes\lambda}=0$, hence $b|_{\cC_s}=0$.
The exact sequence
\[
 0\longrightarrow\cL_A
 \xrightarrow{\ \varpi\ }\cL_A
 \longrightarrow\cL_s\longrightarrow0
\]
then gives $b\in\varpi H$, contradicting the minimality of $m$.
Consequently $m=0$, and $b_K$ extends to a root on $\cC_A$.
This includes the case $b_K=0$.  Any two extensions agree because
$H$ is torsion-free.

The discrete-valuation-ring criterion for the finite-presentation
morphism over the locally Noetherian base now proves properness.
An affine proper morphism is finite.  The scheme constructions and
their canonical base-change identifications descend over $S$, proving
representability and all the asserted conclusions.
\end{proof}

\begin{proposition}[Faithful roots on twisted curves]
\label{prop:faithful-root-properness}
Let $\mathfrak B$ be the stack of marked prestable curves with a line
bundle, with universal marked curve $C\to\mathfrak B$ and universal line
bundle $A$ on $C$.
Fix $N\geq1$ and fixed marking-character sectors.  Let $\mathfrak R_N$
classify data
\begin{equation}
 (q:\mathcal C\to C,\ L,\ \phi:L^{\otimes N}\xrightarrow{\sim}q^*A),
\label{eq:faithful-root-data}
\end{equation}
where $\mathcal C$ is a balanced tame twisted prestable curve with coarse
curve $C$, stack structure occurs only at nodes and markings, and every
geometric stabilizer acts faithfully on $L$.  Then
$\mathfrak R_N\to\mathfrak B$ is algebraic, locally of finite presentation,
proper, separated, quasi-finite, and of Deligne--Mumford type.  The
construction commutes with arbitrary base change.

For positive integers $r,\lambda$ and $R=\lambda r$, choose the
order-$r$ marking-character sectors to be those induced from the order-$R$
sectors by $\lambda$th tensor power and removal of the character kernels.
Tensor power followed by the AOV relative coarse space for the kernel of
the induced root character then defines a canonical morphism
\begin{equation}
 \rho_{R,r}:\mathfrak R_R\longrightarrow\mathfrak R_r                 
\label{eq:faithful-root-comparison}
\end{equation}
which is proper, quasi-finite, of Deligne--Mumford type, and compatible with
base change.  More precisely, let $t$ be the stabilizer order at a node,
let $w\in\mathbb Z/t\mathbb Z$ be the character to be made faithful, and
put $h=\gcd(t,\widehat w)$ for any integral representative $\widehat w$ of
$w$; this is independent of the representative, and $\gcd(t,0)=t$.
For a chart base ring $B$ and twisted smoothing parameter $\tau\in B$,
the relative coarsening on a balanced node chart is
\begin{equation}
 \left[\Spec B[x,y]/(xy-\tau)\big/\mu_t\right]
 \longrightarrow
 \left[\Spec B[X,Y]/(XY-\tau^h)\big/\mu_{t/h}\right],
 \qquad X=x^h,\quad Y=y^h.
 \label{eq:faithful-root-local-chart}
\end{equation}
Both actions are balanced, and the group homomorphism is
$\zeta\mapsto\zeta^h$.  For the $\lambda$th power of a faithful character,
$h=\gcd(t,\lambda)$.  The coarse smoothing parameter $\tau^t$ is unchanged.
This is a chart of the source curve; a one-variable root chart describes
only the induced change of smoothing parameters on the base.
The induced map on relative Isom spaces need not be an equivalence: the
kernel of the stabilizer character is precisely the inertia removed by
relative coarsening.
\end{proposition}

\begin{proof}
The assertions are smooth local on $\mathfrak B$, so take a scheme atlas
$S\to\mathfrak B$.  Form the tame root gerbe
\[
 \mathcal G_N=\sqrt[N]{A/C_S}\longrightarrow C_S.
\]
A representable twisted map $\mathcal C\to\mathcal G_N$ whose induced
coarse map is the identity of $C_S$ is precisely the datum
\eqref{eq:faithful-root-data}.  Such a map is stable: its coarse map is the identity, so no component
is contracted; its relative automorphisms consist of finite curve ghosts
and finite root scalars.  The AOV
morphism from twisted stable maps to $\mathcal G_N$ to stable maps to its
coarse target $C_S$ is proper and quasi-finite
\cite[Theorems~3.1 and~4.3]{AOV}.  Its base change along the family of
identity maps $C_S\to C_S$ is precisely the required root stack over $S$.
The prescribed marking-character sectors are open and closed.  This proves
algebraicity, local finite presentation, properness, separatedness, and
Deligne--Mumford type after the atlas, hence over $\mathfrak B$.  No
closedness assertion about the identity-map section is needed.

For a fixed coarse curve, the remaining choices are roots of a fixed line
bundle, node and marking stabilizer orders dividing $N$, and finite ghost
data.  The root choices form a torsor under the finite $N$-torsion of the
generalized Jacobian.  Thus the morphism is quasi-finite.

If a possibly nonrepresentable full root is given, the kernel of the
stabilizer character on $L$ is the relative inertia of the map to
$\mathcal G_N$.  The AOV relative coarse-space construction quotients the
source by this kernel, is functorial on arrows, and commutes with arbitrary
base change in the tame setting
\cite[Theorem~3.1 and Proposition~3.4]{AOV}.  On the balanced node chart, the
invariant algebra is
\[
 \bigl(B[x,y]/(xy-\tau)\bigr)^{\mu_h}
   =B[x^h,y^h]/(x^hy^h-\tau^h).
\]
Indeed, every invariant monomial reduces, using $xy=\tau$, to a monomial
in $x^h$, $y^h$, and $\tau$.  The residual action of $\mu_{t/h}$ has weights
$(1,-1)$.  This proves \eqref{eq:faithful-root-local-chart}, compatibly
with branch interchange and gluing.  Markings have the analogous
one-branch root chart.

For $R=\lambda r$, first take the $\lambda$th tensor power of the root
line.  The resulting character may acquire a kernel; quotienting by that
kernel gives \eqref{eq:faithful-root-comparison}.  It is a morphism over the
same marked Picard stack.  Since $\mathfrak R_R$ is proper and
$\mathfrak R_r$ is separated over that stack, its graph factorization proves
properness.  Once the finitely many high node sectors are fixed, the fibre
is a torsor under finite $\lambda$-torsion in a generalized Jacobian, with
finite ghost automorphisms.  Hence it is quasi-finite and of
Deligne--Mumford type.  Every part of the construction is canonical under
base change.
\end{proof}

\begin{remark}[Why the root-order morphism is not merely tensor power]
Tensor power can create a kernel in the stabilizer character.  AOV
relative coarsening removes this kernel and restores representability;
the comparison therefore need not preserve relative Isom sheaves.
\end{remark}

Recall the universal pair $(M_\Gamma^+,t_\Gamma^+)$, the morphism
$g_\Gamma^+$, the integer $B_0$, the substack $W_0$, the admissible base
root order $r$, and $W_{r,\varnothing}$ from
\Cref{subsec:theorem-spaces}.  We retain the full $2$-base change along $W_0\hookrightarrow\mathfrak U^+_\varnothing$.

\begin{lemma}[The finite-type substack and the uniform slope bound]
\label{lem:uniform-slope-bound}
The maxima in \eqref{eq:B0-definition} are finite, and the substack $W_0$
defined in \eqref{eq:W0-definition} is a finite-type open substack of
$\mathfrak U^+_\varnothing$.  It contains the image of $g_\Gamma^+$ and the
smooth one-vertex locus with nonzero section.  For every geometric point $(C,M,t)$ of
$W_0$ and every subset $S$ of the vertices of its dual graph, write $C_v$
for the coarse component corresponding to $v$.  Then
\begin{equation}
 \left|\sum_{v\in S}\deg(M|_{C_v})\right|\leq B_0.
\label{eq:B0-bound}
\end{equation}
For every base root order $r$ satisfying \eqref{eq:large-r}, the stack
$W_{r,\varnothing}$ defined in \eqref{eq:bounded-rooted-opens} is finite type.
Every realizable integral tropical type of a positive logarithmic
enhancement over $W_0$, and hence over $W_{r,\varnothing}$, has
coarse slope $m_e$ on each bounded edge satisfying
\begin{equation}
 |m_e|\leq4B_0<r.
\label{eq:slope-bound}
\end{equation}
The same estimate holds for a type appearing only in the special fibre of a
trait in $W_{r,\varnothing}$.
\end{lemma}

\begin{proof}
\proofstep{Step 1: finiteness of $B_0$ and the definition of $W_0$}
The stack $B_\Gamma$ is of finite type.  After passing to a finite-type atlas,
the universal nodal curve admits a finite constructible stratification on
which the dual graph and the multidegree of $M_\Gamma^+$ are constant.  Hence
only finitely many graphs, multidegree vectors, and subset sums occur, proving
that the maxima in \eqref{eq:B0-definition} are finite.

The conditions in \eqref{eq:W0-definition} are locally constructible.  They are
preserved under generisation: nodes may be smoothed, hence edges are
contracted, and the degree on a component after contraction is the sum of the
degrees on the components that merge.  Because the bound is imposed for every
union of components, it remains valid after such contractions.  Thus the
selected locally constructible subset is stable under generisation and defines an
open substack.

The same bounds imply finite type.  The number of edges, and therefore the
number of vertices, is bounded; the genus and number of markings are fixed;
and taking $A=\{v\}$ bounds every component degree.  Consequently only
finitely many dual graphs, marking distributions, and multidegrees occur.
Each corresponding line--section stratum is of finite type, so their finite
union $W_0$ is finite type.  By construction,
$g_\Gamma^+(B_\Gamma)\subset W_0$.  A smooth one-vertex curve has no edges
and its only component has degree $d^+\leq B_0$, so the smooth one-vertex
locus with nonzero section also lies in $W_0$.

\proofstep{Step 2: finite type of the rooted stack}
Over $W_0$, the stack of $r$-twisted curves and faithful $r$th roots of the
coarse line is of finite type by
\Cref{prop:faithful-root-properness}.  The prescribed marking
character sectors are open and closed, representability and the inequalities
\eqref{eq:crumplin-degree-requirement} are open, and, once the root line is
fixed, roots of the coarse section are finite by
\Cref{lem:finite-section-roots}.  Hence $W_{r,\varnothing}$ is finite
type.  Equip it with the pullback of the canonical logarithmic structure
on the stack of twisted prestable curves.  With this convention, the
projection retaining the twisted curve is strict.  This does not assert
strictness of the projection to the coarse-curve base $W_0$: at a node of
index $s_e$, the canonical smoothing parameters induce the characteristic
map $\mathbb N\to\mathbb N$, $1\mapsto s_e$.

\proofstep{Step 3: the incidence bound}
Fix a positive logarithmic enhancement of a geometric point of $W_0$,
and let $G$ be its dual graph.  For $v\in V(G)$, let $C_v$ be the
corresponding coarse component.  Let $d_v=\deg(M|_{C_v})$ be the degree of
its coarse line bundle $M$ on that component,
and let $c_v$ be the sum of retained marking contacts on that component,
and put $b_v=d_v-c_v$.  Since both the total coarse degree and the total
marking contact equal $d^+$,
\begin{equation}
 \sum_{v\in V(G)}b_v=0.
\label{eq:b-total-zero}
\end{equation}
For every vertex subset $S$,
\begin{equation}
 \left|\sum_{v\in S}b_v\right|
 \leq\left|\sum_{v\in S}d_v\right|+
       \sum_{v\in S}c_v
 \leq B_0+d^+\leq2B_0.
\label{eq:b-bound}
\end{equation}
Orient every bounded edge $e$, and write $\mathrm{tail}(e)$ and
$\mathrm{head}(e)$ for its initial and terminal vertices.  Put
$\boldsymbol m=(m_e)_{e\in E(G)}$ and $\boldsymbol b=(b_v)_{v\in V(G)}$.
We use the outgoing-minus-incoming incidence convention
\[
 \partial:\mathbb Z^{E(G)}\longrightarrow\mathbb Z^{V(G)},\qquad
 (\partial\boldsymbol m)_v
 =\sum_{\mathrm{tail}(e)=v}m_e-\sum_{\mathrm{head}(e)=v}m_e.
\]
Thus the integral slopes of every realizable tropical type satisfy
$\partial\boldsymbol m=\boldsymbol b$.

\proofstep{Step 4: a level-set estimate}
Let $Q$ be the sharp fine and saturated characteristic monoid of the
base logarithmic point.  Write $e_v\in Q$ for the elevation of a
vertex $v$ and $\delta_e\in Q$ for the smoothing element of a
bounded edge.  After contracting any zero-length edges, choose a
homomorphism $h:Q\to\mathbb R_{\geq0}$ strictly positive on every
nonzero element of $Q$.  Such a homomorphism exists because $Q$ is
sharp, fine, and saturated.  Put
\[
 H_v=h(e_v),\qquad \ell_e=h(\delta_e)>0.
\]
Realizability gives, for every oriented edge,
\[
 H_{\mathrm{head}(e)}-H_{\mathrm{tail}(e)}=m_e\ell_e.
\]
If $m_e=0$ there is nothing to prove.  Otherwise, reverse its
orientation if necessary so that $m_e>0$, and choose a real number
$a$ strictly between its endpoint heights and different from every
vertex height.  Let $S=\{v:H_v<a\}$.  Orient all edges crossing
this cut from $S$ to its complement.  Their slopes are positive by
the displayed relation.  Summing the incidence equations over $S$
therefore gives
\[
 0<m_e\leq\sum_{f\text{ crossing the cut}}m_f
   =\sum_{v\in S}b_v\leq2B_0,
\]
where the final bound is \eqref{eq:b-bound}.  Thus every bounded
edge satisfies the stronger estimate $|m_e|\leq2B_0$.
A loop has equal endpoint heights and hence slope zero.

\proofstep{Step 5: special fibres and the root order}
The estimate is fibrewise and uses only the degree bounds and
realizability.  It consequently applies also to edges and graphs
appearing only in the special fibre of a trait in
$W_{r,\varnothing}$.  For a rooted enhancement the descended line
has the same coarse slopes.  In particular,
\[
 |m_e|\leq2B_0\leq4B_0<r,
\]
which proves \eqref{eq:slope-bound} with the previously chosen root
order.  The remaining marking divisibility and age requirements
follow from \eqref{eq:large-r}; the additional hypotheses used for
generic component degrees are checked in
\Cref{prop:compatible-support}.

\end{proof}

\begin{remark}[The role of realizability in the slope bound]
The level-set argument uses the relation between slopes, lengths,
and vertex elevations of a realizable tropical map.  The incidence
equations alone do not bound circulations around circuits.  For a
bounded family satisfying the same degree and nonnegative-contact
bounds, the level-set estimate applies in any genus.  Arithmetic
genus one is used later in the classification and virtual-class
calculation for the essential component types.
\end{remark}

We need a proper comparison between different root orders on these open substacks.
The following statement is formulated on ambient root stacks so that no
map between the complete substacks satisfying \eqref{eq:crumplin-degree-requirement} at the two root orders is assumed.

\subsection{Properness of the comparison between root orders}

The previous lemma fixes the root order only after the relevant geometric
support and every bounded-edge slope of a realizable enhancement has been bounded.  The
following lemma makes the resulting comparison morphism proper.

\begin{lemma}[Properness after base change to a finite-type open]
\label{lem:proper-root-comparison}
Let $r,\lambda$ be positive integers and put $R=\lambda r$.
Fix the retained positive contact orders $c_i$, with $0<c_i<r$ and
$c_i\mid r$.  For $N=r,R$, give the corresponding positive marking source
index $N/c_i$ and faithful character one; markings of contact order $0$
remain untwisted.  Tensor power
and AOV relative coarse space by the kernel of the root character define
\[
 \overline\pi_{R,r}^+:\overline O_R^+\longrightarrow\overline O_r^+.
\]
For every finite-type open $W\subset O_r^+$, form the $2$-Cartesian square
\begin{equation}
\begin{tikzcd}
 W_R\ar[r]\ar[d,"\pi_{R,r,W}^+"']
   & \overline O_R^+\ar[d,"\overline\pi_{R,r}^+"]\\
 W\ar[r,hook]
   & \overline O_r^+ .
\end{tikzcd}
\label{eq:proper-root-open-base-change}
\end{equation}
Then $W_R$ is finite type, it is contained in $O_R^+$, and
$\pi_{R,r,W}^+$ is proper, quasi-finite, and of Deligne--Mumford type.  No
stabilization of the source curve is involved.
\end{lemma}

\begin{proof}
First forget the section.  Proposition
\ref{prop:faithful-root-properness} gives algebraic separated proper stacks
$\mathfrak R_R$ and $\mathfrak R_r$ over the same marked Picard stack and the
proper comparison morphism
$\rho_{R,r}:\mathfrak R_R\to\mathfrak R_r$ obtained by tensor power and AOV
relative coarsening.

Its quasi-finiteness and Deligne--Mumford type are also part of that
proposition; no new genus-one argument is needed at the line-bundle level.

On a test scheme $S$, let $\kappa_{R,r}:\mathcal C_R\to\mathcal C_r$
be the partial coarsening of the two twisted curves, and let
$L_R^+,L_r^+$ be the root lines.  Their specified root identifications give
$(L_R^+)^{\otimes\lambda}\simeq\kappa_{R,r}^*L_r^+$.
For a given section $\sigma_r^+$ of $L_r^+$, reintroduce the high-order
root section by imposing
\[
 (\sigma_R^+)^{\otimes\lambda}=\kappa_{R,r}^*\sigma_r^+.
\]
Lemma~\ref{lem:finite-section-roots} shows that this is finite over the
stack of high root lines, so the comparison with sections remains proper
and quasi-finite.  For every proper subcurve $E$ of the common coarse
curve, let $\mathcal E_R\subset\mathcal C_R$ and
$\mathcal E_r\subset\mathcal C_r$ be the corresponding reduced subcurves.
Then
\[
 \lambda\deg(L_R^+|_{\mathcal E_R})=\deg(L_r^+|_{\mathcal E_r}).
\]
Hence the upper-left corner of the $2$-base change of the substack satisfying
\eqref{eq:crumplin-degree-requirement} at order $r$ is contained in the
corresponding substack at order $R$.  The morphism
$\pi_{R,r,W}^+$ in \eqref{eq:proper-root-open-base-change} is a base change of
the preceding proper morphism and therefore has all the asserted properties.
\end{proof}

The opens $W_{N,\varnothing}$ and $W_{N,Z}$ are the $2$-base changes defined by the diagrams in \Cref{subsec:theorem-spaces}.  The preceding lemma
proves the required properness, and \Cref{lem:uniform-slope-bound} proves
finite type.  We now check all their components, rather than selecting
components meeting a particular geometric image.

\subsection{Genus-one components and the constant coefficient}

The special feature used here is not polynomiality alone.  In genus one,
Crumplin identifies the positive universal virtual class with the ordinary
fundamental class.  Thus, once properness has been established, it suffices
to compare the generic degrees of the irreducible components.

\begin{lemma}[Essential types and the virtual class in genus one]
\label{lem:genus-one-essential-types}
For Crumplin's nonnegative contact data with $\ell_{\mathrm{cur}}$ markings,
the irreducible components are the reduced closures $Z_{[\tau]}$ of the
distinguished strata indexed by inducible essential mod-$N$ types
$[\tau]$.  In genus one, such a type is either the trivial type or has a
tree-shaped graph with one internal vertex of genus one and rational
external leaves.  The latter alternative includes the one-vertex, no-edge
type whose unique vertex is internal; this is not the trivial type, whose
unique vertex is external.  Here an internal vertex represents a component
on which the section is identically zero, and an external vertex represents
a component on which it is nonzero.  Every irreducible component has
dimension $\ell_{\mathrm{cur}}$, the universal virtual dimension.  Here
$\ell_{\mathrm{cur}}=n+\rho$ when all labels are retained, and
$\ell_{\mathrm{cur}}=\rho_+$ after forgetting $Z$.  On the finite-type opens
$W_{N,Z}$ used here, with $\ell_{\mathrm{cur}}=n+\rho>0$ and $N=\lambda r$
as in \eqref{eq:large-r}, the universal virtual class is the sum of the
reduced fundamental cycles of these irreducible components, with
multiplicity one.  The component and dimension assertions also apply when
$\ell_{\mathrm{cur}}=0$; no Chow assertion on a stack with non-affine
stabilizers is intended.
\end{lemma}

\begin{proof}
\proofstep{Step 1: components and dimensions}
By definition, an essential type is bipartite and every internal vertex
has positive genus \cite[Definition~2.21]{Crumplin}.  The inducible
essential types index the irreducible components
\cite[Theorem~2.24]{Crumplin}.

Write $G$ for the underlying graph and $V_+(G)$ for its internal
vertices.  The genus identity
\[
 1=b_1(G)+\sum_{v\in V(G)}g_v
\]
shows that, if there is an internal vertex, it is the unique vertex of
positive genus and has genus one.  Moreover, $b_1(G)=0$, so $G$ is a
tree.  Bipartiteness then forces every other vertex to be an external
rational leaf attached to the internal vertex.  This includes the
one-vertex internal type.  If there is no internal vertex,
bipartiteness excludes every edge, and connectedness gives the trivial
type.

For an essential type, Crumplin's dimension formula is
\[
 \dim Z_{[\tau]}
 =
 3g-3+\ell_{\mathrm{cur}}
 +\sum_{v\in V_+(G)}(g_v-1),
\]
where $g$ is the total arithmetic genus
\cite[Corollary~4.3]{Crumplin}.  The final sum vanishes in both
genus-one alternatives, giving
$\dim Z_{[\tau]}=\ell_{\mathrm{cur}}$.
These arguments also apply when $\ell_{\mathrm{cur}}=0$.

\proofstep{Step 2: the obstruction theory for sections}
For the virtual-cycle assertion, work on $W_{N,Z}$ with
$\ell_{\mathrm{cur}}>0$.  We use the standard relative obstruction
theory
\[
 \mathbb E=(\mathbf R\pi_*L)^\vee
\]
for the morphism forgetting the section and retaining the twisted curve
and root line \cite[Section~4.2]{Crumplin}.

The Picard stack of twisted curves with the prescribed total degree
and marking characters is smooth.  Here the curves and their nodes
are allowed to deform; we do not restrict to a fixed nodal stratum.
Let $S$ be a smooth atlas of this Picard stack, and let
$\pi:\mathcal C_S\to S$ and $L$ be the universal twisted curve and
root line.  Locally on $S$, represent $\mathbf R\pi_*L$ by a complex
of vector bundles
\[
 [E^0\xrightarrow{d}E^1]
\]
in degrees $0,1$.  If $\eta$ denotes the tautological section on
$\mathbf V(E^0)$, the section stack is the zero locus of $d(\eta)$
in the smooth ambient space $\mathbf V(E^0)$.  The pullback of
$W_{N,Z}$ is an open subspace of this zero locus.

The Picard stack has dimension $\ell_{\mathrm{cur}}$, and
orbifold Riemann--Roch gives
\[
 \chi(\mathcal C,L)
 =
 \deg L-\sum_p\operatorname{age}_p(L)=0.
\]
Consequently, after accounting for the relative dimension of the
atlas, Step~1 shows that every component of this zero locus has
codimension $\operatorname{rank}E^1$.  Its ideal is generated by
that many equations in a regular ambient local ring, so these
equations form a regular sequence.  The zero locus is therefore a
local complete intersection, and the standard obstruction theory
of sections gives its ordinary fundamental cycle.  Thus
\[
 [W_{N,Z}]^{\vir}=[W_{N,Z}],
\]
with the scheme-theoretic generic multiplicities on the right.
We determine those multiplicities next.

\proofstep{Step 3: generic multiplicity one}
On the distinguished main locus, the curve is smooth and the
nonzero section determines its line bundle from the prescribed
marking divisor.  This locus is therefore smooth, so the main
component has generic multiplicity one.

Consider a non-main component meeting $W_{N,Z}$ and a general point
$(\mathcal C,L,\sigma)$ of its distinguished stratum.  Its graph has
one internal elliptic vertex and $k$ external rational leaves.
For an edge $e$ directed from a leaf $C_v$ towards the elliptic
component, the leaf-section description gives
\[
 m_e=d_v-\sum_{p_i\in C_v}c_i,
 \qquad 0<m_e\leq B_0.
\]
Here $d_v$ is the degree of the descended $N$th-power line on $C_v$.
The positivity follows because the section on the leaf must vanish
at its attaching node.  Our choice of $r$ implies $m_e\mid N$.
Representability therefore gives stabilizer order
$s_e=N/m_e>1$ and external gerby contact one.

Let $p_E:\mathcal E\to E$ be the coarse map from the internal elliptic
normalization, and write $\widetilde q_1,\ldots,\widetilde q_k$ for
its attaching gerbes, with coarse points $q_1,\ldots,q_k$.
The line-bundle description of
\cite[Proposition~2.13]{Crumplin} gives
\[
 L|_{\mathcal E}
 \simeq
 p_E^*J\otimes
 \mathcal O_{\mathcal E}
 \left(
   \sum_p\widetilde c_p\,\widetilde p
   -\sum_{i=1}^k\widetilde q_i
 \right),
 \qquad J\in\operatorname{Pic}^0(E).
\]
The first sum runs over the markings on $\mathcal E$, and
$\widetilde c_p$ denotes their gerby contact.
When $k=0$, the same description follows directly from the degree
and marking-character conditions.

Each marking coefficient satisfies
$0\leq\widetilde c_p<s_p$, where $s_p$ is its stabilizer order.
Thus the marking factors contribute no integral twist after
coarse pushforward, whereas each coefficient $-1$ at an attaching
gerbe contributes $-q_i$.  Setting $Q=q_1+\cdots+q_k$, we obtain
\[
 (p_E)_*(L|_{\mathcal E})\simeq J(-Q).
\]
The Jacobian factor varies freely in the universal stratum, and the
bounded open imposes only numerical conditions.  We may therefore
take $J$ nontrivial at a general point.

On each rational leaf, the coarse pushforward of the root line is
$\mathcal O_{\mathbf P^1}$, so its first cohomology vanishes.
At every attaching node, the root line has nontrivial stabilizer
character.  Its restriction to the node gerbe consequently has
zero cohomology, since the stabilizer is tame.  The normalization
sequence therefore identifies the full section-obstruction space as
\[
 H^1(\mathcal C,L)\simeq H^1(E,J(-Q)).
\]

Suppose first that $k>0$.  Consider the connecting homomorphism
\begin{equation}
 \delta:\bigoplus_{i=1}^kJ|_{q_i}
 \longrightarrow H^1(E,J(-Q))
 \quad\text{associated to}\quad
 0\longrightarrow J(-Q)\longrightarrow J
 \longrightarrow J|_Q\longrightarrow0.
 \label{eq:generic-star-obstruction}
\end{equation}
We identify this with the obstruction map for smoothing the attaching
nodes.  At the $i$th node, put $s_i=s_{e_i}$ and choose a balanced
local smoothing
\[
 xy=t_i,
\]
with $x$ on the external branch and $y$ on the elliptic branch.
Because the external gerby contact is one, the external section
has leading term $a_i x$ with $a_i\ne0$.  On the overlap with the
elliptic branch, this becomes $a_it_i/y$.

In a compatible equivariant frame $e_i$, a regular generator of
the coarse pushforward on the elliptic branch is
$y^{s_i-1}e_i$.  Writing $z_i=y^{s_i}$ for the coarse coordinate,
we have
\[
 \frac{a_it_i}{y}e_i
 =
 \frac{a_it_i}{z_i}\bigl(y^{s_i-1}e_i\bigr).
\]
This is a simple principal part of $J(-Q)$, representing an element
of $J|_{q_i}$.  Its Cech obstruction is the corresponding summand
of \eqref{eq:generic-star-obstruction}.  Hence, after choosing these
local trivializations, the node-smoothing obstruction map is
$\delta$ composed with invertible diagonal scalings.

Since $J$ is a nontrivial degree-zero line bundle on a smooth
elliptic curve,
\[
 H^0(E,J)=H^1(E,J)=0.
\]
The connecting map $\delta$ is therefore an isomorphism.
The node-smoothing directions consequently surject onto the full
section-obstruction space.  Together with the section directions,
which account for the image of $E^0\to E^1$, they make the
differential of the local defining section surjective.
Thus the section stack is smooth at this general point.

When $k=0$, we instead have
$H^1(\mathcal C,L)=H^1(E,J)=0$, which gives generic smoothness
directly.  Additional contact-zero markings affect neither the
cohomology calculation nor the local smoothing calculation.

Every component of $W_{N,Z}$ therefore has generic multiplicity
one.  Combining this with Step~2 yields
\[
 [W_{N,Z}]^{\vir}
 =
 \sum_{\substack{[\tau]\ \mathrm{inducible\ essential}\\
                  Z_{[\tau]}\cap W_{N,Z}\ne\varnothing}}
 [Z_{[\tau]}\cap W_{N,Z}],
\]
where each intersection carries its reduced structure.
\end{proof}

\begin{remark}[Scope of the Crumplin inputs]
We use Crumplin's component classification and the preceding genus-one
virtual-class identity, together with his generic-degree formula
\cite[Theorems~2.24, 3.2, and~3.25]{Crumplin}.  In genus one, the graph
underlying every essential type is a tree, so Lemma~2.8 of that paper
determines its weighting uniquely;
the general lifting argument of his Lemma~3.20 is not needed here.
The fixed integral edge slopes used below give compatible residues at all
root orders, or equivalently the special case of his
$\widehat{\bbZ}$-tropical types obtained from
$\bbZ\to\widehat{\bbZ}=\varprojlim_N\bbZ/N\bbZ$.
We still verify properness and the complete list of components on the
finite-type opens defined by $2$-base change.  These are not consequences of a
generic-degree formula, and we do not apply Crumplin's complete-space
Theorem~4.6 to independently chosen opens.
\end{remark}

\begin{proposition}[Genus-one components and the constant coefficient]
\label{prop:compatible-support}
Assume $n+\rho>0$, as in the standing convention for universal Chow
calculations.  Fix $r$ as in \eqref{eq:large-r}; throughout this proposition,
$\lambda\geq1$ is an integer and $R=\lambda r$.
There is a finite set $I$ of graphs with the vertex genera, degrees,
marking contacts, image faces, and integral edge slopes fixed, containing
the trivial type $\tau_{\mathrm{triv}}$ exactly once.  Write $[\tau_N]$ for
the mod-$N$ type obtained by reducing those slopes and taking the
stabilizer orders prescribed by representability.  Every $[\tau_N]$ is
inducible and essential for $N=\lambda r$.
Here $Z_{[\tau_N]}$ denotes the reduced closure of its distinguished
stratum in the space $O_{N,\varnothing}^+$ with $Z$ forgotten.
The following assertions hold.  Put
\begin{equation}
 Z_{r,\varnothing}^{\main}
 :=\bigl(Z_{\main,r,\varnothing}
     \mathop{\times}_{O_{r,\varnothing}^+}W_{r,\varnothing}\bigr)_{\red},
 \label{eq:zero-free-main-component}
\end{equation}
where $Z_{\main,r,\varnothing}$ is Crumplin's main component after the
markings in $Z$ have been forgotten.
\begin{enumerate}[label=\textup{(\roman*)}]
\item $W_{R,\varnothing}$ is finite type and its comparison morphism to
$W_{r,\varnothing}$ is proper, quasi-finite, and of Deligne--Mumford type.
\item Its irreducible components are exactly the nonempty restrictions
$Z_{[\tau_R]}\mathop{\times}_{O_{R,\varnothing}^+}W_{R,\varnothing}$, for $\tau\in I$.
Every component is taken with its full closure inside the open; boundary
specializations and intersections of components are not removed.
\item The forgetful morphisms $F_N$ are smooth and faithfully flat.  The
square comparing the two root orders and forgetting $Z$ is Cartesian,
and the universal section obstruction theory pulls back with this square.
For $N=r,R$,
\begin{equation}
 [W_{N,Z}]^{\vir}
   =\sum_{T\in\operatorname{Irr}(W_{N,Z})}[T]
 \quad\text{in }A_{n+\rho}(W_{N,Z})_{\bbQ},
 \label{eq:genus-one-decomposition}
\end{equation}
where each component has its reduced structure, and
\begin{equation}
 Z_{W_r}^{\main}
 =\bigl(W_{r,Z}\mathop{\times}_{W_{r,\varnothing}}
        Z_{r,\varnothing}^{\main}\bigr)_{\red}.
 \label{eq:zero-order-marking-pullbacks}
\end{equation}
Equivalently, this is the reduced upper-left corner of the $2$-Cartesian
marking-restoration square restricted to the main component.
\item The function $Q_Z$ defined in \eqref{eq:QZ} satisfies
\begin{equation}
 Q_Z(\lambda)
 =[Z_{W_r}^{\main}]
   +\lambda\bigl([W_{r,Z}]^{\vir}-[Z_{W_r}^{\main}]\bigr).
 \label{eq:explicit-genus-one-polynomial}
\end{equation}
In particular it has degree at most one, and
\begin{equation}
 \CT_\lambda Q_Z(\lambda)=[Z_{W_r}^{\main}].
 \label{eq:crumplin-constant}
\end{equation}
\end{enumerate}
\end{proposition}

\begin{proof}
\proofstep{Step 1: integer slopes on the generic component types}
By \Cref{lem:genus-one-essential-types}, every non-main component
of $W_{r,\varnothing}$ is indexed by an inducible essential type whose
graph is an elliptic-centred star.  Orient its edges from the rational
leaves toward the elliptic vertex.  If $v$ is a leaf and $e$
is its unique edge, the nonzero section on $C_v$ has its prescribed zeros
at the markings and a positive zero at the joining node.  Crumplin's
line-bundle description therefore gives
\begin{equation}
 m_e=d_v-\sum_{p_i\in C_v}c_i>0,
 \qquad d_v=m_e+\sum_{p_i\in C_v}c_i.
 \label{eq:external-balance}
\end{equation}
Here \cite[Proposition~2.13(1)]{Crumplin} excludes further zeros: the
age-zero part of the line has degree zero and its nonzero section
trivializes it.  Since the generic object lies over $W_0$,
\begin{equation}
 1\leq m_e\leq d_v\leq B_0,
 \label{eq:star-slope-bound}
\end{equation}
and hence $m_e\mid r$.  The size expression in Crumplin's degree theorem
simplifies, for a leaf, to
\begin{equation}
 d_v-m_e+\sum_{p_i\in C_v}c_i
 =2\sum_{p_i\in C_v}c_i\leq2B_0<r.
 \label{eq:crumplin-large-bound}
\end{equation}
Thus all its size, divisibility, and nonzero nodal-contact requirements
hold for every component over $W_0$, not only for components meeting the
geometric image.  This is why no component-selection open is needed.
The nodal requirements are vacuous for a type with no edges.  Summing the leaf
identities and using the total degree fixes the internal degree as
$d_0=\sum_{p_i\in C_0}c_i-\sum_e m_e$, where $C_0$ is the coarse elliptic
component and $d_0$ is its coarse line-bundle degree.

There is no independent weighting to choose: on a tree, cutting an edge
and summing the balancing equations determines its slope modulo every $N$.
Consequently the above integers, together with their negatives on the
opposite orientations, give the unique compatible mod-$N$ weighting for
this graph and these degrees \cite[Lemma~2.8]{Crumplin}.  This uniqueness
concerns the combinatorial weighting, not the possible line bundles on
the elliptic component.

\proofstep{Step 2: existence of the lifts and exhaustion of components}
A compatible weighting must also be inducible.  At $N=\lambda r$, take
marking indices $N/c_i$ for the retained positive markings and node indices
$N/m_e$; the contact-zero labels in $Z$ are absent at this stage.
For each vertex $v$, let $\mathcal C_v$ be the corresponding normalized
twisted component and let $q_v:\mathcal C_v\to C_v$ be its coarse map.
Write $\widetilde p_i$ for the reduced marking gerbe and
$\widetilde q_{e,v}$ for the reduced gerbe of the branch of node $e$ on
$\mathcal C_v$, regarded as Cartier divisors on that normalization.
On an external leaf $v$ with its unique edge $e$, use
$\mathcal O_{\mathcal C_v}(\sum_{p_i\in C_v}\widetilde p_i+
\widetilde q_{e,v})$ with its canonical section; on the internal elliptic
component use
\[
 \mathcal O_{\mathcal C_0}\!\left(\sum_{p_i\in C_0}\widetilde p_i
                       -\sum_e\widetilde q_{e,0}\right)\otimes q_0^*J_N
\]
with zero section and $J_N\in\operatorname{Pic}^0(C_0)$ nontrivial.
The balanced node characters glue the lines, and the external sections
vanish at the nodes and glue to zero.  The marking and node characters
are faithful.  The $N$th power descends with exactly the required
component degrees, and every union of components has rooted degree of
absolute value at most $B_0/N<1/2$.

This construction also lifts a generic order-$r$ object, rather than
merely some object of its type.  Let $J_r$ be the nontrivial degree-zero
factor of that object's root line on $C_0$.  Choose $J_N$ with
$J_N^{\otimes\lambda}\simeq J_r$, using surjectivity of multiplication by
$\lambda$ on the elliptic Picard variety.  The remaining section scalars
and fibre identifications admit $\lambda$th roots after a finite base
extension.  Since the graph is a tree there is no gluing equation around
a circuit.  Tensor power and partial coarsening then recover the given
order-$r$ object.  If $J_r$ is nontrivial, so is $J_N$.
The trivial type lifts by the canonical divisor of the markings on a
smooth source, and its generic comparison has degree one
\cite[Example~3.27(1)]{Crumplin}.

Let $I$ consist of the integral data just obtained from the finitely many
components of $W_{r,\varnothing}$.  Properness and quasi-finiteness are
\Cref{lem:proper-root-comparison}.  Every component of
$W_{R,\varnothing}$ has dimension $\rho_+$ by
\Cref{lem:genus-one-essential-types}; the same is true at order $r$.
A proper quasi-finite morphism preserves the dimension of an irreducible
component, so its closed image must be a whole component of
$W_{r,\varnothing}$, not a lower-dimensional boundary locus.
The comparison preserves the generic graph, genera, degrees, and
vanishing components \cite[Lemma~3.6]{Crumplin}.  The tree-weighting
uniqueness therefore shows that the source component is indexed by $I$.
Conversely the generic lifts above show that every type in $I$ occurs on a source component lying over, and dominating, the corresponding base component.
This proves the complete component list before reintroducing $Z$.
We use this unmarked or partially marked list only geometrically; the
cycle calculation will take place after all labelled markings are retained.

\proofstep{Step 3: reintroduce the labelled markings}
The morphism $F_N$ is the ordered configuration space of distinct sections
of the smooth nonstacky locus of the universal curve.  It is smooth,
surjective, and of relative dimension $|Z|$.  Partial coarsening is the
identity on this locus, so
\begin{equation}
\begin{tikzcd}
 W_{R,Z}\ar[r,"\pi_{R,r}^+"]\ar[d,"F_R"']&
 W_{r,Z}\ar[d,"F_r"]\\
 W_{R,\varnothing}\ar[r,"\pi_{R,r}"']&W_{r,\varnothing}
\end{tikzcd}
\label{eq:zero-order-marking-square}
\end{equation}
is Cartesian.  Here the vertical arrows are the restrictions of the
marking-forgetful maps, denoted by the same symbols $F_N$.
Let $\pi_{N,\varnothing}:\mathcal C_{N,\varnothing}\to W_{N,\varnothing}$
and $\pi_{N,Z}:\mathcal C_{N,Z}\to W_{N,Z}$ be the universal twisted
curves, and let $L_{N,\varnothing}^+$ and $L_{N,Z}^+$ be their universal
root lines.  The second curve and line are the pullbacks of the first
along $F_N$.  The line-level bases for the section obstruction theories
pull back by the same marking-restoration construction.  Derived flat
base change therefore gives
\[
 F_N^*\mathbf R(\pi_{N,\varnothing})_*L_{N,\varnothing}^+
 \simeq\mathbf R(\pi_{N,Z})_*L_{N,Z}^+,
\]
compatibly with the section obstruction theories and their maps to the
relative cotangent complexes over those line-level bases.  Crumplin's local section-obstruction calculation therefore
applies on $W_{N,Z}$.  It gives \eqref{eq:genus-one-decomposition},
with multiplicity one, without forming a proper pushforward on
$W_{N,\varnothing}$ when that stack is completely unmarked.

For the main component, the generic source curve is smooth and connected,
so its ordered configuration space is geometrically irreducible.  A smooth
surjective pullback of the reduced irreducible main component is reduced,
and each of its irreducible components dominates the base.  The generic
fibre is irreducible, so this pullback has just one irreducible component
and is the reduced closure of the pulled-back distinguished locus.
This proves
\eqref{eq:zero-order-marking-pullbacks}.  No assertion of irreducibility
of configuration spaces on reducible curves is needed.

\proofstep{Step 4: the two possible powers of the root-order multiplier}
We perform the Chow calculation on $W_{N,Z}$.  Since $n+\rho>0$, every
elliptic component has a marking or an attaching node, so its automorphism
group fixing the special points is affine.  Automorphisms of the other
components, the line bundle, and the twisted structure also have affine
stabilizers.  The stacks are therefore stratified by global quotients, and
proper rational pushforward for the relative-DM comparison maps is available
\cite[Proposition~3.5.9]{Kresch} and \cite[Theorem~B.17 and
Proposition~B.18]{BSS}.  This avoids the non-affine translation stabilizers
that can occur on a completely unmarked smooth elliptic curve.

To compare degrees after restoring $Z$, first pass to a common finite
etale cover of the generic stratum that labels the irreducible
components of the source curve.  For each distribution of the labels
in $Z$, the marking-restoration fibre is then a product of ordered
configuration spaces on smooth irreducible curves, and is
geometrically irreducible.  The comparison retains the labelled
components and this distribution.  Generic degree is preserved by
this base change, with degrees added if a source component splits.
Descending the calculation shows that the total generic degree above
each marked target component equals the degree above its unmarked
target component.
Write $T_r$ for a marked target component and $T_R$ for a marked
source component dominating it.  Proper pushforward and \eqref{eq:genus-one-decomposition} give
\begin{equation}
 Q_Z(\lambda)
 =\sum_{T_r\in\operatorname{Irr}(W_{r,Z})}
   \left(\sum_{T_R\,\mathrm{dominating}\,T_r}\deg(T_R/T_r)\right)[T_r].
 \label{eq:componentwise-root-pushforward}
\end{equation}
The fixed integers defining the corresponding $\tau\in I$ give the
compatible family required by \cite[Theorem~3.25]{Crumplin}, whose
hypotheses were verified in Step~1.  Writing $G$ for the graph of $\tau$,
its exponent is
\[
 b_1(G)+2\sum_{v\in V_+(G)}g_v-|V_+(G)|
 =\begin{cases}0,&\tau=\tau_{\mathrm{triv}},\\
                 1,&\tau\ne\tau_{\mathrm{triv}}.
   \end{cases}
\]
For each fixed compatible type, the theorem gives a degree
$C\lambda^{k}$ with $C$ independent of $\lambda$.  Its hypotheses hold
for every integer $\lambda\geq1$, including $\lambda=1$.  At that value
the comparison is the identity, so $C=1$.  Consequently
\begin{equation}
 \sum_{T_R\,\mathrm{dominating}\,T_r}\deg(T_R/T_r)=
 \begin{cases}
  1,&\text{on the main component},\\
  \lambda,&\text{on every other component}.
 \end{cases}
 \label{eq:star-root-degree}
\end{equation}
This also covers the all-degenerate one-vertex type: the nodal
hypotheses are then vacuous and its exponent is $2\cdot1-1=1$.

By Step~3 the marked main component is irreducible.  Summing the degrees
one and $\lambda$ over the complete component list gives
\eqref{eq:explicit-genus-one-polynomial}.  As a check, at $\lambda=1$
the comparison is the identity and $Q_Z(1)=[W_{r,Z}]^{\vir}$.
\end{proof}

\begin{remark}[What the genus-one calculation does and does not remove]
In particular,
\[
 [Z_{W_r}^{\main}]=2Q_Z(1)-Q_Z(2).
\]
This is an alternative expression for the constant coefficient, not a
replacement for the properness and obstruction-theory arguments.
Crumplin's statement that the trivial type is the only essential type with
root-order exponent zero holds more generally.  The extra genus-one input
is that the virtual class itself is the sum of the fundamental cycles of
the irreducible components indexed by inducible essential types.
For example, in genus $g>1$ a single internal genus-$g$ vertex has
component dimension $4g-4+\ell$, larger than the virtual dimension
$3g-3+\ell$ \cite[Corollary~4.3 and Section~4.5]{Crumplin}.
The same component-degree calculation would then omit genuine
virtual contributions.  The completely unmarked genus-one auxiliary case
is also not used: the Chow formalism of \cite{BSS} excludes the non-affine
translation stabilizers of $\mathfrak M_{1,0}$.  This is why
\Cref{thm:main-component} is stated with a retained marking and why the
calculation above is made after reintroducing $Z$.

The star classification describes the essential type indexing each
irreducible component, equivalently the type on its distinguished dense
stratum.  Rational circuits can still occur in special fibres, which is why
\Cref{lem:uniform-slope-bound} and the full $2$-base changes
are retained.  Nor are intersections of components discarded: refined
Gysin pullback can be supported there.  Finally, the universal
virtual-class identity makes no assertion that the geometric moduli
space of maps to $X$ is unobstructed.
\end{remark}

\section{The space \texorpdfstring{$K_r^+$}{K-r-plus} and Crumplin's main component}
\label{sec:positive-space}

We prove that the positive BNR space is reduced and irreducible, finite
over the chosen bounded open, and maps with generic degree one onto
Crumplin's main component.  These assertions complete
\Cref{thm:main-component}.  We then recover the negative chimera through
BNR's classical puncturing construction and refined zero-section pullback.

\subsection{Modular interpretation}

The stack $K_r^+$ and the morphism
$\omega_r^+:K_r^+\to\overline O_r^+$ were defined in
\eqref{eq:positive-space-fibre-product} and
\eqref{eq:positive-forgetful-map}.  We now identify that Cartesian
construction with a moduli problem of basic rooted logarithmic maps.  This is
also the point at which we verify that the definition retains the complete
line--section data, not only the characteristic monoids.

For precision, define
$\mathfrak{LogOrb}^{\mathrm{bas}}_{\Lambda_\Gamma^+}
(\Acal_r\mid\Dcal_r)$ independently of \eqref{eq:positive-space-fibre-product}.
For a marking of coarse contact $c_i\geq0$, put
\[
 s_i=\frac r{\gcd(r,c_i)},\qquad
 \widetilde c_i=\frac{s_ic_i}{r}\quad(c_i>0),
 \qquad s_i=1,\quad\widetilde c_i=0\quad(c_i=0).
\]
We write $p_i$ for the coarse marking and $\widetilde p_i$ for its
marking divisor on the twisted curve.
An object over a scheme $S$ is a logarithmically twisted prestable curve
$(\cC,M_\cC)\to(S,M_S)$ of genus one with the prescribed ordered marking
indices, in the sense of \cite{OlssonLog,OlssonTwisted}, and a basic
representable logarithmic map
\begin{equation}
 \widetilde f:(\cC,M_\cC)\longrightarrow(\Acal_r,M_{\Acal_r}).
\label{eq:basic-rooted-map}
\end{equation}
In line--section language it is a pair $(\cL,s)$ with
\begin{equation}
 \deg\cL=d^+/r,
 \qquad
 \deg((\cL^{\otimes r})_{\mathrm{desc}})=d^+.
\label{eq:rooted-degrees}
\end{equation}
Here the subscript $\mathrm{desc}$ denotes descent to the relative coarse
curve $C$.
If $h_r\in\Gamma(\cC,\overline M_\cC)$ is the rooted target coordinate and
$q_{\log}:(\cC,M_\cC)\to(C,M_C)$ is logarithmic coarsening, the datum includes
a section $h\in\Gamma(C,\overline M_C)$ and a full
Deligne--Faltings compatibility identification
\begin{equation}
 rh_r=q_{\log}^*h,
 \qquad
 (\cL,s)^{\otimes r}\simeq q_{\log}^*\DF_{M_C}(h).
\label{eq:relative-coarsening}
\end{equation}
The second identification is the composite of the $r$th tensor power of
$(\cL,s)\simeq\DF_{M_\cC}(h_r)$, the monoidal identification
$\DF_{M_\cC}(h_r)^{\otimes r}\simeq\DF_{M_\cC}(rh_r)$, and the
pullback identification induced by the first equality.
The logarithmic structure on the coarsened curve in
\eqref{eq:relative-coarsening} is obtained over the rooted logarithmic
base.  The resulting unrooted logarithmic map need not be basic.
Whenever we map to a stack of basic unrooted logarithmic maps, we also
perform canonical basicification.  This replaces the base logarithmic
structure by the canonical basic one without changing the underlying
coarse curve or the descended line--section pair.

Fix a combinatorial type $\Theta$ with source graph $G_\Theta$, vertex set
$V(\Theta)$, and bounded-edge set $E(\Theta)$.  Write $V_0(\Theta)$ for
the vertices assigned the target face $\{0\}$ and $V_+(\Theta)$ for those
assigned the ray $\mathbb R_{\geq0}$.  Let $\cC_v$ be the component
indexed by $v$, and put
\[
 \widetilde d_v:=\deg(\cL|_{\cC_v})=d_v/r.
\]
At an oriented node $e:v_1\to v_2$, with edge length in the twisted lattice
$\widetilde\ell_e$, rooted vertex positions $\widetilde x_v$, coarse slope $m_e$, and
node index $s_e$, the conditions are
\begin{equation}
 \widetilde x_{v_2}=\widetilde x_{v_1}+\widetilde m_e\widetilde\ell_e,
 \qquad
 |m_e|<r,
 \qquad
 s_e=\frac r{\gcd(r,m_e)},
 \qquad
 \widetilde m_e=\frac{s_em_e}{r}.
\label{eq:rooted-node-data}
\end{equation}
Let $H(v)$ be the set of bounded-edge flags incident at $v$, each oriented
away from $v$.  For $f\in H(v)$ let $e(f)$ be its edge and
$\widetilde m_f$ its outgoing rooted slope.  Reversing a flag changes the
sign of the slope; a loop contributes both flags, with opposite slopes.
The orbifold degree at $v$ satisfies
\begin{equation}
 \widetilde d_v=
 \sum_{f\in H(v)}\frac{\widetilde m_f}{s_{e(f)}}
 +\sum_{\widetilde p_i\subset\cC_v}\frac{\widetilde c_i}{s_i}.
\label{eq:orbifold-balancing}
\end{equation}
The rooted cone $\widetilde\tau_\Theta$ consists of the nonnegative
coordinates $(\widetilde\ell_e,\widetilde x_v)$ for
$e\in E(\Theta)$ and $v\in V_+(\Theta)$ satisfying the continuity
equations \eqref{eq:rooted-node-data}; set $\widetilde x_v=0$ for
$v\in V_0(\Theta)$.  These equations include the circuit relations.  Its
integral lattice is
\[
 \widetilde N_\Theta:=
 \operatorname{Span}_{\mathbb R}(\widetilde\tau_\Theta)
 \cap\mathbb Z^{E(\Theta)\sqcup V_+(\Theta)}.
\]
The corresponding coarse coordinates satisfy
\[
 \ell_e=s_e\widetilde\ell_e,
 \qquad x_v=r\widetilde x_v.
\]
Define
\begin{equation}
 Q_\Theta=\widetilde\tau_\Theta^\vee
              \cap\widetilde N_\Theta^\vee.
\label{eq:basic-monoid}
\end{equation}
Here the dual cone is taken in
$\widetilde N_{\Theta,\mathbb R}^\vee$, where
$\widetilde N_{\Theta,\mathbb R}=\widetilde N_\Theta\otimes_{\mathbb Z}\mathbb R$,
and $\widetilde N_\Theta^\vee=\operatorname{Hom}(\widetilde N_\Theta,\mathbb Z)$.
Write $\widetilde\rho_e,\widetilde h_v\in Q_\Theta$ for the restricted
edge-length and vertex-height coordinate covectors, respectively; their
values at a point of the cone are $\widetilde\ell_e$ and
$\widetilde x_v$.  For $v\in V_0(\Theta)$ set $\widetilde h_v=0$.
Basicness means that the \emph{canonical} map
$Q_\Theta\to\overline M_{S,\bar s}$, induced by the node smoothing and
vertex-degeneracy data, is an isomorphism at every geometric point
$\bar s\to S$, with $\Theta$ the type of that fibre.  All characteristic
monoids are fine, saturated, and sharp, and their
groupifications are torsion-free.  These requirements include circuit equations and
saturation; an abstract isomorphism of monoids is not substituted for the
canonical one.

\begin{lemma}[Deligne--Faltings reconstruction]
\label{lem:DF-reconstruction}
For a fine saturated logarithmic stack $(Y,M_Y)$, the
Deligne--Faltings construction below is the standard one
(see \cite[Theorem~3.6]{BV} and \cite{OlssonLog}).  A logarithmic map to
$\Acal=[\bbA^1/\bbG_m]$ is equivalently a characteristic section
$h\in\Gamma(Y,\overline M_Y)$, a line--section pair $(L,s)$, and an
isomorphism of line--section pairs
\begin{equation}
 (L,s)\xrightarrow{\sim}\DF_{M_Y}(h).
\label{eq:DF-diagram}
\end{equation}
This equivalence includes arrows and commutes with arbitrary base change.
For the rooted generator $h_r$,
\begin{equation}
 \DF(h_r)^{\otimes r}\simeq\DF(rh_r).
\label{eq:DF-power}
\end{equation}
Thus the construction retains units, line-bundle gluing, and automorphisms;
it is not merely a construction from the ghost sheaf.
\end{lemma}

\begin{proof}
Write
\[
 \operatorname{Div}_Y=[\mathcal O_Y/\mathcal O_Y^*]
\]
for the symmetric monoidal stack of line bundles with section.  It is not
a Picard stack: a pair with zero section is not tensor-invertible.  The
Deligne--Faltings functor associated to $M_Y$ is the symmetric monoidal
functor
\[
 \DF_{M_Y}:\overline M_Y\longrightarrow\operatorname{Div}_Y.
\]
We recall its construction because this is where the unit and line-bundle
gluing data enter.  If $h\in\overline M_Y(U)$, let $P_h$ be the sheaf on
the small etale site of $U$ whose sections over $U'\to U$ are the lifts
$m\in M_Y(U')$ of $h|_{U'}$.  Multiplication by
$\mathcal O_{U'}^*$ makes $P_h$ an $\mathcal O_U^*$-torsor.  The
structure homomorphism $\alpha:M_Y\to\mathcal O_Y$ is equivariant for this
action, so $(P_h,\alpha|_{P_h})$ is equivalently a line bundle $L_h$ with
a section $s_h$.  By definition
$\DF_{M_Y}(h)=(L_h,s_h)$.  Multiplication in $M_Y$ gives canonical,
associative, and symmetric isomorphisms
\begin{equation}
 \DF_{M_Y}(h+h')\simeq
 \DF_{M_Y}(h)\otimes\DF_{M_Y}(h').
 \label{eq:revised-display-1}
\end{equation}

The divisorial logarithmic structure of
$\Acal=[\mathbb A^1/\mathbb G_m]$ has the global Deligne--Faltings chart
$\mathbb N\to\operatorname{Div}_{\Acal}$, whose generator is the
universal line--section pair.  An ordinary logarithmic chart by
$\mathbb N$ is used only after locally trivializing this line.  Therefore a
logarithmic morphism $g:(Y,M_Y)\to\Acal$ consists of its underlying
morphism, hence a line--section pair $(L,s)$ on $Y$, together with a map
of logarithmic structures.  On characteristics the latter sends the
target generator to a section $h\in\Gamma(Y,\overline M_Y)$.  Compatibility with
the structure homomorphisms is precisely a $2$-isomorphism
\[
 (L,s)\xrightarrow{\sim}\DF_{M_Y}(h).
\]
This constructs the functor from logarithmic maps to the data in the
statement.

Conversely, start with $(h,(L,s),\epsilon)$, where
$\epsilon:(L,s)\simeq\DF_{M_Y}(h)$.  Etale locally choose a lift
$\widetilde h\in M_Y$ of $h$.  Under $\epsilon$, the local equation of
$s$ is the image $\alpha(\widetilde h)$.  On an overlap two lifts differ
by a unique unit, and the corresponding transition function of $L$ is prescribed by the
Deligne--Faltings identification.
Consequently the local homomorphisms from the chart $\mathbb N$ glue to a
map from the pullback of the target logarithmic structure to $M_Y$.
This gives a logarithmic morphism to $\Acal$.  Changing the local lifts
changes neither the glued map nor its $2$-isomorphism class.

A morphism between two triples over the fixed logarithmic stack $(Y,M_Y)$
requires equality of their characteristic sections and an isomorphism of
line--section pairs commuting with the two DF isomorphisms.  For a change
of logarithmic base, the same condition is imposed after pullback.  The
preceding local construction sends such a morphism to, and
recovers it from, a unique $2$-morphism of logarithmic maps.  Thus the two
constructions are quasi-inverse equivalences of groupoids, not merely a
bijection on isomorphism classes.  Formation of $P_h$, its associated
line, and the displayed descent data commutes with arbitrary base change,
which proves the base-change assertion.

Finally, iterating the preceding symmetric-monoidal isomorphism for the
sum of $r$ copies of $h_r$ gives canonically
\[
 \DF(h_r)^{\otimes r}\simeq\DF(rh_r).
\]
The isomorphism carries the $r$th power of the DF section to the DF section
of $rh_r$, proving \eqref{eq:DF-power} as an isomorphism of line--section
pairs.
\end{proof}

\begin{proposition}[Modular interpretation of $K_r^+$]
\label{prop:positive-space-modular-equivalence}
There is a canonical equivalence
\begin{equation}
 K_r^+\simeq
 \mathfrak{LogOrb}^{\mathrm{bas}}_{\Lambda_\Gamma^+}
       (\Acal_r\mid\Dcal_r),
\label{eq:positive-space-modular-equivalence}
\end{equation}
compatible with arbitrary base change and inducing isomorphisms of
automorphism group schemes.  Under this equivalence,
\begin{equation}
 \omega_r^+:K_r^+\longrightarrow\overline O_r^+
\label{eq:positive-space-forgetful}
\end{equation}
retains the twisted curve, the actual root and descended line--section
pairs, and their ordinary $r$th-power identification.  It forgets the
logarithmic structures and their Deligne--Faltings identifications with
characteristic sections.
\end{proposition}

\begin{proof}
BNR's Cartesian presentation of universal logarithmic maps is an equality
of stacks, not only of characteristic monoids; in the positive datum there
are no puncturing equations, so the puncturing substack is the whole Artin
fan.  Restricting their presentation to the cone stack $T_r^+$ gives the
2-fibre product in \eqref{eq:positive-space-fibre-product}
\cite[Construction~3.15 and Lemma~3.18]{BNR}.  We unpack this presentation
and compare it with the moduli description above.

An $S$-object of the fibre product consists of a twisted prestable curve,
a map to the Artin fan of a rooted tropical type, and the specified
$2$-isomorphism over the Artin fan of twisted curves.  On a geometric fibre
the rooted type supplies the canonical basic monoid $Q_\Theta$, the vertex
positions and edge slopes, and hence a characteristic piecewise-linear
section $h_r$ satisfying the continuity and circuit equations.  The Artin-fan point and the curve compatibility supply the full
logarithmic structure.  Applying \Cref{lem:DF-reconstruction} to $h_r$
in this structure reconstructs the line--section pair, including units
and gluing, and \eqref{eq:DF-power} supplies its
$r$th-power identification.  The marking and node sectors in $T_r^+$ give
exactly the prescribed stabilizer indices and faithful characters.  This
constructs a basic rooted logarithmic map and defines a functor from the
left side of \eqref{eq:positive-space-modular-equivalence} to the right.

Conversely, tropicalizing a basic rooted logarithmic map records its twisted
dual graph, genera, component degrees, image faces, marking and node
indices, vertex positions, and edge slopes.  The logarithmic-map equation
gives the continuity relations, the divisor-degree formula gives the
orbifold balancing relation, and representability gives BNR's coprimality
condition.  Basicness identifies the base characteristic with the canonical
monoid $Q_\Theta$, so the full logarithmic data define an $S$-point of the Artin fan
$\Acal(T_r^+)$ and the required $2$-isomorphism over
$\Acal(\mathsf M^{\mathrm{trop,tw}})$.  The cone stack itself is a
stack on cones, not a moduli functor on arbitrary schemes.  This gives the inverse functor.

Both composites recover the same canonical monoid, Deligne--Faltings
functor, root line, section, and power isomorphism.  Morphisms preserve these
canonical charts; descent through graph automorphisms is already part of the
cone-stack and $2$-fibre-product formalism.  Hence the functors are inverse
on objects and arrows and identify automorphism group schemes.  Every
construction is compatible with base change.  The description of
$\omega_r^+$ follows immediately.
\end{proof}

The modular interpretation in
\Cref{prop:positive-space-modular-equivalence} will be used below both to
analyze the components of $K_r^+$ and to prove properness of its forgetful
morphism.

\begin{remark}[Why this is not the whole orbifold stack]
On a smooth genus-one curve, a zero section of a nontrivial
$L\in\operatorname{Pic}^0(C)$ is an ordinary map to the boundary
$B\bbG_m\subset\Acal_r$.  It is not in $K_r^+$: a one-vertex basic
logarithmic map has constant characteristic coordinate and
\Cref{lem:DF-reconstruction} forces the corresponding fibrewise line to be
the DF divisor, with no independent Jacobian twist.  This is the precise
positive-genus obstruction to the genus-zero identification with the full
orbifold space.
\end{remark}

\subsection{Irreducibility and the trivial-type locus}

\begin{proposition}[Irreducibility of $K_r^+$]
\label{prop:positive-space-irreducible}
Put $\ell=n+\rho$ and $d=d^+=\sum_i c_i$.  Use the marking conventions
of the preceding subsection and the node conventions
\eqref{eq:rooted-node-data}, with $0\leq c_i<r$.
Then the stack $K_r^+$ defined in \eqref{eq:positive-space-fibre-product} is
reduced and irreducible.  The open substack $(K_r^+)^\circ$ on which the
source is smooth and the universal section is not identically zero is
nonempty and dense.  These assertions include the case $d=0$ and any number
of markings of contact order $0$.  The stack $K_r^+$ is smooth of pure
relative dimension $3\cdot1-3+\ell=\ell$ over its Artin fan; in particular,
every finite-type open substack is pure-dimensional.
\end{proposition}

\begin{proof}
\proofstep{Step 1: the common zero face}
Let $\Theta^\circ$ have one vertex of genus one and degree $d/r$, no
bounded edges, vertex image $0$, and the prescribed labelled legs.  Its
balancing equation is $d=\sum_i c_i$.  For a type $\Theta$ of $T_r^+$,
the zero face of its cone is this type: setting all edge lengths and vertex positions to
zero contracts the entire graph.  The genus of the contracted vertex is
$\sum_v g_v+b_1(G_\Theta)=1$, and its degree is the sum of the vertex
degrees.  This argument includes a circuit and does not use tree
uniqueness of slopes.  There are no puncturing inequalities on the
positive cone stack.

The Artin cone of a sharp saturated toric monoid $Q$ is
$[\Spec\bbC[Q]/\Spec\bbC[Q^{\mathrm{gp}}]]$.  It is reduced and
irreducible, and its dense open torus quotient is the zero-face stratum.
The charts of $\Acal(T_r^+)$ glue along face opens and all contain the
same zero-face stratum $\Acal(\Theta^\circ)$.  Thus
$\Acal(T_r^+)$ is reduced and irreducible, with this dense open stratum.
The graph-automorphism identifications do not alter this conclusion.

\proofstep{Step 2: the smooth curve-moduli presentation}
By \eqref{eq:positive-space-fibre-product}, the morphism
\[
 q:K_r^+\longrightarrow\Acal(T_r^+)
\]
is a base change of the boundary-stratification morphism
$\mathfrak M^{\mathrm{tw}}_{1,\ell}\to
\Acal(\mathsf M^{\mathrm{trop,tw}})$.
This morphism is smooth and surjective, of pure relative dimension
$3\cdot1-3+\ell=\ell$; see the twisted-curve Cartesian square and its
proof in \cite[Lemma~3.5 and Lemma~3.18]{BNR}.  The assertion also follows
on labelled curve charts: deformations of the pointed normalizations are
unobstructed and the node-smoothing parameters give the boundary
coordinates.  Unstable pointed components are retained as Artin stacks;
no stabilization is performed.

A smooth morphism is open.  Hence the inverse image of a dense open of
its target is dense in its source: the image of any nonempty source open
is a nonempty target open and meets the dense open.  It follows that
$q^{-1}\Acal(\Theta^\circ)$ is dense in $K_r^+$.

\proofstep{Step 3: identify this dense open}
On $\Theta^\circ$ the source is smooth and the root section is nonzero.
Its forced marking zeros already have degree
$\sum_i\widetilde c_i/s_i=d/r$.  Therefore its divisor is exactly
$\sum_i\widetilde c_i\widetilde p_i$ and its line--section pair is
\[
 \left(\mathcal O\Bigl(\sum_i\widetilde c_i\widetilde p_i\Bigr),
       s_{\sum_i\widetilde c_i\widetilde p_i}\right).
\]
The second entry denotes the canonical section of this effective-divisor
line bundle.
The marking root stacks are determined by the marked smooth coarse curve.
The nonzero section fixes the scalar of the line bundle, so there is no
independent Jacobian choice or residual scalar in the fibre.  Consequently
$q^{-1}\Acal(\Theta^\circ)=(K_r^+)^\circ$ is the irreducible stack of
smooth genus-one curves with these ordered markings and their prescribed
marking roots.  This remains true with no markings, when it is an Artin
stack; a Chow pushforward on that unmarked stack is not being asserted.

The dense open is irreducible, so $K_r^+$ is irreducible.  Smoothness over
the reduced Artin fan proves reducedness.  Artin cones have dimension zero,
so every nonempty finite-type open has pure dimension $\ell$.  For $d=0$
all $c_i$ vanish and the dense pair is $(\mathcal O,1)$.  A constant
positive vertex height adds a boundary stratum, not another component.
On a smooth fibre its associated line has no independent degree-zero
twist: this follows from the full Deligne--Faltings reconstruction in
\Cref{lem:DF-reconstruction}, not from its numerical degree alone.
\end{proof}

\subsection{Properness over the finite-type open}

We first prove properness after forgetting the markings of contact order $0$.  Retain
$W_0$, $B_0$, and the base root order $r$ from
\Cref{lem:uniform-slope-bound}, and abbreviate
\[
 \mathfrak U^+:=\mathfrak U^+_\varnothing,
 \qquad \ell_+:=\rho_+.
\]
Throughout this subsection the curve moduli and tropical curve cone stacks
carry precisely these $\ell_+$ remaining positive-contact markings.  Let
$T_\varnothing^+$ denote the corresponding unrooted positive tropical
cone stack, and let $T_{r,\varnothing}^+$ be its rooted version with BNR's
size condition at order $r$.  These differ from $T_r^+$, which retains
every labelled marking.
Let $(M,t)$ be the universal line--section pair on $\mathfrak U^+$, and let
$q_r:O_{r,\varnothing}^+\to\mathfrak U^+$ be the natural forgetful morphism.  Put
\begin{equation}
 W:=O_{r,\varnothing}^+\times_{\mathfrak U^+}W_0
     =W_{r,\varnothing}.
 \label{eq:revised-display-3}
\end{equation}
For every union $C_A$ of irreducible components, the bounds fixed in
\Cref{lem:uniform-slope-bound} give
\begin{equation}
 \left|\deg(M|_{C_A})\right|\leq B_0,
 \qquad
 \sum_{p_i\in C_A}c_i\leq B_0.                               
 \label{eq:positive-open-bounds}
\end{equation}
The second inequality uses the fact that the remaining contact orders are
positive and have total sum $d^+\leq B_0$.  If
$q:\mathcal C\to C$ is the relative coarse morphism, an object of $W$
carries
\begin{equation}
 (L,s)^{\otimes r}\simeq q^*(M,t).
 \label{eq:revised-display-4}
\end{equation}
We use the full fibre product $W=W_{r,\varnothing}$.  No exact-type
stratum or union of pairwise disjoint components is substituted for it;
all intersections and specializations lying over $W_0$ remain present.

\begin{lemma}[The rooted line--section space over $W_0$]
\label{lem:ordinary-root-stack-separated}
The morphism
$q_{r,W_0}:W\to W_0$ is proper and of Deligne--Mumford type.  In
particular it is separated and $W$ is finite type.
\end{lemma}

\begin{proof}
The subcurve bounds on $W_0$ and $r>2B_0$ imply
$|\deg(L|_E)|\leq B_0/r<1/2$ for every root over $W_0$.  Thus imposing
\eqref{eq:crumplin-degree-requirement} removes no objects from this
root fibre.  After forgetting the section, the morphism is the base change
to $W_0$ of the proper faithful-root stack of
\Cref{prop:faithful-root-properness}.  Reintroducing the root section
amounts to the equation $s^r=q^*t$, whose solution space is finite by
Lemma~\ref{lem:finite-section-roots}.  Composition with this finite section-root morphism preserves
properness, finite type, and Deligne--Mumford type.
\end{proof}

Let $\mathfrak L^+$ be the stack of basic positive logarithmic prestable
maps to $(\mathcal A\mid\mathcal D)$ with the positive marking data after $Z$ has been forgotten, and let
\begin{equation}
 F_{\log}:\mathfrak L^+\longrightarrow\mathfrak U^+
 \label{eq:revised-display-5}
\end{equation}
forget the logarithmic enhancement while retaining the actual line bundle and
section.  The Deligne--Faltings isomorphism belongs to the logarithmic
enhancement and is part of the fibre of this forgetful morphism, not part of
its target.  Put
$\mathfrak L_{W_0}^+=\mathfrak L^+\times_{\mathfrak U^+}W_0$.

We now make the finite type support precise before invoking any compactness
statement for logarithmic enhancements.

A \emph{finite face-closed cone substack} means the finite union of the
cones over a selected finite set of labelled types, together with every
graph-automorphic copy and every face.  Its associated Artin fan is the
corresponding open substack of the ambient Artin fan.

\begin{lemma}[Finite face-closed cone support]
\label{lem:finite-face-closed-support}
There is a finite face-closed cone substack
\begin{equation}
                         T^+_\Sigma\subset T_\varnothing^+
 \label{eq:revised-display-6}
\end{equation}
containing every realizable basic type above $W_0$.  It can be chosen before
$r$; the index $\Sigma$ records the selected finite family of types and
their faces.  Every bounded-edge slope on every cone of $T^+_\Sigma$ satisfies
\begin{equation}
                              |m_e|\leq4B_0.
 \label{eq:revised-display-7}
\end{equation}
\end{lemma}

\begin{proof}
Choose a finite-type scheme atlas of $W_0$.  A finite-type family of
prestable curves has only finitely many labelled dual graphs after a finite
stratification, and the multidegrees and marking distributions occurring on
that atlas form finite sets.  The root-independent estimate of \Cref{lem:uniform-slope-bound} gives
$|m_e|\leq4B_0$ for every realizable type over $W_0$, including special-fibre types.
In particular this applies to rational circuits as well as trees; the
classification of the essential types indexing the irreducible components
is not used here.

There are therefore only finitely many choices of graph, multidegree,
marking distribution, integral slopes in the bounded interval, and image
faces of the rank-one target.  Take the corresponding cones, every
graph-automorphic copy, and every face.  This gives the required finite
face-closed support.  Notice that no properness or quasi-compactness of a
logarithmic-enhancement stack is used in this construction.
\end{proof}

Let $\mathfrak L^+_{W_0,\Sigma}$ denote the stack of basic fs rank-one
logarithmic enhancements of the fixed underlying prestable curve and actual
line--section pair over $W_0$, with type in $T^+_\Sigma$.

\begin{lemma}[Finiteness of basic logarithmic enhancements]
\label{lem:unrooted-log-enhancements-proper}
The forgetful morphism
\begin{equation}
 F_{\log,W_0}:\mathfrak L^+_{W_0,\Sigma}\longrightarrow W_0
 \label{eq:revised-display-8}
\end{equation}
is representable and finite.  Its formation is compatible with arbitrary
base change.  The assertion includes the locus on which the section
vanishes identically on every component.
\end{lemma}

\begin{proof}
We reduce to the finite forgetful morphism of
\cite[Theorem~1.2.1]{Chen} for a smooth projective pair.  This avoids
applying a projective-target compactness theorem directly to the Artin
stack $\Acal$.

\proofstep{Step 1: choose a projective realization of the fixed pair}
Work smoothly locally on $W_0$, with actual data $(C/S,M,t)$.  After a
further smooth cover, add auxiliary ordered smooth sections meeting every
unstable component until the marked curve is stable.  These sections are
used only to obtain a relatively ample line bundle $H$ on $C/S$; they are
not part of the mapping problem or its contact data.  The cover can be
chosen near each point and is surjective after taking their union.
Thus the construction below proves a property local on the original base.

After increasing $k$ and shrinking the base, set
\[
 B=H^{\otimes k},\qquad A=M\otimes H^{\otimes k}.
\]
We may assume that $B$ is relatively very ample, that $A$ is relatively
globally generated, and that both direct images are locally free and
commute with base change.  Write $\pi_C:C\to S$ for the curve projection.
After restricting to a locus of constant ranks, put
\[
 a+1=\operatorname{rk}\bigl((\pi_C)_*A\bigr),
 \qquad b+1=\operatorname{rk}\bigl((\pi_C)_*B\bigr).
\]
Choose frames of these direct images.  Evaluation then defines
\begin{equation}
 \begin{gathered}
 u:C\longrightarrow Z_{\mathrm{aux}}\times S,
 \qquad Z_{\mathrm{aux}}=\mathbb P^a\times\mathbb P^b,\\
 u^*\mathscr L\xrightarrow{\sim}M,
 \qquad\mathscr L=\mathcal O_{Z_{\mathrm{aux}}}(1,-1).
 \end{gathered}
 \label{eq:projective-realization-base}
\end{equation}
The second projection is a closed immersion.  The displayed identification
is the actual one obtained from the chosen evaluation quotients, not just
an equality of degrees or isomorphism classes.

Use the convention that $\mathbb P(E)$ parametrizes lines in $E$, and put
\begin{equation}
 Y=\mathbb P_{Z_{\mathrm{aux}}}(\mathcal O_{Z_{\mathrm{aux}}}\oplus\mathscr L),
 \qquad D_0=\mathbb P_{Z_{\mathrm{aux}}}(\mathcal O_{Z_{\mathrm{aux}}}).
 \label{eq:projective-realization-target}
\end{equation}
Let $p_{\mathrm{aux}}:Y\to Z_{\mathrm{aux}}$ be the projection.  This is
a smooth projective pair.  The chart on which the first coordinate
of the line is nonzero is $\operatorname{Tot}(\mathscr L)$, and $D_0$ is
its zero section.  The graph of $t$ gives
\begin{equation}
 j:C\longrightarrow\operatorname{Tot}(\mathscr L)\times S
       \subset Y\times S,
 \qquad
 j^*(\mathcal O_Y(D_0),s_{D_0})\xrightarrow{\sim}(M,t).
 \label{eq:projective-realization-graph}
\end{equation}
Indeed, $\mathcal O_Y(D_0)=\mathcal O_Y(1)\otimes p_{\mathrm{aux}}^*\mathscr L$,
and the graph line generated by $(1,t)$ trivializes $j^*\mathcal O_Y(1)$.
This verifies the identity of sections, including over nonreduced bases
and when $t=0$.  The map $j$ is a closed immersion over $S$, since its
projection to $\mathbb P^b\times S$ is one.  It is therefore an ordinary
stable map with the original markings: no component is contracted.

\proofstep{Step 2: an equivalence of enhancement groupoids}
The classifying morphism $(Y,D_0)\to(\Acal,\Dcal)$ is strict.  By
\eqref{eq:projective-realization-graph}, its composite with $j$ is the
fixed map classified by $(M,t)$.  Enhancing this fixed composite to a
logarithmic map is consequently equivalent to enhancing the fixed map
$j$ to a logarithmic map to $(Y,D_0)$.  Explicitly, both enhancements give
a homomorphism from the same pulled-back rank-one divisorial logarithmic
structure to the logarithmic curve.  The contact orders, continuity
relations, and canonical basic monoid agree.  This identification respects
arrows and their specified underlying isomorphisms and is functorial in
arbitrary changes of $S$.

Let $\beta_j$ be the locally constant class of $j$ and let $\mathbf c$
be the original ordered positive contacts.  After decomposing the base
according to this class if necessary, there is a Cartesian square
\begin{equation}
\begin{tikzcd}[column sep=large]
 \mathfrak L^+_{W_0,\Sigma}\times_{W_0}S\ar[r,"\mathrm{realize}"]\ar[d]&
 \Mbar^{\log}_{1,\ell_+,\beta_j}(Y,D_0;\mathbf c)\ar[d,"\mathrm{forget\ log}"]\\
 S\ar[r,"{[j]}"']&\Mbar_{1,\ell_+,\beta_j}(Y).
\end{tikzcd}
 \label{eq:enhancements-Chen-base-change}
\end{equation}
Here the log-map stack is Chen's basic/minimal stable logarithmic mapping
stack with the indicated contacts, the lower-right stack parametrizes the
underlying stable maps, and the upper horizontal arrow is the enhancement
equivalence induced by the projective realization of $(M,t)$.  A point of
the fibre product includes
an identification of its underlying stable map with $j$, and therefore
fixes the complete pair $(M,t)$, including its scalar.  In particular,
there is no additional $\mathbb G_m$ in the fibre when $t=0$.
The restriction to $\Sigma$ removes no enhancements: all of their types
are in $T^+_\Sigma$ by \Cref{lem:finite-face-closed-support}.  Since this
is an open Artin-fan restriction containing every geometric point of the
enhancement stack, it removes no objects over nonreduced bases either.

\proofstep{Step 3: finiteness and descent}
The right vertical arrow of \eqref{eq:enhancements-Chen-base-change} is
finite and representable by \cite[Theorem~1.2.1]{Chen}.  Hence so is the
left vertical arrow.  Finiteness and representability descend along the
chosen smooth covers of $W_0$.  The equivalence of fibre groupoids already
proved also shows compatibility with arbitrary base change.  No
properness of $\Acal$, no extension of a chosen line bundle up to an
uncontrolled twist, and no removal of the all-degenerate locus is used.
\end{proof}

\begin{remark}[Why the projective realization is auxiliary]
The projective choices verify finiteness smoothly locally; the
fibre-groupoid equivalence and descent concern the intrinsic stack
$\mathfrak L^+_{W_0,\Sigma}$ with its fixed curve and line--section pair.
\end{remark}

For notational simplicity below we write
$\mathfrak L^+_{W_0}=\mathfrak L^+_{W_0,\Sigma}$.

By \Cref{lem:finite-face-closed-support}, only the finite
face-closed cone substack $T^+_\Sigma$ is needed.  The base root order $r$
fixed in \eqref{eq:large-r} satisfies BNR's divisibility conditions and
\begin{equation}
                 r>\max\{4B_0,M_{\boldsymbol\mu}\}.
 \label{eq:revised-display-9}
\end{equation}
It also satisfies
$\operatorname{lcm}(1,\ldots,4B_0)\mid r$.  This is a single cofinal
divisibility condition, chosen before varying $\lambda$.  Thus every
nonzero slope occurring on this support divides $r$.
For every cone of $T^+_\Sigma$, define its rooted cone by
\begin{equation}
 s_e=\frac r{\gcd(r,m_e)},
 \qquad \widetilde m_e=\frac{s_em_e}{r},
 \qquad
 \ell_e=s_e\widetilde\ell_e,
 \qquad x_v=r\widetilde x_v,
 \label{eq:revised-display-10}
\end{equation}
with $s_e=1$ when $m_e=0$ and with the analogous fixed formula at the
markings.  By \eqref{eq:revised-display-7}--\eqref{eq:revised-display-9}, these cones lie in the substack of BNR's tropical
twisted cone stack $T_{r,\varnothing}^+$ satisfying the size condition
$|m_e|<r$.  Denote the corresponding finite face-closed cone substack by
$T^+_{r,\Sigma}\subset T_{r,\varnothing}^+$.

\begin{lemma}[Comparison of the rooted and unrooted cone stacks]
\label{lem:rooted-unrooted-cone-comparison}
Let $\mathfrak M_{1,\ell_+}$ denote the stack of ordinary genus-one
prestable curves with the remaining ordered markings, and let
\begin{equation}
 \begin{aligned}
 K_{1,\Sigma,W_0}
  &: =\left(
     \mathfrak M_{1,\ell_+}
       \mathop\times^{(2)}_{\mathcal A(\mathsf M^{\mathrm{trop}})}
       \mathcal A(T^+_\Sigma)\right)
       \times_{\mathfrak U^+}W_0,\\
 K_{r,\Sigma,W_0}
  &: =\left(
     \mathfrak M^{\mathrm{tw}}_{1,\ell_+}
       \mathop\times^{(2)}_{\mathcal A(\mathsf M^{\mathrm{trop,tw}})}
       \mathcal A(T^+_{r,\Sigma})\right)
       \times_{\mathfrak U^+}W_0.
 \end{aligned}
\label{eq:revised-display-11}
\end{equation}
In these fibre products, the morphism to $\mathfrak U^+$ retains the
coarse marked curve and its actual unrooted line--section pair.  For the
rooted construction this pair is the descended $r$th tensor power of the
root pair.
Then $K_{1,\Sigma,W_0}=\mathfrak L^+_{W_0,\Sigma}$, where the subscript
$1$ denotes the ordinary unrooted logarithmic construction and imposes
no BNR root-size condition at root order one.  The root-forgetting
morphism of Artin fans is induced by \eqref{eq:revised-display-10}.
There is a
2-Cartesian square
\begin{equation}
\begin{tikzcd}
 K_{r,\Sigma,W_0}\ar[r,"\mathrm{coarsen}"]\ar[d,"\mathrm{trop}"']
 & K_{1,\Sigma,W_0}\ar[d,"\mathrm{trop}"]\\
 \mathcal A(T^+_{r,\Sigma})\ar[r,"\mathrm{forget\ root}"'] & \mathcal A(T^+_\Sigma).
\end{tikzcd}
 \label{eq:revised-display-12}
\end{equation}
Here the vertical arrows are the projections to the Artin fans.  The
upper horizontal arrow coarsens the rooted logarithmic map and then
applies canonical basicification.  The finite Kummer homomorphism from
the unrooted basic monoid to the rooted basic monoid is induced by
\eqref{eq:revised-display-10}; its Deligne--Faltings realization gives
the corresponding change of base logarithmic structures.  The lower
horizontal arrow is a generalized root stack.  Consequently
$K_{r,\Sigma,W_0}\to K_{1,\Sigma,W_0}$ is proper.
\end{lemma}

\begin{proof}
\proofstep{Step 1: the unrooted identification}
The first equality follows from the modular interpretation in
\Cref{prop:positive-space-modular-equivalence}, restricted to the finite
face-closed support $T^+_\Sigma$.  This is precisely BNR's unrooted
Cartesian presentation; the Deligne--Faltings reconstruction retains the
units and line gluing.

\proofstep{Step 2: the finite Kummer lattice extension}
Formula \eqref{eq:revised-display-10}
gives a bijection between the cones of $T^+_{r,\Sigma}$ and
$T^+_\Sigma$ and a compatible finite Kummer extension of their lattices.
Therefore
\begin{equation}
 \mathcal A(T^+_{r,\Sigma})\longrightarrow
 \mathcal A(T^+_\Sigma)
 \label{eq:revised-display-13}
\end{equation}
is a generalized root stack and is proper.

\proofstep{Step 3: the 2-Cartesian square}
For the Cartesian assertion one must retain the twisted-curve bases.  BNR's
twisted-curve Cartesian presentation and generalized-root construction
\cite[Lemma~3.5, Construction~3.10, Lemma~3.11,
Construction~3.15, and Lemma~3.18]{BNR} give
\begin{equation}
 \mathfrak M^{\mathrm{tw}}_{1,\ell_+}
 =\mathfrak M_{1,\ell_+}
   \mathop\times^{(2)}_{\mathcal A(\mathsf M^{\mathrm{trop}})}
   \mathcal A(\mathsf M^{\mathrm{trop,tw}}).
 \label{eq:revised-display-14}
\end{equation}
Substituting \eqref{eq:revised-display-14} into the second line of \eqref{eq:revised-display-11} and cancelling the
intermediate twisted tropical-curve factor gives
\begin{equation}
 K_{r,\Sigma,W_0}
 \simeq K_{1,\Sigma,W_0}
       \mathop\times^{(2)}_{\mathcal A(T^+_\Sigma)}
       \mathcal A(T^+_{r,\Sigma}),
 \label{eq:revised-display-15}
\end{equation}
which is exactly \eqref{eq:revised-display-12}.  Properness follows by base change from \eqref{eq:revised-display-13}.
This statement is made on the specified finite face-closed support.
No global assertion is needed outside it: the bound on unrooted slopes
is what ensures that every relevant cone has a rooted lift satisfying
BNR's size condition.
\end{proof}

\begin{remark}[The BNR rooting construction]
Properness follows by base change from BNR's generalized root stack on
the Artin fan.  This construction is not genus-one specific; genus one
enters the component and virtual-class calculation of Section~4.
\end{remark}

\begin{proposition}[Properness of $K_{r,\varnothing}^+$ over $W$]
\label{prop:positive-space-proper}
Let
\[
 K_{W}:=K_{r,\varnothing}^+\times_{\overline O_{r,\varnothing}^+}W.
\]
Under \eqref{eq:revised-display-9}, the forgetful morphism
\[
                    \omega_{W}:K_{W}\longrightarrow W
\]
is representable by algebraic spaces and finite; in particular, it is
projective and proper.
\end{proposition}

\begin{proof}
\proofstep{Step 1: identify the base-changed BNR space}
Because $T^+_\Sigma$ contains every unrooted chart above $W_0$, and because
\eqref{eq:revised-display-7}--\eqref{eq:revised-display-9} put every corresponding rooted type inside BNR's size condition $|m_e|<r$, associativity of the defining $2$-fibre products gives canonical equivalences
\begin{equation}
 K_{r,\Sigma,W_0}
   =K_{r,\varnothing}^+\times_{\mathfrak U^+}W_0
   =K_{r,\varnothing}^+\times_{\overline O_{r,\varnothing}^+}W
   =K_{W}.
 \label{eq:revised-display-16}
\end{equation}
This is an equality in families, not only on geometric points.  Indeed, for
an $S$-object of $K_{r,\varnothing}^+$ whose descended line--section pair
lies in $W_0$, coarsening followed by canonical basicification gives an
$S$-object of $\mathfrak L^+_{W_0}$.  Its basic chart factors through $T^+_\Sigma$ by construction.  Formula \eqref{eq:revised-display-10}
is the unique lift with the stabilizer orders and characters prescribed there,
so the rooted type map
factors through $T^+_{r,\Sigma}$.  These factorizations agree on overlaps
and under generization because both cone substacks contain all faces and all automorphic copies of
the selected labelled types.  The converse inclusion is tautological.  Thus \eqref{eq:revised-display-16} is an
equality of fibre categories after every base change.  By
Lemma~\ref{lem:rooted-unrooted-cone-comparison}, $K_{W}$ is proper
over $K_{1,\Sigma,W_0}=\mathfrak L^+_{W_0}$.  Lemma
\ref{lem:unrooted-log-enhancements-proper} makes the latter proper over
$W_0$.  Hence $K_{W}\to W_0$ is proper.

\proofstep{Step 2: properness over the rooted space over $W_0$}
The morphism $W\to W_0$ is separated by
Lemma~\ref{lem:ordinary-root-stack-separated}.  Its diagonal is therefore proper.  The graph of the
$W_0$-morphism $\omega_{W}$ is a base change of that diagonal and is
proper.  Projection of $K_{W}\times_{W_0}W$ to $W$ is proper because
$K_{W}\to W_0$ is proper.  Their composite is $\omega_{W}$.
Separatedness of an algebraic stack is not being used to assert that
its diagonal, or this graph, is a closed immersion.

\proofstep{Step 3: relative inertia and finiteness}
Work on a labelled rooted basic chart.  An automorphism over the identity
of the complete object in $W$ fixes the underlying twisted marked curve,
the root line $L$, its section, and the power identification.  By the
logarithmically twisted-curve correspondence, identity on the twisted
curve fixes the node-smoothing logarithmic data
\cite[Theorem~1.9]{OlssonTwisted}.  Equivalently, a nontrivial character
on a twisted-node smoothing produces a nontrivial node ghost and is
visible on the twisted curve.  The marking roots are likewise retained;
no marking has been rigidified in this fibre.

For the rooted type $\Theta$, use the coordinate covectors
$\widetilde\rho_e,\widetilde h_v\in Q_\Theta$ defined after
\eqref{eq:basic-monoid}.  For an edge oriented from $v$ to $w$ they satisfy
\begin{equation}
 \widetilde h_w-\widetilde h_v
      =\widetilde m_e\widetilde\rho_e.
 \label{eq:rooted-inertia-relations}
\end{equation}
For a character $\chi:Q_\Theta^{\mathrm{gp}}\to\mathbb G_m$ in
relative inertia, the node coordinates have character one.  Since the
graph is connected, \eqref{eq:rooted-inertia-relations} makes the vertex
characters equal.  A nondegenerate vertex fixes their common value to
one.  If every vertex is degenerate, the common value is the scalar on
$L$ in the Deligne--Faltings reconstruction.  The retained identity of
$L$ fixes this value to one even when its section is zero.

Here one must also check generation after saturation.  By its definition
before \eqref{eq:basic-monoid}, the lattice $\widetilde N_\Theta$ is the
intersection of the cone's real linear span with
$\mathbb Z^{E(\Theta)\sqcup V_+(\Theta)}$, and is therefore primitive.
Consequently restriction of the coordinate covectors surjects onto its
dual, which is $Q_\Theta^{\mathrm{gp}}$.  Thus the node and vertex
coordinates together generate the entire characteristic group.  This
rules out a character that is trivial on the coordinates but nontrivial
on a putative additional root.  Torsion-freeness of
$Q_\Theta^{\mathrm{gp}}$ alone would not justify that conclusion.
Type automorphisms are visible on the fixed marked curve.  Hence the
relative inertia is trivial and $\omega_W$ is representable by algebraic
spaces.

The unrooted enhancement stack is finite over $W_0$, and the generalized
root stack over its finite cone support has finitely many geometric
isomorphism classes over each point.  Therefore the geometric fibres
of $\omega_W$ are finite.  The proper representable morphism
$\omega_W$ is quasi-finite and hence finite (in particular, projective).
\end{proof}
\subsection{Reintroducing the omitted markings and recovering the BNR chimera}

\begin{corollary}[Reintroducing the temporarily forgotten markings in $Z$]
\label{cor:zero-order-marking-base-change}
Use the notation of \Cref{subsec:theorem-spaces}, in particular the spaces
$K_{W_r}^+$ and $K^+_{W_{r,\varnothing}}$ defined in
\eqref{eq:positive-space-open} and
\eqref{eq:reintroduced-distinguished-locus}.  Then the square
\begin{equation}
\begin{tikzcd}
 K_{W_r}^+\ar[r,"F_r^K"]\ar[d,"\omega_{W_r}^+"'] &
 K^+_{W_{r,\varnothing}}
       \ar[d,"\omega^+_{W_{r,\varnothing}}"]\\
 W_{r,Z}\ar[r,"F_r"'] & W_{r,\varnothing}
\end{tikzcd}
 \label{eq:revised-display-17}
\end{equation}
is $2$-Cartesian, and $\omega_{W_r}^+$ is representable and finite
(hence projective and proper).
\end{corollary}

\begin{proof}
\proofstep{Step 1: properness before reintroducing the markings in $Z$}
Apply \Cref{prop:positive-space-proper} to
$W=W_{r,\varnothing}$.  This gives a finite representable morphism
before the labelled markings in $Z$ are reintroduced.

It remains to identify the square \eqref{eq:revised-display-17}.  This assertion is made only over
the chosen finite-type open.  Globally, a marking of age $0$ can lie at an
isolated zero whose positive coarse order is a positive multiple of $r$.

\proofstep{Step 2: exclusion of positive multiples of the root order}
Forgetting one of the reintroduced markings changes neither the coarse curve, the
line--section pair, nor any component degree.  Hence the bound $B_0$ applies
to the component containing that marking.  If the root section is not identically zero there, let $a$ be the vanishing
order of its descended $r$th-power section at the marked point.  The order
$a$ is a nonnegative multiple of $r$, while the actual zero divisor on that
component has total degree at most $B_0<r$.  Hence $a=0$.  On a component with identically zero section, the restored logarithmic
map is defined with contact zero: its value at the new marking is the
pullback of the existing vertex degeneracy.  No contact order is inferred
from the vanishing of an ordinary zero section.  Thus every reintroduced marking remains in the contact-order-zero sector,
has source index one, and has trivial stabilizer character.

\proofstep{Step 3: identification of the fibre categories}
Reintroducing a marking is consequently the choice of an ordered section of the
smooth nonstacky locus of the prestable source, disjoint from the other
markings.  It adds no generator or relation to the basic characteristic
monoid, and its Deligne--Faltings value is the vertex degeneracy already
present.  Forgetting and then reintroducing the labelled markings in $Z$ therefore leave the rooted
tropical map, the actual line bundle and section, the chosen power isomorphism, and the basic logarithmic structure unchanged.
This identifies the upper-left groupoid in \eqref{eq:revised-display-17}, functorially in the base,
with the $2$-fibre product of the other three groupoids.  The final assertion
is the base change of the finite representable morphism before the labelled markings in $Z$ are reintroduced.
\end{proof}

\begin{corollary}[Pushforward to Crumplin's main component]
\label{cor:positive-main-cycle}
Return to the full marking set and assume $\ell=n+\rho>0$.  For the
spaces $K_{W_r}^+$, $W_{r,Z}$ and $U_{r,Z}$ defined in
\Cref{subsec:theorem-spaces}, the morphism $\omega_{W_r}^+$ is finite,
its source is reduced and irreducible, and it is an isomorphism over
$U_{r,Z}$.  Moreover,
\begin{equation}
 K_{W_r}^+
 =\overline{K_{W_r}^+\times_{W_{r,Z}}U_{r,Z}}^{\,K_{W_r}^+}_{\red},
 \qquad
 (\omega_{W_r}^+)_*[K_{W_r}^+]=[Z_{W_r}^{\main}]
 \quad\text{in }A_*(W_{r,Z})_{\bbQ}.
 \label{eq:positive-main-cycle}
\end{equation}
\end{corollary}

\begin{proof}
Finiteness follows from \Cref{prop:positive-space-proper} and the
marking-restoration square of \Cref{cor:zero-order-marking-base-change}.
By definition $K_{W_r}^+$ is an open of the full-marked $K_r^+$.
It is nonempty, since the distinguished smooth-source, nonzero-section
locus lies over $W_0$.  Thus \Cref{prop:positive-space-irreducible} makes
it reduced and irreducible.  Its intersection with $(K_r^+)^\circ$ is
exactly the fibre over $U_{r,Z}$; on it, the marking divisor determines
the entire line--section pair and the basic logarithmic structure.
This gives the asserted isomorphism and a nonempty dense open in
$K_{W_r}^+$.

Properness makes the image closed.  Density and continuity show that it
is exactly the closure of $U_{r,Z}$, whose reduced structure is
$Z_{W_r}^{\main}$.  The isomorphism over $U_{r,Z}$ identifies the
generic residual gerbes, so the generic stack-theoretic degree is one.
Proper pushforward of the fundamental cycle therefore gives
\eqref{eq:positive-main-cycle}.

Since $\ell>0$, the affine-stabilizer verification in
\Cref{prop:compatible-support} applies to these full-marked spaces and
their logarithmic enhancements.  The auxiliary unmarked spaces were used
only to prove finiteness; no Chow pushforward on them is needed.
\end{proof}

\begin{remark}[Intersections with the other orbifold components]
The image is the main component, but its boundary may intersect other
orbifold components.  Such intersections are retained throughout.  In
particular, a full fibre product over $W_{r,Z}$ does not assert that the
BNR map is surjective onto every component of $W_{r,Z}$.
\end{remark}

\begin{proof}[Proof of \Cref{thm:main-component}]
The compatible finite-type opens and the constant-coefficient identity are
constructed in \Cref{prop:compatible-support}.  Irreducibility and reducedness
of $K_r^+$ are proved in \Cref{prop:positive-space-irreducible}; properness
over the chosen open, including after reintroducing the labelled markings in $Z$, is
\Cref{prop:positive-space-proper,cor:zero-order-marking-base-change}.  Finally,
\Cref{cor:positive-main-cycle} proves that the resulting morphism is an
isomorphism over Crumplin's distinguished locus of the trivial type and
pushes the fundamental cycle of $K_{W_r}^+$ to the reduced restriction of
$Z_{\main,r}$.  These statements give the
first equality in \eqref{eq:intro-main-component}, while
\eqref{eq:crumplin-constant} gives the second.
\end{proof}

On the negative branch define $W_r^-$ and $K_{W_r}^-$ by the
$2$-Cartesian squares below.  The notation $\Theta_r$ also denotes the
restriction to $W_r^-$, and $\Theta_r^!$ denotes the corresponding
base-changed refined Gysin operation
\[
 \Theta_r^!:A_k(W_{r,Z})_\bbQ\longrightarrow A_{k-m}(W_r^-)_\bbQ.
\]
These conventions apply at every root order under consideration.
\begin{equation}
\begin{tikzcd}
 W_r^-\ar[r]\ar[d,"\Theta_r"']
   & O_r^-\ar[d,"\Theta_r"]\\
 W_{r,Z}\ar[r,hook]&O_r^+,
\end{tikzcd}
\qquad
\begin{tikzcd}
 K_{W_r}^-\ar[r]\ar[d,"\omega_{W_r}^-"']
   & K_r^-\ar[d,"\omega_r^-"]\\
 W_r^-\ar[r,hook]&\overline O_r^-.
\end{tikzcd}
\label{eq:negative-chimera-open}
\end{equation}

Let
\begin{equation}
 \iota:\Dcal^m=(B\bbG_m)^m\hookrightarrow\Acal^m
\label{eq:product-zero-section}
\end{equation}
be the product zero section.  It is a representable regular embedding of
codimension $m$.  Order its factors by the inherited order on $P$.
Let
\[
 \mathrm{ev}_{P,r}:\overline O_r^+\longrightarrow\Acal^m
\]
be the classifying morphism of the ordered evaluation line--section pairs
$(E_{p,r},e_{p,r})_{p\in P}$ from \eqref{eq:evaluation-bundle}.
For each $p\in P$, let $h_{p,r}$ be the section of the base characteristic
sheaf $\overline M_{K_r^+}$ obtained by restricting the rooted target
characteristic to the marking $p$.  This restriction comes from the base
because the marking has contact order zero.  On a chart of rooted type
$\Theta$, with $v_p$ the vertex carrying $p$, it is the height covector
$\widetilde h_{v_p}$; its numerical value at a point of the cone is the
rooted vertex position $\widetilde x_{v_p}$.

\begin{proposition}[The BNR chimera as a zero-section locus]
\label{prop:chimera-zero-locus}
Under BNR's isomorphism $T_r^-\simeq T_r^+$ for the original and
positivised numerical data, the ordered puncturing-offset line--section
pairs $\operatorname{DF}_{M_{K_r^+}}(h_{p,r})$, $p\in P$, define a morphism
\begin{equation}
 \phi_r:K_r^+\longrightarrow\Acal^m.
\label{eq:offset-morphism}
\end{equation}
There is a canonical $2$-Cartesian square of classical algebraic stacks
\begin{equation}
\begin{tikzcd}
 K_r^-\ar[r,"\theta_r"]\ar[d]&K_r^+\ar[d,"\phi_r"]\\
 \Dcal^m\ar[r,"\iota"']&\Acal^m,
\end{tikzcd}
\label{eq:chimera-offset-square}
\end{equation}
where $K_r^-=\PunctOrb_{\Lambda_\Gamma}(\Acal_r\mid\Dcal_r)$ is BNR's
chimera for the original signed data.  The left vertical arrow classifies
the same offset lines with their sections set equal to zero, and $\theta_r$
is the induced morphism to the positive space.  There is a canonical
$2$-isomorphism
\[
 \phi_r\simeq\mathrm{ev}_{P,r}\circ\omega_r^+.
\]
For the following cycle identity, restrict all the stacks and morphisms
to any common finite face-closed support used below, and retain their
symbols for these restrictions.  The same convention applies to cycle
classes in the proof.  Its refined class is
\begin{equation}
 [K_r^-]^{\refc}=\theta_r^![K_r^+].
\label{eq:chimera-refined-class-global}
\end{equation}
Here $\theta_r^!$ is the refined pullback obtained by base change from the
regular embedding \eqref{eq:product-zero-section}.  The Cartesian square
is compatible with arbitrary base change, and the refined Gysin operation
is bivariant.  The displayed cycle identity is compatible with flat
pullback and with the compatible refined or virtual pullbacks used below;
in particular, it restricts to every finite face-closed cone substack under
consideration.
\end{proposition}

\begin{proof}
Write $\Mfrak^{\mathrm{tw},-}_{1,n+\rho}$ for the stack of genus-one
twisted prestable curves with all $n+\rho$ labelled markings and the
marking indices prescribed by the original signed datum at order $r$.
Write $\mathsf M^{\mathrm{trop,tw},-}$ for the corresponding sector of
twisted tropical curves.  These are the negative counterparts of the
curve factors in \eqref{eq:positive-space-fibre-product}.
BNR define the rooted puncturing substack as the classical fibre product
\[
 V(T_r^-)=\Acal(T_r^-)\times_{\Acal^m}\Dcal^m
\]
for the ordered puncturing-offset morphism, and define its refined class by
\begin{equation}
 [V(T_r^-)]^{\refc}=\iota^![\Acal(T_r^-)].
\label{eq:BNR-offset-refined-class}
\end{equation}
Their all-genus Cartesian presentation of the chimera then gives
\begin{equation}
 K_r^-\simeq
 \Mfrak^{\mathrm{tw},-}_{1,n+\rho}
 \mathop{\times}^{(2)}_{\Acal(\mathsf M^{\mathrm{trop,tw},-})}
 V(T_r^-)
\label{eq:BNR-chimera-presentation}
\end{equation}
\cite[Section~3.3 and Lemma~3.18]{BNR}.  The corresponding positive
presentation is \eqref{eq:positive-space-fibre-product}.

The two twisted-curve factors are not literally identical, because the
positivised datum replaces the source index at a former puncture by one.
Work over a test scheme $S$.  For $p\in P$, put $s_p=r/d_p$.
Let $\cC^-\to S$ be the negative twisted source, and let
$q:\cC^-\to\cC^+$ coarsen only the marking roots indexed by $P$.
Write $(L,s_-)$ for the negative root line--section pair and $(M,s_+)$
for the pair obtained on $\cC^+$ by the marking twist and descent.
Let $\widetilde p\subset\cC^-$ be the marking root divisor,
$\vartheta_p$ the tautological section of $\cO_{\cC^-}(\widetilde p)$,
and $x_p$ the canonical section of $\cO_{\cC^+}(p)$; thus
$\vartheta_p^{s_p}=q^*x_p$.  Finally, let $(A,a)$ be the negative
$r$th-power line--section pair descended to $\cC^+$, with the chosen
identification $(L,s_-)^{\otimes r}\simeq q^*(A,a)$.
The negative root line has character $-1\pmod{s_p}$ at the marking
gerbe.  Apply the case $k_p=1$ of
\Cref{lem:marking-picard-equivalence}.  It gives an equivalence, on objects,
arrows, and automorphism group schemes,
\begin{equation}
 L=q^*M\otimes\cO_{\cC^-}
       \left(-\sum_{p\in P}\widetilde p\right),
 \qquad
 q^*s_+=s_-\prod_{p\in P}\vartheta_p,
\label{eq:chimera-coarsen-twist}
\end{equation}
These identifications commute with base change.  Since $r/s_p=d_p$,
the chosen power equation
is carried to
\begin{equation}
 s_+^{\otimes r}=a\prod_{p\in P}x_p^{d_p}
 \quad\text{under }M^{\otimes r}\simeq
 A\Bigl(\sum_{p\in P}d_pp\Bigr),
\label{eq:chimera-positive-power}
\end{equation}
which is exactly the power equation associated with
$\Lambda_\Gamma^+$.

BNR's isomorphism $T_r^-\simeq T_r^+$ changes the vertex degree by the
contributions of the adjacent former punctures and leaves the cone,
continuity equations, circuit equations, saturation, and specialization
maps unchanged \cite[Section~4.3.1, Lemma~4.5]{BNR}.  Its proof is
local on the graph and does not use that the dual graph is a tree.  Together
with \Cref{lem:DF-reconstruction} and
\eqref{eq:chimera-coarsen-twist}, this identifies the complete logarithmic
and line--section data on the two sides, including units, line gluing, arrows,
and automorphisms.

It remains to identify the zero equations.  In terms of the characteristic
section $h_{p,r}$ defined above, the BNR gerby puncturing-offset pair on
$K_r^+$ is
\begin{equation}
 (F_{p,r},u_{p,r})
 :=\operatorname{DF}_{M_{K_r^+}}(h_{p,r}).
\label{eq:offset-pair}
\end{equation}
The positive target characteristic at the former puncture is exactly
$h_{p,r}$, because its marking contact is zero.  Deligne--Faltings
reconstruction and restriction to the marking therefore give a canonical
isomorphism of line--section pairs
\begin{equation}
 (\omega_r^+)^*(E_{p,r},e_{p,r})
 \simeq(F_{p,r},u_{p,r}),
\label{eq:evaluation-offset-identification}
\end{equation}
compatible with arrows and base change.  Thus the scheme-theoretic zero
fibre of the ordered offset morphism is precisely the positive
line--section locus whose evaluation at every former puncture vanishes.
By \eqref{eq:section-descent-with-vanishing}, division by the tautological
marking-root sections reconstructs the negative data uniquely.  Hence this
zero fibre is the classical BNR chimera, proving
\eqref{eq:chimera-offset-square}.

Finally, smooth pullback of
\eqref{eq:BNR-offset-refined-class} through the Cartesian presentation is
the refined zero-section pullback on $K_r^+$.  This proves
\eqref{eq:chimera-refined-class-global}.  Nothing here uses BNR's
genus-zero isomorphism between the chimera and the whole ordinary orbifold
space, and no derived enhancement of the chimera is asserted.
\end{proof}

\begin{remark}[Connection with BNR's classical chimera]
The puncturing condition is BNR's classical fibre product with the zero
section in $\Acal^m$; its refined class is the ordinary refined Gysin
pullback identified above.
\end{remark}

The evaluation-offset identification
\eqref{eq:evaluation-offset-identification} identifies the inverse image
of $K_{W_r}^+$ under $\theta_r$ with the space $K_{W_r}^-$ defined in
\eqref{eq:negative-chimera-open}; the following lemma records the resulting
Cartesian square.  Restrict \eqref{eq:chimera-offset-square} accordingly.
Write
\[
 \theta_{W_r}:K_{W_r}^-\longrightarrow K_{W_r}^+
\]
for the resulting morphism and set
\begin{equation}
 [K_{W_r}^-]^{\refc}:=\theta_{W_r}^![K_{W_r}^+]
 \quad\text{in }A_*(K_{W_r}^-)_\bbQ.
\label{eq:chimera-refined-class}
\end{equation}
This is the restriction of BNR's refined class; no Chow class on the
unrestricted, generally non-quasi-compact chimera is being asserted.

\begin{lemma}[The evaluation and puncturing-offset zero-section squares]
\label{lem:chimera-zero-section-square}
The square of classical stacks
\begin{equation}
\begin{tikzcd}
 K_r^-\ar[r,"\theta_r"]\ar[d,"\omega_r^-"']&
 K_r^+\ar[d,"\omega_r^+"]\\
 \overline O_r^-\ar[r,"\overline\Theta_r"']&\overline O_r^+
\end{tikzcd}
\label{eq:chimera-zero-section-square}
\end{equation}
is $2$-Cartesian.  On the chosen finite-type open,
\begin{equation}
 (\omega_{W_r}^-)_*[K_{W_r}^-]^{\refc}
 =\Theta_r^![Z_{W_r}^{\main}]
 \quad\text{in }A_*(W_r^-)_\bbQ.
\label{eq:negative-main-cycle}
\end{equation}
\end{lemma}

\begin{proof}
Proposition~\ref{prop:negative-zero-locus} identifies the bottom row as the
base change of the product zero section by the ordered evaluation
line--section pairs.  Proposition~\ref{prop:chimera-zero-locus} identifies
the top row as the base change of the same zero section by the ordered
puncturing-offset pairs.  The isomorphisms
\eqref{eq:evaluation-offset-identification} give the $2$-isomorphism
$\phi_r\simeq\mathrm{ev}_{P,r}\circ\omega_r^+$, so the displayed square
is $2$-Cartesian.

On the selected open, $\omega_{W_r}^+$ is finite, hence projective, by
\Cref{prop:positive-space-proper}; its base change
$\omega_{W_r}^-$ is also finite.  Compatibility of refined Gysin pullback
with projective pushforward gives
\[
\begin{aligned}
 (\omega_{W_r}^-)_*[K_{W_r}^-]^{\refc}
 &= (\omega_{W_r}^-)_*\theta_{W_r}^![K_{W_r}^+]\\
 &= \Theta_r^!(\omega_{W_r}^+)_*[K_{W_r}^+]\\
 &= \Theta_r^![Z_{W_r}^{\main}],
\end{aligned}
\]
where the last equality is \eqref{eq:positive-main-cycle}.
\end{proof}

For $R=\lambda r$, define $W_R^-$ by the $2$-Cartesian square
\begin{equation}
\begin{tikzcd}
 W_R^-\ar[r]\ar[d,"\pi_{R,r}^-"']
   & \overline O_R^-\ar[d,"\overline\pi_{R,r}^-"]\\
 W_r^-\ar[r,hook]
   & \overline O_r^- .
\end{tikzcd}
\label{eq:full-negative-open}
\end{equation}
Compatibility of root-order comparison with $\Theta_R$ and $\Theta_r$
also gives a canonical equivalence
\[
 W_R^-\simeq O_R^-\mathop{\times}_{O_R^+}W_{R,Z},
\]
so the same space is the negative zero-locus base change of $W_{R,Z}$.
For $N\in\{r,R\}$, the notation $W_N^-$ denotes the corresponding
space just defined, and $\Theta_N:W_N^-\to W_{N,Z}$ its zero-locus
morphism.  Combining
\Cref{lem:power-gysin,prop:compatible-support,lem:chimera-zero-section-square} gives the
universal identity needed later.

\begin{proposition}[Universal genus-one identity]
\label{prop:universal-identity}
For fixed $r$, the class-valued function
\[
 \lambda\longmapsto
 (\lambda r)^m(\pi_{\lambda r,r}^-)_*[W_{\lambda r}^-]^{\vir}
\]
agrees, for all sufficiently large positive integers $\lambda$, with a
polynomial in $\lambda$ with coefficients in $A_*(W_r^-)_\bbQ$.
Its constant coefficient satisfies
\begin{equation}
 \CT_\lambda\!\left(
 (\lambda r)^m(\pi_{\lambda r,r}^-)_*
 [W_{\lambda r}^-]^{\vir}\right)
 =r^m(\omega_{W_r}^-)_*[K_{W_r}^-]^{\refc}
\label{eq:universal-identity}
\end{equation}
in $A_*(W_r^-)_\bbQ$.
\end{proposition}

\begin{proof}
Put $R=\lambda r$.  Proposition \ref{prop:negative-zero-locus}, restricted to
the finite-type open, gives
\begin{equation}
 [W_R^-]^{\vir}=\Theta_R^![W_{R,Z}]^{\vir}.
 \label{eq:revised-display-19}
\end{equation}
The positive comparison is proper by
Lemma~\ref{lem:proper-root-comparison}.  The negative comparison factors
through the closed nilpotent immersion into the base change of that
proper map in \eqref{eq:genuine-power-cartesian-square}, and is therefore
proper as well.  Applying the power-section formula
\eqref{eq:scaled-root-change} to \eqref{eq:revised-display-19} yields,
for each positive integer $\lambda$,
\begin{equation}
 R^m(\pi_{R,r}^-)_*[W_R^-]^{\vir}
 =r^m\Theta_r^!(\pi_{R,r}^+)_*[W_{R,Z}]^{\vir}.
 \label{eq:revised-display-20}
\end{equation}

The right-hand side is eventually polynomial in $\lambda$ by
\Cref{prop:compatible-support}, proving the polynomiality assertion.
The refined Gysin map $\Theta_r^!$ is a fixed homomorphism of rational Chow
groups; in particular it commutes with finite sums and with extraction of
the coefficient of $\lambda^0$.  Taking constant coefficients in
\eqref{eq:revised-display-20}, using \eqref{eq:crumplin-constant}, and then using
\eqref{eq:negative-main-cycle}, gives the complete chain
\[
 \begin{aligned}
 &\CT_\lambda\!\left(
   (\lambda r)^m(\pi_{\lambda r,r}^-)_*
   [W_{\lambda r}^-]^{\vir}\right)\\
 &\quad =r^m\Theta_r^!\CT_\lambda\!\left(
   (\pi_{\lambda r,r}^+)_*[W_{\lambda r,Z}]^{\vir}
   \right)\\
 &\quad =r^m\Theta_r^![Z_{W_r}^{\main}]\\
 &\quad =r^m(\omega_{W_r}^-)_*[K_{W_r}^-]^{\refc}.
 \end{aligned}
\]
This is precisely \eqref{eq:universal-identity} in
$A_*(W_r^-)_\mathbb Q$.
\end{proof}

\begin{remark}[The bridge from the main component to negative contacts]
This identity is the point at which the two independent ingredients of the
paper meet.  Crumplin's genus-one component calculation supplies the constant
term on the positive universal space, while the BNR zero-section description
and Lemma~\ref{lem:power-gysin} transfer it to the original signed data.  The
remaining steps of the paper transport this universal identity to the
geometric pair $(X,D)$ and then apply BNR's root-forgetting normalization.
\end{remark}

\section{Passage from the universal target to \texorpdfstring{$(X,D)$}{(X,D)}}
\label{sec:geometric}

The preceding calculation takes place over the universal line--section target.
This section proves that the relevant fibre products and obstruction theories
are compatible with the classifying map of $(X,D)$, obtains the geometric
identity, and then applies BNR's all-genus root-forgetting pushforward to prove the comparison in
\Cref{thm:main}.

\subsection{Fibre products and comparison between root orders}

Crumplin's Lemma~4.7 gives the analogous universal/geometric Cartesian square
with compatible obstruction theories for the orbifold spaces associated with
the corresponding nonnegative contact data.  The
present section proves the version needed here because our spaces are the
classical zero-section loci inside the spaces associated with the positivised
numerical data and are restricted by the finite-type $2$-base-change squares defined above.

Let
\[
 \mathcal U_X:=\Mbar_{1,n+\rho}(X,\beta)
\]
be the Kontsevich stack fixed in \Cref{sec:setup}.  Let $\mathcal U_A$ be the
Artin stack whose objects are connected, genus-one, $(n+\rho)$-pointed
prestable curves $(C,\mathbf p)$ together with a line--section pair $(L,s)$
of degree $D\cdot\beta$; no stability condition is imposed on this universal
target stack.  The classifying morphism is
\begin{equation}
 u:\mathcal U_X\longrightarrow\mathcal U_A,
 \qquad
 (C,\mathbf p,f)\longmapsto
 (C,\mathbf p;f^*\cO_X(D),f^*s_D).
\label{eq:classifying-map}
\end{equation}
The universal punctured stack and the universal high-age stacks map to
$\mathcal U_A$ by forgetting their logarithmic or root structure and
retaining the coarse curve and descended line--section pair.

For every root order $N$ satisfying the standing assumptions (in particular,
$N=r,R$), write
\[
 O_{X,N}^-:=\Mbar_\Gamma(X_{D,N}).
\]
The coarse line bundle associated with the positivised numerical data,
\[
 f^*\cO_X(D)\Bigl(\sum_{p\in P}d_pp\Bigr),
\]
is independent of $N$.  Because $W_0$ was defined from the geometric data and contains the
complete geometric image, the universal morphism
from $O_{X,N}^-$ factors scheme-theoretically through the corresponding open
$W_N^-$.  Denote this factorisation by
\[
 \psi_N:O_{X,N}^-\longrightarrow W_N^-.
\]
Let $K_{X,r}^-$ be the base change over the finite-type open defined above of the BNR chimera stack along
the classifying map $\mathcal U_X\to\mathcal U_A$.  Because the complete
geometric image lies in $W_r^-$, this base change over that finite-type open is the full
geometric BNR chimera stack over the chosen finite-type support.  Let
$\omega_{X,r}^-:K_{X,r}^-\to O_{X,r}^-$ forget its logarithmic enhancement,
and write $\psi_{K,r}:K_{X,r}^-\to K_{W_r}^-$ for its universal projection.
Finally, write
\[
 \psi_P:\Punct_\Gamma(X\mid D)\longrightarrow
 \Punct_{\Lambda_\Gamma}(\Acal\mid\Dcal)
\]
for the universal punctured projection.

\begin{lemma}[Fibre-product descriptions]
\label{lem:geometric-fibre-products}
There are canonical equivalences
\begin{equation}
 \Punct_\Gamma(X\mid D)
 \simeq\mathcal U_X\times_{\mathcal U_A}
              \Punct_{\Lambda_\Gamma}(\Acal\mid\Dcal),
 \qquad
 O_{X,N}^-\simeq\mathcal U_X\times_{\mathcal U_A}W_N^-
 \quad(N=r,R).
\label{eq:geometric-universal-products}
\end{equation}
Moreover,
\begin{equation}
 K_{X,r}^-
 \simeq\mathcal U_X\times_{\mathcal U_A}K_{W_r}^-
 \simeq O_{X,r}^-\times_{W_r^-}K_{W_r}^-.
\label{eq:geometric-chimera-product}
\end{equation}
In particular, the following square is $2$-Cartesian:
\begin{equation}
\begin{tikzcd}
 K_{X,r}^-\ar[r,"\psi_{K,r}"]\ar[d,"\omega_{X,r}^-"']
   & K_{W_r}^-\ar[d,"\omega_{W_r}^-"]\\
 O_{X,r}^-\ar[r,"\psi_r"']&W_r^-.
\end{tikzcd}
\label{eq:geometric-chimera-cartesian-diagram}
\end{equation}
These are equivalences of stacks, including automorphism group schemes, and
commute with arbitrary base change.
\end{lemma}

\begin{proof}
We prove the assertions as equivalences of fibre groupoids.  Let $S$ be a
scheme.  An object of
\begin{equation}
 \mathcal U_X\times_{\mathcal U_A}
 \Punct_{\Lambda_\Gamma}(\Acal\mid\Dcal)
 \label{eq:revised-display-21}
\end{equation}
consists of a stable map $f:C\to X$, a universal basic punctured map
$a:(C,M_C^\circ)\to(\Acal,\Dcal)$ with the prescribed signed contacts,
where $M_C^\circ$ is the punctured logarithmic structure.  Writing
$a_{\mathrm{und}}$ for its underlying line--section pair, the remaining
datum is a $2$-isomorphism
\begin{equation}
 a_{\mathrm{und}}
 \simeq\bigl(f^*\mathcal O_X(D),f^*s_D\bigr)
 \label{eq:revised-display-22}
\end{equation}
of line--section pairs on the same pointed curve.  The classifying map of
the smooth divisor is strict, and locally along $D$ is given by the single
chart $1\mapsto$ a local equation of $D$.  Hence \eqref{eq:revised-display-22} identifies the
map of logarithmic structures supplied by $a$ with a unique logarithmic
enhancement of $f$.  This is exactly a basic punctured stable map to
$(X,D)$.  Conversely, composing a basic punctured map to $(X,D)$ with the
classifying map $(X,D)\to(\Acal,\Dcal)$ supplies $a$ and the canonical
isomorphism \eqref{eq:revised-display-22}.  Basicness is unchanged because the classifying map is
strict.  The two constructions are inverse on objects and on arrows;
stability is the stability of the retained map $f$.  This proves the first
equivalence in \eqref{eq:geometric-universal-products}.

For the orbifold assertion use the Cartesian target square
\begin{equation}
\begin{tikzcd}
 X_{D,N}\ar[r]\ar[d]&\Acal_N\ar[d]\\
 X\ar[r]&\Acal .
\end{tikzcd}
 \label{eq:revised-display-23}
\end{equation}
An object of
$\mathcal U_X\times_{\mathcal U_A}W_N^-$ is a stable map
$f:C\to X$, a representable twisted universal map
$a_N:\mathcal C\to\Acal_N$, and an identification of their descended
line--section pairs.  Pull $f$ back to $\mathcal C$ and apply the universal
property of \eqref{eq:revised-display-23}.  This gives a unique map
$f_N:\mathcal C\to X_{D,N}$.  It is representable because the stabilizer
acts faithfully on the $\Acal_N$ factor.  Its coarse map is $f$, so its
stability is equivalent to that of $f$.  Conversely, a representable map
to $X_{D,N}$ gives these data by the two projections.  This proves
\[
 \Mbar_\Gamma(X_{D,N})
 \simeq\mathcal U_X\times_{\mathcal U_A}W_N^-.
\]
The factor through $W_N^-$ is not a restriction of the geometric moduli
problem: by construction of $W_0$, every descended pair arising from a
stable map of class $\beta$ factors through the chosen open.  The orbifold equivalence in
\eqref{eq:geometric-universal-products} therefore describes the complete
geometric moduli problem.

Finally, the geometric BNR chimera stack is defined by base change from the
universal BNR chimera stack.  The entire geometric image lies in
$W_r^-$, giving the first equivalence in
\eqref{eq:geometric-chimera-product}.  The just-proved equivalence
$O_{X,r}^-\simeq\mathcal U_X\times_{\mathcal U_A}W_r^-$ and associativity
of $2$-fibre products give
\[
 \begin{aligned}
 \mathcal U_X\times_{\mathcal U_A}K_{W_r}^-
 &\simeq
 (\mathcal U_X\times_{\mathcal U_A}W_r^-)
   \times_{W_r^-}K_{W_r}^-\\
 &\simeq O_{X,r}^-\times_{W_r^-}K_{W_r}^-.
 \end{aligned}
\]
All constructions use universal properties of fibre products and therefore
act functorially on arrows, identify automorphism group schemes, and commute
with arbitrary base change.
\end{proof}

\begin{lemma}[Comparison of geometric root-stack spaces]
\label{lem:geometric-root-change}
For $R=\lambda r$, tensor power and canonical partial source coarsening give
the $2$-Cartesian square
\begin{equation}
\begin{tikzcd}
 O_{X,R}^-\ar[r,"\psi_R"]\ar[d,"\Pi_{R,r}^-"']&
 W_R^-\ar[d,"\pi_{R,r}^-"]\\
 O_{X,r}^-\ar[r,"\psi_r"']&W_r^-.
\end{tikzcd}
\label{eq:geometric-root-square}
\end{equation}
\end{lemma}

\begin{proof}
The comparison $W_R^-\to W_r^-$ retains the same coarse curve and
descended line--section pair, so it is a morphism over $\mathcal U_A$.
The fibre-product descriptions in \Cref{lem:geometric-fibre-products}
therefore give
\[
\begin{aligned}
 O_{X,r}^-\times_{W_r^-}W_R^-
 &\simeq(\mathcal U_X\times_{\mathcal U_A}W_r^-)
                   \times_{W_r^-}W_R^-\\
 &\simeq\mathcal U_X\times_{\mathcal U_A}W_R^-
 \simeq O_{X,R}^-.
\end{aligned}
\]
These are equivalences of fibre categories and include arrows and
automorphisms.  The projection to $O_{X,r}^-$ is the claimed geometric
comparison: the target identity
\begin{equation}
 X_{D,R}\simeq X_{D,r}\times_{\Acal_r}\Acal_R
 \label{eq:target-root-product}
\end{equation}
identifies the pullback of the universal $r$th-root line with the
$\lambda$th tensor power of the universal $R$th-root line, and the source
is partially coarsened by the same character kernel as in the
universal map.  Hence the equivalence identifies the displayed square
with \eqref{eq:geometric-root-square}.  No new logarithmic-extension
argument is needed, and obstruction-theory compatibility is proved
separately in \Cref{lem:root-change-POT}.
\end{proof}

\subsection{Relative obstruction theories}

For $N=r,R$, let
\[
 \pi_N:\mathcal C_N\longrightarrow O_{X,N}^-,
 \qquad f_N:\mathcal C_N\longrightarrow X_{D,N}
\]
be the universal twisted curve and universal map.  Set
$\bar f_N:=p_N\circ f_N$, where $p_N:X_{D,N}\to X$ is the root-stack
projection.  On the geometric chimera, write
$\pi_{K,r}:\mathcal C_{K,r}\to K_{X,r}^-$ and
$f_{K,r}:\mathcal C_{K,r}\to X_{D,r}$ for the universal curve and map,
and put $\bar f_{K,r}:=p_r\circ f_{K,r}$.  On
$\Punct_\Gamma(X\mid D)$, write $\pi_P:\mathcal C_P\to
\Punct_\Gamma(X\mid D)$ for the universal curve and
$f_P:\mathcal C_P\to X$ for the underlying universal map.

\begin{lemma}[Compatibility of the relative obstruction theories]
\label{lem:root-change-POT}
For $N=r,R$, the morphism $\psi_N$ carries the relative perfect obstruction
theory
\begin{equation}
 \mathbb E_N=
 \left(\mathbf R\pi_{N*}\bar f_N^*T_X(-\log D)\right)^\vee,
 \qquad
 \mathbb E_N\longrightarrow\mathbb L_{\psi_N},
\label{eq:relative-POT}
\end{equation}
and these theories form a compatible morphism across the square comparing the two root orders.
The analogous relative complexes on $\psi_{K,r}$ and $\psi_P$ are
\[
 \mathbb E_{K,r}
 :=\bigl(\mathbf R\pi_{K,r*}\bar f_{K,r}^*T_X(-\log D)\bigr)^\vee,
 \qquad
 \mathbb E_P
 :=\bigl(\mathbf R\pi_{P*}f_P^*T_X(-\log D)\bigr)^\vee.
\]
Their arrows to $\mathbb L_{\psi_{K,r}}$ and $\mathbb L_{\psi_P}$ come
from the same universal-target construction.  Together with the compatible
absolute and refined theories in
\Cref{prop:universal-geometric-pullbacks}, these arrows form the compatible
triples used below.
\end{lemma}

\begin{proof}

The target square defining $X_{D,N}$ identifies the relative tangent bundle with the pullback of the
logarithmic tangent bundle of the pair:
\begin{equation}
 T_{X_{D,N}/\Acal_N}
 \simeq p_N^*T_X(-\log D).
 \label{eq:revised-display-25}
\end{equation}
Infinitesimal deformations of $f_N$ with the universal map to $\Acal_N$
fixed are therefore governed by
$R\pi_{N*}\bar f_N^*T_X(-\log D)$.  Since $\pi_N$ is a proper nodal curve and
the sheaf is locally free, this complex is perfect of amplitude $[0,1]$;
its dual is perfect of amplitude $[-1,0]$.

The differential of the universal evaluation morphism constructs a map
\begin{equation}
 \mathbb E_N=
 \bigl(R\pi_{N*}\bar f_N^*T_X(-\log D)\bigr)^\vee
 \longrightarrow\mathbb L_{O_{X,N}^-/W_N^-}.
 \label{eq:revised-display-26}
\end{equation}
More explicitly, let $h:S\to O_{X,N}^-$ be a test morphism and let
$S\hookrightarrow S'$ be a square-zero extension with ideal $J$, equipped
with a compatible lift to $W_N^-$.  The obstruction to lifting the map to
$X_{D,N}$ with its projection to $\Acal_N$ fixed lies in the
hypercohomology group
$\mathbb H^1(S,\mathbf Lh^*\mathbb E_N^\vee\otimes^{\mathbf L}J)$.
If it vanishes, the torsor of lifts and the group of infinitesimal relative
automorphisms are given by the same hypercohomology in degrees $0$ and
$-1$, respectively.
Thus \eqref{eq:revised-display-26} is the standard relative perfect obstruction theory, not merely
an equality of $K$-classes.

Now put $R=\lambda r$ and form the pullback universal curve
\[
 \mathcal C_{r|R}
 :=\mathcal C_r\times_{O_{X,r}^-}O_{X,R}^-.
\]
Let $\bar f_{r|R}:\mathcal C_{r|R}\to X$ be the pullback of
$\bar f_r$.  The partial coarsening in
\Cref{lem:geometric-root-change} is a morphism over $O_{X,R}^-$,
\[
 q:\mathcal C_R\longrightarrow\mathcal C_{r|R}.
\]
The identity $\bar f_R=\bar f_{r|R}\circ q$ gives
\begin{equation}
 \bar f_R^*T_X(-\log D)\simeq q^*\bar f_{r|R}^*T_X(-\log D).
 \label{eq:tame-POT}
\end{equation}
The coarsening is tame, so invariants are exact and
$Rq_*\mathcal O_{\mathcal C_R}=\mathcal O_{\mathcal C_{r|R}}$.  Projection
formula and derived base change consequently give a canonical isomorphism
\begin{equation}
 (\Pi_{R,r}^-)^*R\pi_{r*}\bar f_r^*T_X(-\log D)
 \xrightarrow{\sim}
 R\pi_{R*}\bar f_R^*T_X(-\log D).
 \label{eq:revised-display-27}
\end{equation}
The inverse of its dual is the upper horizontal isomorphism in
\begin{equation}
\begin{tikzcd}
 (\Pi_{R,r}^-)^*\mathbb E_r\ar[r,"\sim"]\ar[d]&
 \mathbb E_R\ar[d]\\
 (\Pi_{R,r}^-)^*\mathbb L_{O_{X,r}^-/W_r^-}\ar[r]&
 \mathbb L_{O_{X,R}^-/W_R^-}.
\end{tikzcd}
\label{eq:POT-map-square}
\end{equation}
To verify commutativity, form the universal evaluation diagram on
$\mathcal C_R$.  Both routes in \eqref{eq:POT-map-square} are obtained by differentiating the
same composite
\[
 \mathcal C_R\xrightarrow{f_R}X_{D,R}
 \longrightarrow X_{D,r}
\]
relative to its composite with $\Acal_R\to\Acal_r$.  Naturality of the
cotangent complex and of the evaluation differential therefore identifies
the two maps.  Hence \eqref{eq:POT-map-square} is a morphism of perfect obstruction theories.
More explicitly, for every square-zero extension $S\hookrightarrow S'$ the two
routes send a deformation of the high-root map to the same obstruction to
lifting the composite map to $X_{D,r}$ with its $\Acal_r$-projection fixed.
The identifications of obstruction classes, torsors of lifts, and infinitesimal
automorphisms are induced by the same evaluation diagram on $\cC_R$.
Consequently the square commutes as a square of arrows to cotangent complexes,
not merely after passing to $K$-theory.

The rooted-chimera and punctured relative complexes are obtained from the
same target diagram by $2$-base change, so the same argument identifies their
relative arrows.  The absolute and refined class identities are established
in \Cref{prop:universal-geometric-pullbacks}; together the two results give
the compatible triples required for virtual base change.
\end{proof}

\begin{proposition}[Universal-to-geometric virtual pullbacks]
\label{prop:universal-geometric-pullbacks}
For $N=r,R$, on the fixed character sectors and the finite-type opens defined by the preceding $2$-base changes,
\begin{equation}
 \begin{aligned}
 [O_{X,N}^-]^{\vir}&=\psi_N^![W_N^-]^{\vir},\\
 [K_{X,r}^-]^{\refc}&=\psi_{K,r}^![K_{W_r}^-]^{\refc},\\
 [\Punct_\Gamma(X\mid D)]^{\refc}
 &=\psi_P^![\Punct_{\Lambda_\Gamma}(\Acal\mid\Dcal)]^{\refc}.
 \end{aligned}
\label{eq:geometric-class-pullbacks}
\end{equation}
These are identities of virtual or refined cycles, not merely isomorphisms
of the underlying complexes.
\end{proposition}

\begin{proof}
For the orbifold identity, use the common section-theory base
$\mathfrak P_N^{\mathrm{sec}}$, in its negative presentation
$\mathfrak P_N^-$: it retains the twisted curve, negative root line,
descended coarse line, and power isomorphism, but no section.  The equivalence
$\mathfrak P_N^-\simeq\mathfrak P_N^{\mathrm{sec}}$ fixes the map
$W_N^-\to\mathfrak P_N^{\mathrm{sec}}$.  The classifying sequence of targets is
\begin{equation}
 X_{D,N}\longrightarrow\Acal_N\longrightarrow B\bbG_m.
\label{eq:target-classifying-sequence}
\end{equation}
Here the last arrow forgets the section and retains the root line; the
notation $T_{X_{D,N}/B\bbG_m}$ refers to the composite in this sequence.
Let $\pi_N^W:\mathcal C_N^W\to W_N^-$ be the universal twisted curve
and let $\mathcal L_N^W$ be its negative root line.  Define the relative
section and map obstruction complexes, respectively, by
\[
 \mathbb F_N^{\Acal}
 :=\bigl(\mathbf R(\pi_N^W)_*\mathcal L_N^W\bigr)^\vee,
 \qquad
 \mathbb F_N^X
 :=\bigl(\mathbf R\pi_{N*}f_N^*T_{X_{D,N}/B\bbG_m}\bigr)^\vee.
\]
The first complex is the restriction of $\mathbb E_N^-$ from
\Cref{prop:negative-zero-locus} to $W_N^-$.  The second is on
$O_{X,N}^-$; their
obstruction-theory arrows are relative to $\mathfrak P_N^{\mathrm{sec}}$.
There is a relative tangent triangle whose first term is
\begin{equation}
 T_{X_{D,N}/\Acal_N}\simeq p_N^*T_X(-\log D).
\label{eq:relative-log-tangent}
\end{equation}
Pulling this triangle to the universal curve, applying derived pushforward,
and dualizing gives a distinguished triangle
\begin{equation}
 \psi_N^*\mathbb F_N^{\Acal}\longrightarrow
 \mathbb F_N^X\longrightarrow
 \mathbb E_N\longrightarrow\psi_N^*\mathbb F_N^{\Acal}[1],
 \qquad
 \mathbb E_N=(R\pi_{N*}\bar f_N^*T_X(-\log D))^\vee.
\label{eq:universal-geometric-POT-triangle}
\end{equation}
The evaluation differential maps this triangle to the cotangent triangle for
$O_{X,N}^-\to W_N^-\to\mathfrak P_N^{\mathrm{sec}}$.  This is the standard
universal-target construction for orbifold stable maps
\cite[Section~5.2]{ACW}.  It is independent of the genus and restricts to a
fixed character or age sector because such sectors are open and closed.
Manolache's composition theorem for compatible triples gives the first
identity.  Crumplin's equation~(37) records the same universal-to-geometric
formula in arbitrary genus for his standing contact data
\cite[Section~4.4]{Crumplin}; it is used here only as corroboration, not as
the source for the high-age sector.

For the second identity, BNR define the geometric chimera by a $2$-Cartesian
base change from the universal chimera and equip the vertical morphism with
relative obstruction theory
\[
 \bigl(R\pi_{K,r*}f_{K,r}^*T_{X_{D,r}}(-\log D_r)\bigr)^\vee
 \simeq
 \bigl(R\pi_{K,r*}\bar f_{K,r}^*T_X(-\log D)\bigr)^\vee
 =\mathbb E_{K,r}.
\]
They then define the geometric refined chimera class by the corresponding
virtual pullback \cite[Definition~3.14 and equation~(27)]{BNR}.  Restriction
to the finite-type open is open flat base change, giving the second identity.

For the third identity, BNR's refined punctured class is, by definition, the
virtual pullback of the universal refined punctured class through the ACGS
relative obstruction theory \cite[Definition~1.13]{BNR}.  The construction and functoriality of that relative theory are established
in \cite[Lemma~4.1 and Proposition~4.2]{ACGS}.  This gives the last identity.  No derived enhancement of the chimera is
required: applying the
fixed bivariant pullback $\psi_{K,r}^!$ to
\eqref{eq:chimera-refined-class-global} already gives the geometric chimera
class by BNR's definition.
\end{proof}

\begin{remark}[Relation with the universal-target method]
The genus-one component analysis takes place universally.  The geometry
of $(X,D)$ enters through the relative virtual pullbacks and compatible
obstruction theories established above.
\end{remark}

\begin{lemma}[Proper pushforward and virtual base change]
\label{lem:proper-virtual-base-change}
Assume $n+\rho>0$.  On the finite-type opens used in
\eqref{eq:negative-main-cycle} and \eqref{eq:geometric-class-pullbacks},
rational Chow pushforward exists for every proper vertical morphism in the
comparison diagrams and commutes with the refined or virtual pullbacks used
there.  In particular,
\begin{equation}
 \psi_r^!(\omega_{W_r}^-)_*
   =(\omega_{X,r}^-)_*\psi_{K,r}^!,
 \qquad
 \psi_r^!(\pi_{R,r}^-)_*
   =(\Pi_{R,r}^-)_*\psi_R^!.
\label{eq:proper-virtual-base-change}
\end{equation}
The same statement applies to the BNR root-forgetting square in
\Cref{prop:BNR-root-forgetting}.
\end{lemma}

\begin{proof}
\proofstep{Step 1: the precise proper/virtual base-change statement}
Consider a Cartesian square
\begin{equation}
\begin{tikzcd}
 F'\ar[r,"q"]\ar[d,"f'"']&F\ar[d,"f"]\\
 G'\ar[r,"p"']&G.
\end{tikzcd}
 \label{eq:abstract-virtual-base-change}
\end{equation}
Assume that the stacks are of finite type with affine geometric
stabilizers, that $p$ is proper of relative Deligne--Mumford type, and
that $f$ is of Deligne--Mumford type with a perfect relative obstruction
theory $E\to\mathbb L_f$.  Equip the pullback with the induced theory
$q^*E\to\mathbb L_{f'}$.  Then, with rational coefficients,
\begin{equation}
 q_*f'^!=f^!p_*.
 \label{eq:abstract-virtual-push-pull}
\end{equation}
Here an independently specified theory on $F'$ must be identified with
$q^*E$ as an arrow to the cotangent complex, not only in $K$-theory.
This is the proper-pushforward extension of Manolache's virtual
base-change theorem proved in \cite[Appendix~B, Theorem~B.0.1]{HHS}.
The affine-stabilizer hypothesis supplies the required quotient
stratifications by \cite[Proposition~3.5.9]{Kresch}; the proper rational
pushforwards are those of \cite[Theorem~B.17]{BSS}.  The theorem applies
to the induced obstruction theory, so no flatness of $p$ is required.

\proofstep{Step 2: verify the hypotheses for the comparison diagrams}
All selected universal stacks have affine stabilizers.  Since $n+\rho>0$,
at least one labelled marking is retained; a genus-one component is
marked or attached at a node.  Hence no elliptic translation group
occurs.  Automorphism groups are built from the affine automorphism
groups of pointed rational and pointed elliptic components, line-bundle
tori, and finite root groups.  Geometric stable-map stacks have finite
stabilizers.  The opens are finite type by their constructions.

The maps $\omega_{W_r}^{\pm}$ are finite representable; the root-order
maps are proper of relative Deligne--Mumford type by
\Cref{lem:proper-root-comparison}; and the BNR root-forgetting maps are
proper of relative Deligne--Mumford type by their finite-support root
construction.  The relevant squares are Cartesian by
\Cref{lem:geometric-fibre-products,lem:geometric-root-change} and BNR's
definition of the chimera.  The theories agree as arrows by
\Cref{lem:root-change-POT,prop:universal-geometric-pullbacks}.
Apply \eqref{eq:abstract-virtual-push-pull} to obtain
\eqref{eq:proper-virtual-base-change}.  Ordinary refined Gysin
commutation is the corresponding assertion of
\cite[Proposition~B.18]{BSS}.

\proofstep{Step 3: the smooth universal pullback}
We also use the following version of proper/flat base change.  In a
Cartesian square
\[
\begin{tikzcd}
 Y'\ar[r,"h'"]\ar[d,"p'"']&Y\ar[d,"p"]\\
 S'\ar[r,"h"']&S,
\end{tikzcd}
\]
suppose that $p$ is proper of relative Deligne--Mumford type, that
$h$ is flat of pure relative dimension, and that both $S$ and $S'$
admit stratifications by global quotient stacks.  Then, with rational
coefficients,
\[
 p'_*h'^*=h^*p_*.
\]
Here $h$ need not be representable.  To see this, use the restricted
Chow groups of Bae--Schmitt--Skowera.  The flat pullback and its
compatibility with the natural maps to ordinary Chow groups are the
properties (i) and (iii) preceding their Proposition~B.8; these properties
do not require $h$ to be representable.  Their Proposition~B.16 identifies
the restricted groups for $p$ and $p'$ with ordinary rational Chow groups,
using the assumed quotient stratifications of $S$ and $S'$, respectively.
On restricted groups, the required commutation follows from proper/flat
commutation for the naive cycles on the vector-bundle presentations in
the construction of proper pushforward.  Thus the proof of
\cite[Proposition~B.18]{BSS} applies with the explicit stratification
assumption on $S'$ replacing the representability hypothesis used there
to obtain that assumption.  This observation applies, in particular, to
the smooth morphisms to puncturing substacks below.
\end{proof}

\begin{remark}[Chow theory]
The retained marking excludes elliptic translation stabilizers.  The
preceding proof verifies both proper/virtual commutation and the smooth
universal pullback without assuming that the latter is representable.
\end{remark}

\subsection{The comparison after base change to \texorpdfstring{$(X,D)$}{(X,D)}}

\begin{proposition}[Comparison after base change to $(X,D)$]
\label{prop:geometric-identity}
Assume $m>0$.  Then
\begin{equation}
 r^m\omega_{X,r*}^-[K_{X,r}^-]^{\refc}
 =\CT_\lambda\!\left(
 (\lambda r)^m(\Pi_{\lambda r,r}^-)_*
 [O_{X,\lambda r}^-]^{\vir}\right).
\label{eq:geometric-identity}
\end{equation}
in $A_*(O_{X,r}^-)_\bbQ$.
\end{proposition}

\begin{proof}
The Cartesian squares of \Cref{lem:geometric-fibre-products,lem:geometric-root-change},
the compatible relative obstruction theories of \Cref{lem:root-change-POT},
and \Cref{lem:proper-virtual-base-change} give
\begin{equation}
 \psi_r^!(\omega_{W_r}^-)_*[K_{W_r}^-]^{\refc}
   =(\omega_{X,r}^-)_*[K_{X,r}^-]^{\refc},
 \label{eq:revised-display-28}
\end{equation}
\begin{equation}
 \psi_r^!(\pi_{R,r}^-)_*[W_R^-]^{\vir}
   =(\Pi_{R,r}^-)_*[O_{X,R}^-]^{\vir}.
 \label{eq:revised-display-29}
\end{equation}
Here the class identifications are those of
\Cref{prop:universal-geometric-pullbacks}.  Apply the fixed linear operator
$\psi_r^!$ to \eqref{eq:universal-identity} and use the two displayed
identities.  Linearity allows the constant coefficient to be taken before
or after this operation and gives \eqref{eq:geometric-identity}.
This uses compatibility of the obstruction theories, not an assertion that
the geometric virtual classes are fundamental classes.
\end{proof}

\subsection{BNR root-forgetting pushforward and the final comparison}

For a signed cone stack or cone support $S$ used below, let
$\epsilon_S:\Acal(S)\to\Acal^m$ be its ordered puncturing-offset
morphism, with factors indexed by $P$.  Write
\[
 V(S):=\Acal(S)\times_{\Acal^m}\Dcal^m,
 \qquad
 [V(S)]^{\refc}:=\iota^![\Acal(S)],
\]
where the fibre product uses $\epsilon_S$ and the product zero section
$\iota$ from \eqref{eq:product-zero-section}; the refined pullback is taken
in this Cartesian square.  The signed cone stacks $T^-$ and $T_r^-$ retain
all labelled markings.

\begin{proposition}[BNR root-forgetting pushforward on the finite-type open]
\label{prop:BNR-root-forgetting}
Assume $n+\rho>0$.  There is a finite face-closed cone substack
$T^{-,\mathrm{geom}}\subset T^-$ whose Artin fan is open in
$\Acal(T^-)$ and such that the geometric punctured projection factors
scheme-theoretically through $V(T^{-,\mathrm{geom}})$.  Use the finite
face-closed convention of Section~5 and define the rooted support
$T_r^{-,\mathrm{geom}}$ by the $2$-Cartesian square
\begin{equation}
\begin{tikzcd}
 T_r^{-,\mathrm{geom}}\ar[r]\ar[d]
   & T_r^-\ar[d]\\
 T^{-,\mathrm{geom}}\ar[r,hook]
   & T^-.
\end{tikzcd}
\label{eq:rooted-geometric-cone-base-change}
\end{equation}
The right vertical arrow forgets the root structure.  The induced map
\[
 \nu_r^{\mathrm{geom}}:
 V(T_r^{-,\mathrm{geom}})
 \longrightarrow V(T^{-,\mathrm{geom}})
\]
is proper.  Its geometric base change is the root-forgetting map
\[
 \alpha_r:K_{X,r}^-\longrightarrow\Punct_\Gamma(X\mid D)
\]
and is proper.  Moreover,
\begin{equation}
 \alpha_{r*}[K_{X,r}^-]^{\refc}
 =r^{-m}[\Punct_\Gamma(X\mid D)]^{\refc}.
\label{eq:BNR-root-forgetting}
\end{equation}
This equality lies in
$A_*(\Punct_\Gamma(X\mid D))_\bbQ$.  In this proposition
$K_{X,r}^-$ and its refined class are understood in BNR's classical sense.
This assertion does not claim a global all-genus cone bijection outside the
bounded support and does not use BNR's genus-zero identification with an
orbifold mapping stack.
\end{proposition}

\begin{proof}
\proofstep{Step 1: use the already bounded support}
For a vertex $v$, let
$P(v):=\{p\in P:\text{the leg labelled }p\text{ is incident to }v\}$.
Positivising a punctured tropical type changes its vertex degree by
$\sum_{p\in P(v)}d_p$ and replaces its punctured legs by zero-contact
legs.  It changes neither bounded-edge slopes nor the vertex-position
and edge-length equations \cite[Lemma~4.5]{BNR}.  Forget all legs in $Z$
without contracting any component.  The actual positive
line--section pair of a geometric punctured map is the one used to
define $g_\Gamma^+$ and $W_0$ in \Cref{subsec:theorem-spaces}.
Its positive type after forgetting $Z$ therefore belongs to $T^+_\Sigma$ by
\Cref{lem:finite-face-closed-support}.

Restore the $Z$-labelled legs in every possible distribution
on the graphs in that support, and reverse the degree shift.  Include
all faces and graph-automorphic copies.  The resulting finite cone
substack may be taken as $T^{-,\mathrm{geom}}$.  Every geometric
punctured type factors through it.  Its Artin fan is open, so the
factorization holds for families, including nonreduced bases.  In
particular, every bounded-edge slope satisfies $|m_e|\leq4B_0<r$.
No new root bound is chosen after fixing $r$.

By BNR's rooting formulas, every cone of this support has exactly one
rooted counterpart, with $s_e=r/\gcd(r,m_e)$,
$\ell_e=s_e\widetilde\ell_e$, and $x_v=r\widetilde x_v$.
The map
\begin{equation}
 \kappa:\Acal(T_r^{-,\mathrm{geom}})
        \longrightarrow\Acal(T^{-,\mathrm{geom}})
 \label{eq:bounded-Artin-root-map}
\end{equation}
is the generalized root stack of
\cite[Construction~3.10 and Lemma~3.11]{BNR} on this support.
It is proper of relative Deligne--Mumford type, and is an isomorphism
over the common zero-face stratum.  Both Artin fans are integral, so
\begin{equation}
 \kappa_*[\Acal(T_r^{-,\mathrm{geom}})]
       =[\Acal(T^{-,\mathrm{geom}})].
 \label{eq:bounded-Artin-root-degree}
\end{equation}

\proofstep{Step 2: recover BNR's rooting factor with the actual zero loci}
Let $(F_p,u_p)$ and $(\widetilde F_p,\widetilde u_p)$ denote the coarse
and gerby puncturing-offset line--section pairs on these Artin fans.
Let $h_{v_p}$ and $\widetilde h_{v_p}$ be the characteristic height
covectors of the vertex carrying $p$ on the coarse and rooted charts.
The coordinate relation $x_{v_p}=r\widetilde x_{v_p}$ is induced by
$\kappa^*h_{v_p}=r\widetilde h_{v_p}$.  Applying the Deligne--Faltings
construction gives canonical identities
\[
 \kappa^*(F_p,u_p)
     \simeq(\widetilde F_p,\widetilde u_p)^{\otimes r}.
\]
Consequently the true fibre product of $V(T^{-,\mathrm{geom}})$
with the upper Artin fan is $Z((\widetilde u_p^r)_{p\in P})$, not
$Z((\widetilde u_p)_{p\in P})$.  The latter is
$V(T_r^{-,\mathrm{geom}})$ and has a canonical nilpotent closed
immersion into the former.  Apply the powered-section calculation of
\Cref{lem:power-gysin}, now to these offset pairs, together with
\eqref{eq:bounded-Artin-root-degree}.  It gives
\begin{equation}
 (\nu_r^{\mathrm{geom}})_*
       [V(T_r^{-,\mathrm{geom}})]^{\refc}
   =r^{-m}[V(T^{-,\mathrm{geom}})]^{\refc}.
 \label{eq:bounded-BNR-root-factor}
\end{equation}
The morphism is proper: its source is closed in the upper Artin fan,
which is proper over the lower Artin fan, and its target is closed in
the latter.  This also follows from the finite-gerbe/root-stack
factorization of \cite[Proposition~3.13]{BNR}.
The calculation is the same refined-intersection normalization as that
proposition, with the nilpotent thickening made explicit.  It does not
assert that the offset square, or every face of BNR's cube, is Cartesian.
No genus-zero orbifold/chimera isomorphism is used.

\proofstep{Step 3: universal and geometric pullbacks}
Put $V=V(T^{-,\mathrm{geom}})$ and
$V_r=V(T_r^{-,\mathrm{geom}})$, and write $\nu_r:V_r\to V$
for $\nu_r^{\mathrm{geom}}$.  Let $\mathcal P$ be the restriction of
$\Punct_{\Lambda_\Gamma}(\Acal\mid\Dcal)$ to $V$, and let
$\mathcal K$ be the corresponding universal chimera.  BNR's
Definition~3.14 supplies the Cartesian square
\[
\begin{tikzcd}[column sep=large]
 \mathcal K\ar[r,"\phi_r"]\ar[d,"a_r"']&
 V_r\ar[d,"\nu_r"]\\
 \mathcal P\ar[r,"\phi"']&V.
\end{tikzcd}
\]
The maps $\phi$ and $\phi_r$ are smooth of pure relative dimension
$n+\rho$: they arise from the prestable-curve presentation in
\cite[equation~(7) and Definition~3.14]{BNR}, whose relative
dimension in genus one is $3\cdot1-3+n+\rho$.  The universal refined
classes are defined by
\[
 [\mathcal P]^{\refc}=\phi^*[V]^{\refc},\qquad
 [\mathcal K]^{\refc}=\phi_r^*[V_r]^{\refc}.
\]
The selected tropical support bounds the source graphs, so these
universal stacks are of finite type.  Their stabilizers are affine by
the marked-source argument in
\Cref{lem:proper-virtual-base-change}; hence they admit the quotient
stratifications needed for proper/smooth base change.  Applying that
formula and \eqref{eq:bounded-BNR-root-factor} gives
\[
 \begin{aligned}
 (a_r)_*[\mathcal K]^{\refc}
 &=\phi^*(\nu_r)_*[V_r]^{\refc}\\
 &=r^{-m}\phi^*[V]^{\refc}
 =r^{-m}[\mathcal P]^{\refc}.
 \end{aligned}
\]
This is the proof of BNR's Theorem~4.1 restricted to the present
support.  Its inputs here are the root construction and refined
intersection calculation of BNR's Section~3, whose applicability to
this genus-one support was checked in Steps~1 and~2; the genus-zero
orbifold identification in BNR's Theorem~4.2 is not used.

Every geometric punctured map lies over $V$, by Step~1.  Applying BNR's
definition of the geometric chimera gives a second Cartesian square
\[
\begin{tikzcd}[column sep=large]
 K_{X,r}^-\ar[r,"\psi_{\mathcal K}"]\ar[d,"\alpha_r"']&
 \mathcal K\ar[d,"a_r"]\\
 \Punct_\Gamma(X\mid D)\ar[r,"\psi_{\mathcal P}"']&
 \mathcal P.
\end{tikzcd}
\]
Both vertical maps are proper of relative Deligne--Mumford type.  The
relative perfect obstruction theory for $\psi_{\mathcal P}$ is the
ACGS theory
\[
 \bigl(R\pi_{P*}f_P^*T_X(-\log D)\bigr)^\vee,
\]
and its pullback, including the arrow to the cotangent complex, is the
relative theory for $\psi_{\mathcal K}$ by BNR's
Definition~3.14 and equation~(27), or by the tame-coarsening argument of
\Cref{lem:root-change-POT}.  These morphisms are of
Deligne--Mumford type because their geometric source stacks are
Deligne--Mumford.  Their virtual pullbacks define the geometric refined
classes by \Cref{prop:universal-geometric-pullbacks}.  Thus
\Cref{lem:proper-virtual-base-change} gives
\[
 \begin{aligned}
 \alpha_{r*}[K_{X,r}^-]^{\refc}
 &=\psi_{\mathcal P}^!(a_r)_*[\mathcal K]^{\refc}\\
 &=r^{-m}\psi_{\mathcal P}^![\mathcal P]^{\refc}
 =r^{-m}[\Punct_\Gamma(X\mid D)]^{\refc}.
 \end{aligned}
\]
This proves \eqref{eq:BNR-root-forgetting} using separate smooth and
geometric virtual pullbacks.
\end{proof}

\begin{remark}[The BNR rooting factor]
The factor $r^{-m}$ is the same rooting factor appearing in BNR's genus-zero
comparison, but the underlying BNR root-forgetting statement is all-genus.
It should be distinguished from the factor produced by powered evaluation
sections in Lemma~\ref{lem:power-gysin}.  The final proof works because these
two normalizations are tracked separately and then cancel in the passage from
the universal signed identity to the Fan--Wu--You constant term.
\end{remark}

\begin{proof}[Proof of \Cref{thm:main}]
If $m=0$, the puncturing substack is the whole tropical Artin fan and the BNR
refined punctured class is the ordinary logarithmic virtual class.  The
logarithmic/relative comparison identifies its pushforward with the ordinary
relative cycle \cite{AMW}; Fan--Wu--You identify that cycle with the constant
coefficient of the root-stack theory when there are no negative contacts
\cite[after Definition~3.2 and Theorem~3.13]{FWY}.  Thus the theorem holds in this case.  Assume from
now on that $m>0$.

\proofstep{Step 1: compare the maps to the common target}
The maps to the common target satisfy
\begin{equation}
 \tau_r\circ\omega_{X,r}^-=\varrho\circ\alpha_r,
 \qquad
 \tau_r\circ\Pi_{R,r}^-=\tau_R.
\label{eq:final-compatibility}
\end{equation}
Both identities retain the same coarse stable map and the same ordered
markings, so they are identities of morphisms including evaluation at the
last $\rho$ markings.

\proofstep{Step 2: push the geometric identity and cancel the root-forgetting factor}
Push the identity of \Cref{prop:geometric-identity} forward by the proper map $\tau_r$.
Using the first identity in \eqref{eq:final-compatibility} and then
\Cref{prop:BNR-root-forgetting}, the left side becomes
\begin{equation}
\begin{aligned}
 \tau_{r*}\!\left(r^m\omega_{X,r*}^-
 [K_{X,r}^-]^{\refc}\right)
 &=r^m(\tau_r\circ\omega_{X,r}^-)_*
 [K_{X,r}^-]^{\refc}\\
 &=r^m\varrho_*\alpha_{r*}[K_{X,r}^-]^{\refc}\\
 &=r^m r^{-m}\varrho_*
 [\Punct_\Gamma(X\mid D)]^{\refc}\\
 &=\varrho_*[\Punct_\Gamma(X\mid D)]^{\refc}.
\end{aligned}
\label{eq:final-left-chain}
\end{equation}
Thus the outer factor $r^m$ cancels the factor in BNR's root-forgetting pushforward, with no
residual contact or stabilizer factor.

For $R=\lambda r$, the second identity in
\eqref{eq:final-compatibility} gives
\begin{equation}
\begin{aligned}
 &\tau_{r*}\CT_\lambda\!\left(
 (\lambda r)^m(\Pi_{\lambda r,r}^-)_*
 [O_{X,\lambda r}^-]^{\vir}\right)\\
 &\qquad=\CT_\lambda\!\left(
 (\lambda r)^m(\tau_r\circ\Pi_{\lambda r,r}^-)_*
 [O_{X,\lambda r}^-]^{\vir}\right)\\
 &\qquad=\CT_\lambda\!\left(
 (\lambda r)^m\tau_{\lambda r,*}
 [\Mbar_\Gamma(X_{D,\lambda r})]^{\vir}\right).
\end{aligned}
\label{eq:final-right-chain}
\end{equation}
The first equality uses additivity of proper pushforward, which permits
coefficientwise application to the eventual Chow-valued polynomial.
Combining \eqref{eq:geometric-identity},
\eqref{eq:final-left-chain}, and \eqref{eq:final-right-chain} proves
\begin{equation}
 \varrho_*[\Punct_\Gamma(X\mid D)]^{\refc}
 =\CT_\lambda\!\left(
 (\lambda r)^m\tau_{\lambda r,*}
 [\Mbar_\Gamma(X_{D,\lambda r})]^{\vir}\right).
\label{eq:punctured-root-constant}
\end{equation}

\proofstep{Step 3: replace the cofinal parameter by the root order}
By definition, the expression inside the constant coefficient in
\eqref{eq:punctured-root-constant} is $F_\Gamma(\lambda r)$.  Write its
eventual polynomial as $F_\Gamma(N)=\sum_j a_jN^j$.  Then
\begin{equation}
 F_\Gamma(\lambda r)=\sum_j a_jr^j\lambda^j,
 \qquad
 \CT_\lambda F_\Gamma(\lambda r)=a_0
 =\CT_NF_\Gamma(N).
\label{eq:final-cofinal-calculation}
\end{equation}
Thus the cofinal substitution preserves the constant coefficient; no
limiting argument is involved.

\proofstep{Step 4: apply the Fan--Wu--You constant-term theorem}
Fan--Wu--You \cite[Theorem~3.13]{FWY} identify
\begin{equation}
 \CT_NF_\Gamma(N)=\mathfrak c_\Gamma(X/D)
\label{eq:final-FWY}
\end{equation}
in the Chow group of the same evaluation-enhanced target $B_\Gamma$.
Equations \eqref{eq:punctured-root-constant}--\eqref{eq:final-FWY} therefore
give both equalities in \eqref{eq:main-comparison}.  By the calculation in \Cref{subsec:expected-dimension}, every term lies in
$A_{d_\Gamma}(B_\Gamma)_\mathbb Q$, where $d_\Gamma$ is the number in
\eqref{eq:vdim}.  This completes the proof.
\end{proof}

Department of Mathematics, Louisiana State University, 303 Lockett Hall, Baton Rouge, LA, United States of America 70803

\medskip

\textit{E-mail}: yuwang@lsu.edu

\end{document}